\documentclass[reqno,a4paper]{mcom-l}

\usepackage[utf8]{inputenc}
\usepackage{amsmath,amssymb,amsfonts,epsfig,mathtools,float,latexsym,color,fancyhdr,stmaryrd}
\usepackage[colorlinks=true,breaklinks=true,linkcolor=lightgreen,citecolor=lightblue]{hyperref}
\usepackage{siunitx}

\definecolor{lightblue}{rgb}{0.22,0.45,0.70}

\definecolor{lightgreen}{rgb}{0.22,0.55,0.20}

\newcommand{\bu}{\boldsymbol{u}}
\newcommand{\bv}{\boldsymbol{v}}
\newcommand{\bnabla}{\boldsymbol{\nabla}}
\newcommand{\bvartheta}{\boldsymbol{\vartheta}}

\newcommand\bH{\mathbf{H}}
\newcommand\ff{\boldsymbol{f}}
\newcommand{\bV}{\mathbf{V}}
\newcommand\bx{\boldsymbol{x}}
\newcommand{\bdiv}{\operatorname*{\mathbf{div}}}

\renewcommand{\div}{\mathrm{div}} 
\newcommand{\Set}[2]{\left\{\,#1: #2\,\right\}}

\makeatletter
\@tfor\next:=abcdefghijklmnopqrstuvwxyz\do{%
	\def\command@factory#1{%
		\expandafter\def\csname v#1\endcsname{\boldsymbol{#1}}
	}
	\expandafter\command@factory\next
}
\@tfor\next:=ABCDEFGHIJKLMNOPQRSTUVWXYZ\do{%
	\def\command@factory#1{%
		\expandafter\def\csname v#1\endcsname{\mathbf{#1}}
	}
	\expandafter\command@factory\next
}

\makeatother

\newcommand\Tend{t_{\mathrm{end}}}
\newcommand{\dt}{\ensuremath{\, \mathrm{d}t}}

\newcommand{\dz}{\ensuremath{\, \mathrm{d}z}}

\newcommand{\pdt}[1]{{\partial_t #1}}
\newcommand\Pol{\mathcal{P}}	
\newcommand{\Dt}{\Delta t}      
\newcommand{\cero}{\boldsymbol{0}}
\newcommand{\Sc}{\mathrm{Sc}}

\newcommand{\Thnorm}{1,\Th}

\newcommand{\IBDM}{\Pi_h\,}
\newcommand{\IL}{\mathcal{L}_h\,}
\newcommand{\IN}{\mathcal{I}_h\,}

\newcommand{\Iu}{\IBDM \vu}
\newcommand{\Ip}{\IL p}
\newcommand{\Ic}{\IN c}
\newcommand{\Is}{\IN s}

\newcommand{\eI}{E}  
\newcommand{\eP}{\xi} 
\newcommand{\eIu}{\eI_{\vu}}
\newcommand{\eIp}{\eI_p}
\newcommand{\eIc}{\eI_{c}}
\newcommand{\eIs}{\eI_{s}}

\newcommand{\ePu}{\eP_{\vu}}
\newcommand{\ePp}{\eP_p}
\newcommand{\ePc}{\eP_{c}}
\newcommand{\ePs}{\eP_{s}}

\newcommand{\vdiv}{\operatorname*{\mathbf{div}}}
\newcommand{\sdiv}{\operatorname*{div}}
\newcommand{\esssup}{\operatorname*{ess \, sup}}

\newcommand{\Th}{\mathcal{T}_h}		
\newcommand{\Eh}{\mathcal{E}_h}		

\newcommand{\dmean}[1]{{\left\{\kern-0.6ex\left\{ #1 
		\right\}\kern-0.6ex\right\}}}
\DeclarePairedDelimiter\djump{\llbracket}{\rrbracket}

\DeclarePairedDelimiter\norm{\lVert}{\rVert}
\DeclarePairedDelimiter\snorm{\lvert}{\rvert}
\DeclarePairedDelimiter\abs{\lvert}{\rvert}
\newcommand{\tnorm}[1]{{\left\vert\kern-0.25ex\left\vert\kern-0.25ex\left\vert #1 
		\right\vert\kern-0.25ex\right\vert\kern-0.25ex\right\vert}}
\newcommand{\Norm}{\norm{\cdot}}

\newtheorem{lemma}{Lemma}[section]
\newtheorem{theorem}{Theorem}[section]

\numberwithin{equation}{section}
\numberwithin{theorem}{section}
\numberwithin{table}{section}
\numberwithin{figure}{section}

\allowdisplaybreaks

\begin{document} 

\title[Error bounds for transient doubly-diffusive flows]{Divergence-conforming methods for transient doubly-diffusive flows: A priori and a posteriori error analysis} 

\author[R. B\"urger]{Raimund B\"urger}
\address{CI$^2$MA and Departamento de Ingenier\'\i a Matem\' atica, Universidad de Concepci\' on, 
Casilla 160-C, Concepci\' on, Chile}
\email{rburger@ing-mat.udec.cl}

\author[A. Khan]{Arbaz Khan}
\address{Department of Mathematics, Indian Institute of Technology Roorkee,  Roorkee 247667, India} 
\email{arbaz@ma.iitr.ac.in}

\author[P. E. M\'endez]{Paul E. M\'endez}
\address{CI$^2$MA and Departamento de Ingenier\'\i a Matem\' atica, Universidad de Concepci\' on, 
Casilla 160-C, Concepci\' on, Chile}
\email{pmendez@ing-mat.udec.cl}

\author[R. Ruiz-Baier]{Ricardo Ruiz-Baier}
\address{School of Mathematics, Monash University, 
  9~Rainforest Walk, Melbourne, VIC~3800, 
Australia; and Institute of Computer Science and Mathematical Modeling, Sechenov University, Moscow, Russian Federation}
\email{ricardo.ruizbaier@monash.edu}

\thanks{This work has been supported by ANID (Chile) through Fondecyt  project 1210610, CMM, project ANID/PIA/AFB170001, and CRHIAM, project ANID/FONDAP/15130015; by the Sponsored Research \& Industrial Consultancy (SRIC), Indian Institute of Technology Roorkee, India through the faculty initiation
grant MTD/FIG/100878; 
and by the Ministry of Science and Higher Education of the Russian Federation within the framework of state support for the creation and development of World-Class Research Centers ``Digital biodesign and personalized healthcare'' No. 075-15-2020-926.}

\subjclass[2010]{65M60; 76S05; 76R50}
\keywords{{\it A priori} and {\it a posteriori} error bounds; Mixed \textbf{H}(div)-conforming 
methods; Coupled Navier-Stokes and double-diffusion; Sedimentation and salinity variations.}

\date{\today}  

\begin{abstract}
The analysis of the double-diffusion model and $\mathbf{H}(\mathrm{div})$-conforming method introduced in [B\"urger, M\'endez, Ruiz-Baier, SINUM (2019), 57:1318--1343] is extended  to the time-dependent case. In addition,   the efficiency and reliability analysis of residual-based {\it a posteriori} error estimators for the steady, semi-discrete, and fully discrete problems is established. The  resulting methods are applied to simulate the sedimentation of small particles in salinity-driven flows. The method consists of Brezzi-Douglas-Marini approximations for velocity and compatible piecewise discontinuous pressures, whereas Lagrangian elements are used for concentration and salinity distribution. Numerical tests confirm the properties of the proposed family of schemes  and of the adaptive strategy guided by the {\it a posteriori} error indicators. 
\end{abstract}

\maketitle

\section{Introduction and problem formulation}
\subsection{Scope}
A number of physical problems of relevance in industrial applications 
involve coupled incompressible flow and double-diffusion transport. 
 We are interested in  numerical schemes 
  for the approximation of a class of coupled 
    equations  that arise as   models of sedimentation of small particles 
      under the effect of salinity of the ambient fluid. The governing model 
      can be stated as follows (cf., e.g., \cite{burns15,reali17}): 
      \begin{subequations}\label{eq:model}
      \begin{align} 
\partial_t \boldsymbol{u} + \boldsymbol{u} \cdot \boldsymbol{\nabla}  \boldsymbol{u} 
 & =  \vdiv \bigl(\nu(c) \boldsymbol{\nabla} \vu \bigr)  
  - \frac{1}{\rho_{\mathrm{m}}} \nabla p + \frac{\rho}{\rho_{\mathrm{m}}}  \boldsymbol{g}, \label{goveq1} \\ 
    \sdiv \boldsymbol{u} & =0, \label{goveq2} \\
    \partial_t s + \boldsymbol{u} \cdot \nabla s  & = \frac{1}{\Sc} \nabla^2 s, \label{goveq3} \\ 
     \partial_t  c + ( \boldsymbol{u} - v_{\mathrm{p}}   \boldsymbol{e}_z ) \cdot \nabla 
      c & = \frac{1}{\tau \Sc} \nabla^2 c,  \label{goveq4} 
 \end{align}
 \end{subequations}        
     posed on a spatial domain $\Omega \subset \mathbb{R}^d$, $d=2$ or~$d=3$, where $t\in (0,t_{\mathrm{end}}]$~is time,
      $\boldsymbol{u}$ is the fluid velocity, $\nu$~is the  concentration-dependent viscosity, 
       $\rho_{\mathrm{m}}$ is the mean density of the fluid, $p$~is the fluid pressure, 
        $\rho$~is density, $\boldsymbol{g}$ is the gravity acceleration, $s$~is the salinity concentration, $c$~is the concentration of solid particles, $\Sc = \nu_{\mathrm{ref}} / \kappa_{\mathrm{s}}$ is the Schmidt number, where~$\kappa_{\mathrm{s}}$ is the diffusivity of salinity, $\nu_{\mathrm{ref}}$ is a reference viscosity in the absence of solid particles, and $\tau = \kappa_{\mathrm{s}} / \kappa_{\mathrm{c}}$ is the inverse of the diffusivity ratio, 
         where $\kappa_{\mathrm{c}}$ is the diffusivity of solid particles,  and $\boldsymbol{e}_z$ is the upward-pointing unit vector.
        We relate the densities through a linearised equation of state
        \begin{align*}
        {\rho} &= {\rho_{\mathrm{m}}}(\alpha s + \beta c). 
        \end{align*}
        Again as in \cite{burns15,reali17}, the solid particles are assumed  mono-sized with radius~$r$, and 
         settle at dimensionless velocity $\smash{v_{\mathrm{p}} = v_{\mathrm{St}} / ( \nu_{\mathrm{ref}} g')^{1/3}}$, where 
         $\smash{v_{\mathrm{St}} =  2 r^2 ( \rho_{\mathrm{p}} - \rho_{\mathrm{m}})g/ (9 \rho_{\mathrm{m}} \nu_{\mathrm{ref}})}$ 
            is the Stokes velocity (settling velocity of a single particle in an unbounded fluid).  
The coupling mechanisms between flow and transport are only due to 
advection for concentration and salinity (where 
the advecting velocity for concentration, $\vu - v_\mathrm{p}\ve_z$, is also divergence-free), 
and through the concentration-dependent viscosity.       
  Further details are provided in later parts of the paper.

To put the paper further into the proper perspective, we mention that there exists 
an abundant body of literature devoted to constructing accurate finite element and 
related schemes for doubly-diffusive flows. Some recent contributions include 
variational multiscale stabilised schemes, least-squares methods, divergence-conforming mixed methods, volume-averaging discretisations, spectral elements, vorticity-based finite element formulations, and similar methods applied to, e.g., 
flows with heat and mass transport \cite{abedi03}, reactive Boussinesq flows \cite{agouzal03}, 
nonlinear advection-reaction-diffusion in the context of bioconvective flows \cite{anaya18,lenarda17}, cross-diffusion and boundary layer effects in doubly-diffusive Navier-Stokes-Brinkman equations \cite{buerger19,dallmann16} and in Darcy-Brinkman equations \cite{yang18}, or phase change models \cite{danaila14,woodfield19,zabaras04,zimmerman18}; 
where the list is far from exhaustive.

The solvability analysis for the continuous and discrete problems usually follows 
energy and fixed-point schemes as done for classical Boussinesq equations, and this 
is also the approach we follow here. The discretisation in space uses an interior penalty 
divergence-conforming 
method for the flow equations (in this case, Brezzi-Douglas-Marini  (BDM) elements of degree $k\geq 1$ for the velocity and discontinuous elements of degree $k-1$ for the pressure following 
\cite{arnold2001, konno2011}),  
combined with Lagrangian elements for the diffusive quantities, and the 
development stands as a natural extension of the formulation in \cite{buerger19} to 
the transient case. As such, it also features exactly divergence-free velocity approximations 
ensuring local conservativity and energy stability, and the error estimates of velocity 
are pressure-robust. The chosen time discretisation is the backward differentiation formula 
of degree 2 (BDF2), which for $k=2$ gives a method of order 2 in space and time.  
Existence of discrete solutions is established by the Brouwer fixed-point 
theory similarly as in \cite{buerger19}, and the error analysis in the semi-discrete and fully-discrete settings is adapted from the theory  of  \cite{aldbaissy2018} for the Boussinesq equations.

In many   applications  where  double-diffusion effects occur, complicated flow patterns exist in zones far from boundary layers and sufficiently refined meshes are needed essentially in the whole spatial domain. However, for salinity-driven settling of solid particles that result in mathematical models such as \eqref{eq:model}, many of the flow features are clustered near zones of high-gradients of concentration, which is where the typical \emph{plumes} are observed \cite{burns15,lenarda17}. This motivates the use of adaptive mesh refinement guided by {\it a posteriori} error indicators. For instance, in the context of phase change models there are some results based on error-related metric change \cite{danaila14,rako20} and on goal-oriented adaptivity \cite{zimmerman18}. 
Regarding the design and rigorous analysis of residual-based {\it a posteriori} error estimators for flow-transport couplings,  the literature is 
predominantly focused on the stationary case (see, e.g., \cite{agroum17,allali05,allendes20,alvarez18,becker02,dib19,wilfrid19,zhang11} and the references therein). Only a few results are available for the time-dependent regime, from 
which we mention the adaptive mixed method for Richards equation in porous media \cite{bernardi14}, the remeshing scheme based on goal-oriented adaptivity for solidification problems advanced in \cite{belhamadia19}, the collection of adaptive schemes for reactive flow discussed in \cite{braack07} and for heat transfer in \cite{larson08}. However, none of these theoretical frameworks is directly applicable to 
\eqref{eq:model}  using divergence-conforming approximations.  

The {\it a posteriori} error analysis we advance here is of residual type, and its analysis uses ideas from the abstract results 
related to spatial estimators for discontinuous Galerkin schemes applied to parabolic problems in 
\cite{georgoulis11}. The approach hinges on a decomposition of the discrete solution into a conforming and a non-conforming contribution,  along with a reconstruction technique (see also \cite{memon12}). This has also been exploited for the construction of {\it a posteriori} estimators of time-dependent Stokes and Navier-Stokes equations \cite{bansch19,zhang17}.  
Our {\it a posteriori} error analysis is divided into three parts. In first part, we present the  error estimator for the steady coupled problem. In second part, we extend the {\it a posteriori} error estimation to the semi-discrete method, and finally we present the {\it a posteriori} error estimator for the unsteady coupled problem. We restrict that part of the analysis to the simpler backward Euler method. To the best of our knowledge, the {\it a posteriori} error estimation advanced in this paper is the first comprehensive study targeted for transient doubly-diffusive flows.

The remainder of this paper is organised as 
 follows.  In what is left of this section we outline the weak formulation of \eqref{eq:model} 
 and state the stability the continuous problem. In Section~\ref{sec:fe}  we introduce 
 the divergence-conforming method in fully discrete form, show existence of discrete solutions using fixed-point arguments, 
 and rigorously establish {\it a priori} error estimates. Section~\ref{sec:aposteriori1} is devoted to the construction and analysis of efficiency and reliability for a residual-based {\it a posteriori} error estimator tailored for 
 the stationary problem. In turn, these upper and lower bounds are used to establish properties 
 of a second family of estimators for the transient case, and addressed in Sections~\ref{sec:aposteriori2} and \ref{sec:aposteriori3}. In Section~\ref{sec:numer} we collect 
 numerical tests that verify the theoretical convergence rates 
 predicted by the {\it a priori} 
 error analysis, confirm the robustness of the proposed {\it a posteriori} error estimators, 
 and illustrate the advantages of adaptive methods in the simulation of doubly-diffusive flows.  


\subsection{Preliminaries} \label{subsec:prel}   Let $\Omega$ be an open and bounded domain in $\mathbb{R}^{d}$, $d=2,3$ with Lipschitz boundary $\Gamma=\partial \Omega$. We denote by $L^p(\Omega)$ and  $W^{r,p}(\Omega)$ the usual Lebesgue and Sobolev spaces  with respective norms $\smash{\Norm_{L^p(\Omega)}}$ and $\smash{\Norm_{W^{r,p}(\Omega)}}$. If $p=2$ we write $H^r(\Omega)$ in place of $\smash{W^{r,p}(\Omega)}$, and denote the corresponding norm by $\smash{\Norm_{r,\Omega}}$, ($\smash{\Norm_{0,\Omega}}$ for $H^0(\Omega) = L^2(\Omega)$).  The space $L^2_0(\Omega)$ denotes the restriction of $L^2(\Omega)$
to functions with zero mean value over~$\Omega$. For $r\geq0$, we write the $H^r$-seminorm as 
$\smash{\snorm{\cdot}_{r,\Omega}}$ and we denote by $(\cdot,\cdot)_{\Omega}$ the usual inner product in $L^2(\Omega)$. Spaces of vector-valued functions  (in dimension $d$)  are denoted in bold face, i.e., $\smash{\vH^r(\Omega) = \left[ H^r(\Omega) \right]^{d}}$, and  we  use the vector-valued Hilbert spaces
	\begin{align*}
	\vH(\sdiv;\Omega) &\coloneqq \bigl\{\vw \in \vL^2(\Omega): \sdiv \vw \in L^2(\Omega) \bigr\}, \\
	\vH_0(\sdiv;\Omega) &\coloneqq \bigl\{\vw \in \vH(\sdiv;\Omega): \vw\cdot\vn_{\partial\Omega} = 0 \text{ on }\partial\Omega \bigr\}, \\
	\vH_0(\sdiv\!^0;\Omega) &\coloneqq \bigl\{\vw \in \vH_0(\sdiv;\Omega) : \sdiv \vw = 0 \text{ in }\Omega \bigr\},
	\end{align*}
where $\vn_{\partial\Omega}$ denotes the outward normal on $\partial\Omega$; and we endow these spaces with the norm~$\norm{\cdot}_{\sdiv,\Omega}$ defined by 
$\smash{\norm{\vw}_{\sdiv,\Omega}^2 \coloneqq \norm{\vw}_{0,\Omega}^2 + \norm{\sdiv \vw}_{0,\Omega}^2}$.
We denote by $L^s(0,\Tend;W^{m,p}(\Omega))$ the Banach space of all $L^s$-integrable functions from $[0,\Tend]$ into $W^{m,p}(\Omega)$, with norm 
\begin{align*} 
\norm{v}_{L^s(0,\Tend;W^{m,p}(\Omega))} = \begin{cases} \displaystyle
\left(\int_{0}^{\Tend} \norm{v(t)}^s_{W^{m,p}(\Omega)} \dt \right)^{1/s} & \text{if $1\leq s < \infty$,} \\
\esssup_{t \in [0,\Tend]} \norm{v(t)}_{W^{m,p}(\Omega)} & \text{if $s = \infty$.} 
\end{cases}
\end{align*}

\subsection{Additional assumptions and weak formulation} \label{subsec:weak}
 As in, e.g., \cite{dib19}, we  assume that viscosity is a Lipschitz 
continuous and uniformly bounded function of concentration, i.e., 
\begin{equation*}
	\bigl|\nu(c_1)-\nu(c_2) \bigr|  \leq L_{\nu}|c_1-c_2| \quad \text{and} \quad 
	 \nu_1 \leq \nu(c) \leq \nu_2,
\end{equation*}
for any $c,c_1,c_2\in \mathbb{R}$, and where $L_{\nu},\nu_1,\nu_2$ are positive constants. 

For simplicity of notation in presenting the analysis we will restrict the weak form 
to the case of homogeneous Dirichlet boundary conditions for velocity, concentration, and salinity. Testing each equation in problem \eqref{eq:model}  against suitable functions and integrating by parts whenever adequate, gives the following 
 weak formulation: For all $t\in(0,\Tend]$, find $(\vu,p,s,c) \in \vH^1_0(\Omega) \times L_0^2(\Omega) 
\times H^1_0(\Omega) \times H^1_0(\Omega)$ such that 
\begin{subequations}
\label{eq:main_var_mod}
 \begin{align}
  (\partial_t \vu, \vv)_{\Omega} + a_1(c;\vu,\vv) + c_1(\vu;\vu,\vv) + b(\vv,p) & = F(s,c,\vv) \quad \text{for all $\vv\in \vH_0^1(\Omega)$,} \\
	 b(\vu,q) & = 0 \quad\text{for all $q \in L_0^2(\Omega)$,} \\ 
	  (\partial_t s, \varphi)_\Omega + a_2(s,\varphi) + c_2(\vu;s,\varphi) & = 0  \quad   \text{for all $\varphi\in H_0^1(\Omega)$,} \label{eq:main_var_mod-3}\\
	  (\partial_t c, \psi)_\Omega +\frac{1}{\tau} a_2(c,\psi) + c_2(\vu-v_\mathrm{p}\ve_z;c,\psi) & = 0  \quad   \text{for all $\psi\in H_0^1(\Omega)$,} \label{eq:main_var_mod-4}
	\end{align}
	\end{subequations} 
where the bilinear and trilinear forms $a_1: H_0^1(\Omega)\times\vH_0^1(\Omega)\times\vH_0^1(\Omega)\to \mathbb{R}$, $a_2: H_0^1(\Omega)\times H_0^1(\Omega) \to \mathbb{R}$, $b: \vH_0^1(\Omega)\times L_0^2(\Omega)\to\mathbb{R}$, $c_1: \vH_0^1(\Omega)\times\vH_0^1(\Omega)\times\vH_0^1(\Omega)\to \mathbb{R}$, $c_2: \vH_0^1(\Omega)\times H_0^1(\Omega)\times H_0^1(\Omega)\to \mathbb{R}$, as well as the linear functional $F:\vH^1_0(\Omega)\to \mathbb{R}$, are defined as follows for all $\vu,\vv,\vw\in \vH_0^1(\Omega)$, $q\in L_0^2(\Omega)$, and $\varphi,\psi\in H_0^1(\Omega)$: 
	\begin{align*}
	&a_1(c;\vu,\vv) \coloneqq   \bigl(\nu(c)\bnabla \vu, \bnabla\vv \bigr)_{\Omega}, \quad 
	c_1(\vw;\vu,\vv) \coloneqq \bigl((\vw \cdot \bnabla) \vu, \vv \bigr)_{\Omega},  \quad 
	F(s,c,\vv) = (\alpha s + \beta c)( \vg,\vv)_\Omega, \\ 
	&b(\vv,q) \coloneqq  \frac{1}{\rho_\mathrm{m}}(q,\sdiv \vv)_{\Omega}, \quad 
	 a_2(\varphi,\psi) \coloneqq \frac{1}{\text{Sc}}(\nabla \varphi, \nabla \psi)_{\Omega},    
	\quad c_2(\vv; \varphi, \psi)\coloneqq \bigl((\vv \cdot \nabla) \varphi, \psi \bigr)_{\Omega},
	\end{align*}

\subsection{Stability of the continuous problem}\label{sec:wellp}

We begin with he following auxiliary result
\begin{lemma} \label{lemma:L4_decomp}
	For $d = 2$ the following inequality holds:  
	\begin{align*}
	\norm{\vv}^2_{\vL^4(\Omega)} \leq \sqrt{2} \norm{\vv}_{0,\Omega} \snorm{\vv}_{1,\Omega}.
	\end{align*}
\end{lemma}

The variational forms defined above are continuous for all 
$\vu,\vv,\in \smash{\vH^1_0}(\Omega)$, $q\in L^2_0(\Omega)$, and $\varphi,\psi \in   \smash{H^1_0}(\Omega)$: 
\begin{subequations}
	\begin{align}
	\abs[\big]{a_1(\cdot,\vu,\vv)} &\leq C_a \norm{\vu}_{1,\Omega} \norm{\vv}_{1,\Omega}, \qquad 
	\abs[\big]{a_2(\varphi,\psi)} \leq \hat{C}_a \norm{\varphi}_{1,\Omega} \norm{\psi}_{1,\Omega}, \label{eq:cont_a2}\\
	\abs[\big]{b(\vv,q)} &\leq C_b \norm{\vv}_{1,\Omega} \norm{q}_{0,\Omega}, \\
	\abs[\big]{c_1(\vw;\vu,\vv)} &\leq C_c \norm{\vw}_{1,\Omega} \norm{\vu}_{1,\Omega} \norm{\vv}_{1,\Omega}, \qquad 
	\abs[\big]{c_2(\vu;\varphi,\psi)} \leq \hat{C}_c \norm{\vu}_{1,\Omega} \norm{\varphi}_{1,\Omega} \norm{\psi}_{1,\Omega}.
	\end{align}
\end{subequations}

We also recall (from \cite[Chapter I, Lemma 3.1]{Girault2005}, for instance) the following Poincar\'e-Friedrichs inequality:
\begin{align} \label{eq:poincare}
\norm{u}_{0,\Omega} \leq C_p \snorm{u}_{1,\Omega}\qquad \text{for all $u \in H_0^1(\Omega)$}.
\end{align}

Next, using \eqref{eq:poincare} readily gives the coercivity of~$a_2$ and also, for a fixed concentration, that of~$a_1$, i.e., 
\begin{subequations}
\begin{align}
a_1(\cdot,\vv,\vv) & \geq \alpha_a \norm{\vv}^2_{1,\Omega} \qquad \text{for all $\vv \in \vH^1_0({\Omega})$,}  \label{eq:coerciv_a}\\
a_2(\varphi,\varphi) & \geq \hat{\alpha}_a \norm{\varphi}^2_{1,\Omega} \qquad \text{for all $\varphi \in H_{0}^1({\Omega})$.}  \label{eq:coerciv_am}
\end{align}\end{subequations}

Using the definition and characterisation of the kernel of~$b(\cdot,\cdot)$ 
\begin{align*} 
\vX\coloneqq\bigl\{ \vv \in \vH^1_0(\Omega)\, : \, b(\vv,q)=0 \; \forall q \in L_0^2(\Omega) \bigr\} = \bigl\{ \vv \in \vH_0^1(\Omega)
 \, : \, \sdiv \vv = 0 \; \text{in $\Omega$} \bigr\},   
\end{align*}
and using integration by parts we can readily observe that 
\begin{align} 
\text{$c_1(\vw;\vv,\vv) = 0$  and  $c_2(\vw;\varphi,\varphi) = 0$ \qquad for all  
		$\vw \in \vX, \vv \in \vH^1(\Omega),  \varphi \in H^1(\Omega)$.} 
\label{eq:c_zero}
\end{align}
It is well known that 
the bilinear form $b(\cdot,\cdot)$ satisfies the inf-sup condition (see, e.g., \cite{temam77}):
\begin{align*}
\sup_{\vv \in \vH_0^1(\Omega) \backslash \{\cero\}} \frac{b(\vv,q)}{\norm{\vv}_{1,\Omega}} &\geq \beta \norm{q}_{0,\Omega} 
 \quad \text{for all $q \in L_0^2(\Omega)$}.   
\end{align*}
Finally, for $\vv \in \vW^{1,\infty}(\Omega)$ and $\varphi \in W^{1,\infty}(\Omega)$ 
there exists an embedding constant $C_{\infty} > 0 $ such that
\begin{equation*}
	\norm{\vv}_{1,\Omega} \leq C_{\infty} \norm{\vv}_{\vW^{1,\infty}(\Omega)} \quad \text{and} \quad 
	\norm{\varphi}_{1,\Omega} \leq C_{\infty} \norm{\varphi}_{W^{1,\infty}(\Omega)}.
\end{equation*}

\begin{lemma}[Stability]If $\vg \in L^2(t,\Tend;\vL^2(\Omega))$, $\vu_0 \in \vL^2(\Omega)$ and $s_0,c_0 \in L^2(\Omega)$, then, for any solution $\vu,s,c$ of \eqref{eq:main_var_mod} and for $t \in (0,\Tend]$, there exists a constant $\gamma>0$ such that
	\begin{align*}
	\norm{\vu}_{L^2(0,t;\vH^1(\Omega))} + \norm{s}_{L^2(0,t;H^1(\Omega))} + \norm{c}_{L^2(0,t;H^1(\Omega))} &\leq \gamma \bigl(\norm{\vu_0}_{0,\Omega} + \norm{s_0}_{0,\Omega} + \norm{c_0}_{0,\Omega} \bigr),
	\end{align*}
	where $\gamma$ depends on $\eta_1, \Sc, \rho,\rho_{\mathrm{m}},C_p$, $\norm{g}_{\infty,\Omega}$, $\alpha$ and $\beta$.
\end{lemma}
\begin{proof}
We can take $\vu$ on $\vX$ and due to the inf-sup condition we can solve an equivalent reduced problem where $b(\cdot,\cdot)$ is removed from the variational form \eqref{eq:main_var_mod}. Setting $\vv=\vu$  and using \eqref{eq:c_zero}, \eqref{eq:coerciv_a}, we have
	\begin{align*}
	\frac{1}{2} \frac{\mathrm{d}}{\mathrm{d} t}\norm{\vu}_{0,\Omega}^2 + \alpha_a \norm{\vu}_{1,\Omega}^2 \leq \norm{\vg}_{\infty,\Omega}C \bigl( \norm{s}_{0,\Omega} + \norm{c}_{0,\Omega} \bigr) \norm{\vu}_{0,\Omega}.
	\end{align*}
	Now we use Young's inequality with $\varepsilon = \alpha_a/4$ to get
	 	\begin{align*}
	 \frac{\mathrm{d}}{\mathrm{d} t}\norm{\vu}_{0,\Omega}^2 + \frac{\alpha_a}{2} \norm{\vu}_{1,\Omega}^2 \leq \frac{C \norm{g}_{\infty,\Omega}^2}{\alpha_a} \bigl(\norm{s}_{0,\Omega}^2 + \norm{c}_{0,\Omega}^2 \bigr).
	 \end{align*}
	Integrating this equation between $0$ and $t$ yields, in particular
	\begin{equation} \label{eq:res_stab_1}
	\norm{\vu(\cdot,t)}_{0,\Omega}^2 + \alpha_a \int_0^t \norm{\vu(\cdot,z)}_{1,\Omega}^2 \dz \leq \norm{\vu(\cdot,0)}_{0,\Omega} + \frac{C}{\alpha_a} \int_0^t \norm{c(\cdot,z)}_{0,\Omega}^2 \dz + \frac{C}{\alpha_a} \int_0^t \norm{s(\cdot,z)}_{0,\Omega}^2 \dz.
	\end{equation}
	Analogously, using \eqref{eq:coerciv_am} and \eqref{eq:c_zero} on \eqref{eq:main_var_mod-3} and \eqref{eq:main_var_mod-4}, after integrating between 0 and $t$ we find that
	\begin{align}
	\norm{s(\cdot,t)}_{0,\Omega}^2 + 2\hat{\alpha}_a \int_0^t \norm{s(\cdot,z)}_{1,\Omega}^2 \dz &\leq \norm{s(\cdot,0)}_{0,\Omega}^2, \label{eq:res_stab2} \\
	\norm{c(\cdot,t)}_{0,\Omega}^2 + 2\hat{\alpha}_a \int_0^t \norm{c(\cdot,z)}_{1,\Omega}^2 \dz &\leq \norm{c(\cdot,0)}_{0,\Omega}^2. \label{eq:res_stab3}
	\end{align}
	Finally, we derive the sought result from \eqref{eq:res_stab_1}, \eqref{eq:res_stab2} and \eqref{eq:res_stab3}.
\end{proof}

A problem similar to \eqref{eq:main_var_mod} has been studied in \cite{agroum2015}. Assuming that $F$ belongs to $L^2(0,\Tend;\vH^{-1}(\Omega))$, that the initial velocity $\vu_0$ belongs to $\vL^2(\Omega)$ and the initial data for the coupled species ($s,c$ in the context of our problem) belongs to $L^2(\Omega)$, the authors showed  existence of a  solution by using the Galerkin method and applying the Cauchy-Lipschitz theorem, and proved its  uniqueness in two dimensions. Such analysis can be applied to \eqref{eq:main_var_mod} by noting that  $F$ is a Lipschitz-continuous function, and assuming the initial data belongs to appropriate spaces. This is, however, not the focus of the paper.

\section{Finite element discretisation and {\it a priori} error bounds}\label{sec:fe}

We discretise  space by a family of regular partitions, denoted  $\Th$, of $\Omega \subset \mathbb{R}^d$ into simplices~$K$ 
(triangles in 2D or tetrahedra in 3D) of diameter~$h_K$. We   label by  $K^-$ and $K^+$ the elements adjacent to an edge, while 
$h_e$  stands for the maximum diameter of the edge. 
If~$\vv$ and~$w$ are 
smooth vector and scalar fields defined on~$\Th$,  then  ($\vv^\pm, w^\pm$)  denote the traces of~($\vv,w$) on~$e$ 
 that are  the extensions from the interior of~$K^+$ and~$K^-$, respectively. Let $\vn_e$ denote the outward unit normal vector to~$e$ on~$K$. 
The average $\dmean{\cdot}$ and jump $\djump{\cdot }$ operators are defined as 
\begin{align*}
\dmean{\vv} &\coloneqq (\vv^-+\vv^+)/2, \quad \dmean{w} \coloneqq  (w^-+w^+)/2, \quad 
\djump{\vv } \coloneqq  (\vv^- - \vv^+), \quad \djump{w } \coloneqq  (w^- - w^+),
\end{align*}
whereas for boundary jumps and averages we adopt the conventions   $\dmean{\vv} = \djump{\vv} = \vv$, 
and $\dmean{w} = \djump{w} = w$. In addition, 
$\nabla_h$ will  denote  the broken gradient operator.

For $k\geq 1$  and a mesh $\Th$ on $\Omega$, let us consider the discrete spaces (see e.g.~\cite{bdm85})
\begin{gather*}
\vV_h \coloneqq \Set{ \vv_h \in \vH(\sdiv; \Omega)}{\vv_h|_K \in [\Pol_{k}(K)]^{\mathrm{d}}\quad\forall K \in \Th}, \\
\mathcal{Q}_h \coloneqq \Set{ q_h \in L^2(\Omega)}{q_h|_K \in \Pol_{k-1}(K)\quad\forall K \in \Th}, \\
\mathcal{M}_h \coloneqq \Set{ s_h \in C(\bar{\Omega})}{l_h|_K \in \Pol_{k}(K) \quad\forall K \in \Th},\qquad 
\mathcal{M}_{h,0} \coloneqq \mathcal{M}_h \cap H^1_0(\Omega),
\end{gather*}
which, in particular, satisfy  $\sdiv\vV_h\subset \mathcal{Q}_h$ (cf. \cite{konno2011}). Here $\Pol_k(K)$ denotes the local space 
spanned by polynomials of degree up to  $k$ and $\vV_h$ is the space of divergence-conforming BDM  elements. We then 
state the following semi-discrete Galerkin formulation for
problem \eqref{eq:main_var_mod}: 
Find $(\vu_h, p_h, s_h,c_h) \in \vV_h\times \mathcal{Q}_h\times \mathcal{M}_h\times\mathcal{M}_h$ such that:
\begin{align}  \label{eq:discr_var_mod} 
\begin{split}  
(\pdt \vu_h, \vv_h)_{\Omega} + a^h_1(c_h;\vu_h,\vv_h) + c^h_1(\vu_h;\vu_h,\vv_h) - b(\vv_h,p_h) &= F(s_h, c_h, \vv_h) \quad \text{for all $\vv_h\in \vV_h$,} \\
b(\vu_h,q_h) &= 0 \quad\text{for all $q_h \in \mathcal{Q}_h$,}\\
(\pdt s_h, \varphi_h)_{\Omega} + a_2(s_h,\varphi_h) + c_2(\vu_h;s_h,\varphi_h) &= 0 \quad   \text{for all $\varphi_h\in \mathcal{M}_{h}$,}\\
(\pdt c_h, \psi_h)_{\Omega} + \frac{1}{\tau} a_2(c_h,\psi_h) + c_2(\vu_h-v_\mathrm{p}\ve_z;c_h,\psi_h) &= 0 \quad   \text{for all $\psi_h\in \mathcal{M}_{h}$.}
\end{split}
\end{align}
The discrete versions of the trilinear forms $a_1^h(\cdot;\cdot,\cdot)$ and $c_1^h(\cdot;\cdot,\cdot)$ are defined using a symmetric interior penalty and an upwind approach, respectively (see, e.g., \cite{arnold2001, konno2011}):
\begin{align}
a^h_1(c_h;\vu_h, \vv_h)  
& \coloneqq \int_{\Omega} \nu(c_h) \boldsymbol{\nabla}_h(\vu_h) : \boldsymbol{\nabla}_h(\vv_h)  +
\sum_{e\in\mathcal{E}_h} \int_e \biggl( -\dmean{\nu(c_h) \boldsymbol{\nabla}_h(\vu_h) \vn_e} \cdot \djump{\vv_h } \nonumber  \\&  \qquad - \dmean{\nu(c_h) \boldsymbol{\nabla}_h(\vv_h) \vn_e} \cdot \djump{\vu_h }
+ \frac{a_0}{h_e}\nu(c_h) \djump{\vu_h } \cdot \djump{\vv_h } \biggr), \label{a1hdef} \\
c_1^h (\vw_h; \vu_h, \vv_h) & \coloneqq \int_\Omega (\vw_h \cdot \nabla)\vu_h\cdot\vv_h + \sum_{K \in \Th}\int_{\partial K \backslash \Gamma} \boldsymbol{\hat{w}}_h^{\mathrm{up}}(\vu_h) \cdot \vv_h, \nonumber 
\end{align}
where the upwind flux is defined as $\smash{\boldsymbol{\hat{w}}_h^{\mathrm{up}}(\vu_h) \coloneqq \frac{1}{2}(\vw_h \cdot \vn_k - \abs{\vw_h \cdot\vn_K})(\vu^e_h - \vu_h})$, and $\vu_h^e$ is the trace of $\vu_h$ taken from within the exterior of $K$.

We partition  the interval $[0,\Tend]$ into $N$ subintervals $[t_{n-1},t_n]$ of length $\Dt$. We  use the implicit BDF2 scheme where 
all first-order time derivatives are approximated using the centred operator  
\begin{align} \label{centrop} 
\partial_t \vu_h (t^{n+1})\approx \frac{1}{\Dt}\biggl(\frac{3}{2}\vu_h^{n+1}-2\vu_h^n
+\frac{1}{2}\vu_h^{n-1}\biggr), \end{align} 
(similarly for $\partial_t c$) and for the first time step a first-order backward Euler method is used 
from $t^0$ to $t^1$, starting from the interpolates $\vu_h^0, s_h^0$ of the initial data.  
In what follows, we define the difference operator 
 \begin{align*} 
  \mathcal{D} y^{n+1}  \coloneqq  3 y^{n+1} - 4 y^n   + y^{n-1} 
  \end{align*} 
   for any quantity indexed by the time step~~$n$.  For instance,  \eqref{centrop} can be written as 
   $\partial_t \vu_h (t^{n+1})\approx  \frac{1}{2 \Delta t}  \mathcal{D} \boldsymbol{u}_h^{n+1}$.

 The resulting set of nonlinear equations is solved by   an iterative 
Newton-Raphson method with exact Jacobian.  
Hence the complete discrete system is given by 
\begin{align}   \label{eq:full_disc_scheme} 
\begin{split} 
 \frac{1}{3} \left(\mathcal{D}  \vu_h^{n+1},\vv_h \right)_{\Omega}    
&= \frac{2}{3}\Dt \bigl(- a^h_1(c_h^{n+1};\vu_h^{n+1},\vv_h) - c^h_1(\vu_h^{n+1};\vu_h^{n+1},\vv_h) - b(\vv_h,p_h^{n+1}) + F(s_h^{n+1},c_h^{n+1},\vv_h)\bigr) \\ & \quad \text{for all $\vv_h\in \vV_h$,}  \\
b(\vu_h^{n+1},q_h) & = 0 \quad\text{for all $q_h \in \mathcal{Q}_h$,}  \\
  \frac{1}{3} \left( \mathcal{D} s_h^{n+1},\varphi_h \right)_{\Omega} &  = \frac{2}{3}\Dt \bigl(-a_2(s_h^{n+1},\varphi_h) - c_2(\vu_h^{n+1};s_h,\varphi_h)\bigr) \quad   \text{for all $\varphi_h\in \mathcal{M}_{h}$,} \\
  \frac{1}{3} \left(\mathcal{D} c_h^{n+1},\psi_h \right)_{\Omega}  & = \frac{2}{3}\Dt \biggl(-\frac{1}{\tau} a_2(c_h^{n+1},\psi_h)-c_2(\vu_h^{n+1}-v_\mathrm{p}\ve_z;c_h^{n+1},\psi_h)\biggr) \quad   \text{for all $\psi_h\in \mathcal{M}_{h}$.}  
  \end{split} 
\end{align}


For the subsequent analysis, we introduce  for $r \geq 0$ the broken $\vH^r(\Th)$ space 
$$
\vH^r(\Th) = \Set{ \vv \in \vL^2(\Omega)}{\vv|_K \in \vH^r(K), K \in \Th},
$$
as well as the following mesh-dependent broken norms
\begin{gather*}
\norm{\vv}_{*,\Th}^2  \coloneqq \sum_{K\in \Th} \norm{\bnabla_h(\vv)}_{\vL^2(K)}^2 + \sum_{e\in\Eh} \frac{1}{h_e} \norm{ \djump{\vv } }_{\vL^2(e)}^2, \\ 
\norm{\vv}_{\Thnorm}^2   \coloneqq \norm{\vv}_{\vL^2(\Omega)}^2 + \norm{\vv}_{*,\Th}^2 \quad \text{for all $\vv \in \vH^1(\Th)$,}  \quad 
\norm{\vv}_{2,\Th}^2 \coloneqq \norm{\vv}_{\Thnorm}^2 + \sum_{K\in \Th} h_K^2 \snorm{\vv }_{H^2 (K)}^2 \quad \text{for all $\vv \in \vH^2(\Th)$,} 
\end{gather*}
where the stronger norm $\smash{\Norm_{2,\Th}}$ is used to show continuity. This norm is  equivalent to 
$\smash{\Norm_{\Thnorm}}$ on $\smash{\vH^1(\Th)}$ (see \cite{arnold2001}).
Finally, adapting the argument used in \cite[Proposition 4.5]{karakashian1998} we have the discrete Sobolev embedding:
for $r=2,4$ there exists a constant $C_{\textnormal{emb}}>0$ such that
\begin{align}
\norm{\vv}_{\vL^{r}(\Omega)} \leq C_{\textnormal{emb}} \norm{\vv}_{\Thnorm} \qquad \text{for all }\vv \in \vH^{1}(\Th).
\label{eq:discr_emb}
\end{align}
Using these norms, we can establish continuity of the 
trilinear and bilinear forms involved. The proof  follows from   \cite[Section 4]{arnold2001}. 
\begin{lemma} \label{lemma4:ab_bounds}
	The following properties hold:
	\begin{align*}
	\abs[\big]{a^h_1(\cdot,\vu,\vv)}  &\leq C \norm{\vu}_{2,\Th}\norm{\vv}_{\Thnorm}  &&\text{\em for all $\vu \in \vH^2(\Th)$, $\vv\in \vV_h$}, \\
	\abs[\big]{a^h_1(\cdot,\vu,\vv)}  &\leq \tilde{C}_a \norm{\vu}_{\Thnorm}\norm{\vv}_{\Thnorm}     &&\text{\em for all $\vu,\vv\in \vV_h$,}  \\
	\abs[\big]{b(\vv,q)}  &\leq \tilde{C}_b \norm{\vv}_{\Thnorm}\norm{q}_{0,\Omega}   &&\text{\em for all $ \vv \in \vH^1(\Th)$, $q\in L^2(\Omega)$,}  
	\end{align*} 
	and for all $ \vu, \vv, \vw  \in \smash{\vH^1(\Th)}$	and  $\smash{\varphi, \psi \in H^1(\Omega)}$, there holds 
	\begin{subequations}
		\begin{align}
		\abs[\big]{c_2(\vw; \varphi, \psi)}  &\leq \tilde{C} \norm{\vw}_{\Thnorm}\norm{\varphi}_{1,\Omega} \norm{\psi}_{1,\Omega}.  \label{eq:bound_Cm0}
		\end{align}
	\end{subequations}
\end{lemma}
Moreover, for   $c_1,c_2 \in H^1(\Omega)$, $\vu \in \vC^1(\Th) \cap \vH_0^1(\Omega)$ and $\vv\in \vV_h$, there 
holds 
\begin{align}
& \abs[\big]{a^h_1(c_1;\vu,\vv) - a^h_1(c_2;\vu,\vv)}  \leq \tilde{C}_{\mathrm{Lip}} \norm{c_1 - c_2}_{1,\Omega}\norm{\vu}_{\vW^{1,\infty}(\Th)}\norm{\vv}_{1,\Th},  \label{eq:a1lipschitz}
\end{align}
where the constant $\smash{\tilde{C}_{\mathrm{Lip}} >0}$ is independent of~$h$ (cf.\ \cite{buerger19}).
Note that while the coercivity of the form $a_2(\cdot,\cdot)$ in the discrete setting is readily implied by \eqref{eq:coerciv_am},  
there also holds (cf. \cite[Lemma 3.2]{konno2011})
\begin{equation}
a^h_1(\cdot,\vv,\vv) \geq \tilde{\alpha}_a \norm{\vv}_{\Thnorm}^2 \quad \text{for all $\vv\in \vV_h$,} 
\label{eq:coerciv_ah}
\end{equation}
provided that the stabilisation parameter $a_0>0$ in \eqref{a1hdef} is sufficiently large and independent of the mesh size. 

Let $\vw \in \vH_0(\sdiv^0;\Omega)$, then due to the skew-symmetric form of the operators~$c_1^h$ and~$c_2$, and the positivity of the non-linear upwind term of~$c_1^h$, we can write
\begin{subequations}\begin{align}
c^h_1(\vw;\vu,\vu)  &\geq 0 \quad \text{for all $\vu\in \vV_h$,} \label{eq:chgeq0} \\
c_{2}^h(\vw;\psi_h,\psi_h) &= 0 \quad \text{for all $\psi_h \in \mathcal{M}_{h}$,}
\label{eq:cTSeq0}
\end{align}\end{subequations}
as well as the following relation (which is based on \eqref{eq:discr_emb} and follows by the same method as in \cite{karakashian1998}): For any $\vw_1,\vw_2, \vu \in \vH^2(\Th)$ there holds for all $\vv \in \vV_h$
\begin{align} \label{eq:bound_ch1}
\abs[\big]{c_1^h(\vw_1;\vu,\vv)} - \abs[\big]{c_1^h(\vw_2;\vu,\vv)} &\leq \tilde{C}_c \norm{\vw_1-\vw_2}_{\Thnorm} \norm{\vv}_{\Thnorm} \norm{\vu}_{\Thnorm}.
\end{align} 
We also have
\begin{equation*}
F(\psi, \phi, \vv) \leq C_f \bigl(\norm{\psi}_{0,\Omega} + \norm{\phi}_{0,\Omega} \bigr)\norm{\vv}_{0,\Omega} \quad \text{for all } \vv \in \vV_h.
\end{equation*}
Finally, we recall from \cite{konno2011} the following discrete inf-sup condition for $b(\cdot,\cdot)$, where $\tilde{\beta}$ is independent of~$h$: 
\begin{equation}
\sup_{\vv_h\in \vV_h\backslash\{\cero\}} \frac{b(\vv_h,q_h)}{\norm{ \vv_h}_{1,\Th}} \geq 
\tilde{\beta} \norm{q_h}_{0,\Omega} \quad \text{for all $q_h \in \mathcal{Q}_h$.}  
\label{eq:infsup_disc}
\end{equation}


We will use the following algebraic relation: for any real numbers $a^{n+1}$, $a^n$, $a^{n-1}$ and 
defining $\Lambda a^n \coloneqq  a^{n+1} - 2a^n + a^{n-1}$, we have
\begin{align} \label{eq:alg_id} 
2(3a^{n+1} - 4a^n + a^{n-1}, a^n) &= \abs{a^{n+1}}^2 + \abs{2a^{n+1}-a^n}^2 + \abs{\Lambda a^n}^2 
  - \abs{a^{n}}^2 - \abs{2a^n-a^{n-1}}^2. 
\end{align}

\begin{theorem}
	Let $\smash{(\vu_h^{n+1}, p_h^{n+1},s_h^{n+1}, c_h^{n+1}) \in \vV_h \times \mathcal{Q}_h \times \mathcal{M}_{h,0} \times \mathcal{M}_{h,0}}$ be a solution of problem \eqref{eq:full_disc_scheme}. Then the following bounds are satisfied, where $C_1, C_2$ and $C_3$ are constants that are independent of $h$ and $\Dt$:
	\begin{align} \label{eq:utcs_bounds}
	\begin{split} 
	&\norm{\vu_h^{n+1}}^2_{0,\Omega} + \norm{2\vu^{n+1}_h - \vu_h^n}^2_{0,\Omega} +  \sum_{j=1}^{n} \norm{\Lambda \vu_h^j}^2_{0,\Omega} + \sum_{j=1}^{n} \Dt \norm{\vu^{j+1}_h}^2_{\Thnorm} \\ 
	& \quad \leq C_1 \bigl( \norm{s_h^{1}}^2_{0,\Omega} + \norm{2s_h^{1} - s_h^{0}}^2_{0,\Omega}+\norm{c_h^{1}}^2_{0,\Omega} + \norm{2c_h^{1} - c_h^{0}}^2_{0,\Omega}+ \norm{\vu_h^1}^2_{0,\Omega} + \norm{2\vu_h^1 - \vu_h^0}^2_{0,\Omega}  \bigr), \\
	&\norm{s_h^{n+1}}^2_{0,\Omega} + \norm{2s^{n+1}_h - s_h^{n}}^2_{0,\Omega} + \sum_{j=1}^{n} \norm{\Lambda s_h^{j}}^2_{0,\Omega} + \sum_{j=1}^{n} \Dt \norm{s^{j+1}_h}^2_{1,\Omega}\leq C_2 \bigl( \norm{s_h^{1}}^2_{0,\Omega} + \norm{2s_h^{1} - s_h^{0}}^2_{0,\Omega} \bigr), \\
	&\norm{c_h^{n+1}}^2_{0,\Omega} + \norm{2c^{n+1}_h - c_h^{n}}^2_{0,\Omega} + \sum_{j=1}^{n} \norm{\Lambda c_h^{j}}^2_{0,\Omega} + \sum_{j=1}^{n} \Dt \norm{c^{j+1}_h}^2_{1,\Omega} \leq C_3 \bigl( \norm{c_h^{1}}^2_{0,\Omega} + \norm{2c_h^{1} - c_h^{0}}^2_{0,\Omega} \bigr).
	\end{split}\end{align}
\end{theorem}
\begin{proof}
	First we take $\vv_h = 4\vu_h^{n+1}$ and $q_h = 4p_h^{n+1}$ in the first and second equation of \eqref{eq:full_disc_scheme}, respectively and apply \eqref{eq:alg_id}, \eqref{eq:coerciv_ah} and \eqref{eq:chgeq0} to deduce the estimate
	\begin{align*}
	&		\norm{\vu_h^{n+1}}^2_{0,\Omega} + \norm{2\vu^{n+1}_h - \vu_h^n}^2_{0,\Omega} + \norm{\Lambda \vu_h^n}^2_{0,\Omega} + 4 \Dt \tilde{\alpha}_a \norm{\vu^{n+1}_h}^2_{\Thnorm} 
	\\ & \quad 
	\leq 4 \Dt C_f (\norm{s_h^{n+1}}_{0,\Omega}+\norm{c_h^{n+1}}_{0,\Omega}) \norm{\vu_h^{n+1}}_{0,\Omega} + \norm{\vu_h^n}^2_{0,\Omega} + \norm{2\vu_h^n - \vu_h^{n-1}}^2_{0,\Omega}.
	\end{align*}
	Using Young's inequality with $\varepsilon = \tilde{\alpha}_a/2$ and summing over $n$ we can assert that
	\begin{align}  \label{eq:pr_3} \begin{split} 
	&\norm{\vu_h^{n+1}}^2_{0,\Omega} + \norm{2\vu^{n+1}_h - \vu_h^n}^2_{0,\Omega} + \sum_{j=1}^{n} \norm{\Lambda \vu_h^j}^2_{0,\Omega} + 2 \tilde{\alpha}_a \sum_{j=1}^{n} \Dt \norm{\vu^{j+1}_h}^2_{\Thnorm}  \\ &\quad \leq  \frac{C}{\tilde{\alpha}_a} \sum_{j=1}^n \Dt\norm{s_h^j}_{0,\Omega}^2  + \frac{C}{\tilde{\alpha}_a} \sum_{j=1}^n \Dt\norm{c_h^j}_{0,\Omega}^2 + \norm{\vu_h^1}^2_{0,\Omega} + \norm{2\vu_h^1 - \vu_h^{0}}^2_{0,\Omega}.  \end{split} 
	\end{align}
	
	Similarly in the third equation of \eqref{eq:full_disc_scheme}, we take $\smash{\varphi_h = 4s_h^{n+1}}$ and use property \eqref{eq:cTSeq0} and relation \eqref{eq:alg_id} to deduce the inequality 
	\begin{align*}
	& \norm{s_h^{n+1}}^2_{0,\Omega} + \norm{2s^{n+1}_h - s_h^{n}}^2_{0,\Omega} + \norm{\Lambda s_h^{n}}^2_{0,\Omega} + 4 \alpha_2 \Dt \norm{s^{n+1}_h}^2_{1,\Omega}   \leq  \norm{s_h^{n}}^2_{0,\Omega} + \norm{2s_h^{n} - s_h^{n-1}}^2_{0,\Omega}.
	\end{align*}
	Hence, summing over $n$, we get
	\begin{align*} \begin{split} 
	&	\norm{s_h^{n+1}}^2_{0,\Omega} + \norm{2s^{n+1}_h - s_h^{n}}^2_{0,\Omega} + \sum_{j=1}^{n} \norm{\Lambda s_h^{j}}^2_{0,\Omega} + 4\hat{\alpha}_a\sum_{j=1}^{n} \Dt \norm{s^{j+1}_h}^2_{1,\Omega}  \leq \norm{s_h^{1}}^2_{0,\Omega} + \norm{2s_h^{1} - s_h^{0}}^2_{0,\Omega}.
	\end{split} 
	\end{align*}
	We proceed in the same way taking $\smash{\psi_h = 4 c_h^{n+1}}$ in the fourth equation of \eqref{eq:full_disc_scheme}, to get the third result. We get the first result by substituting bounds for $c_h$ and $s_h$ into  \eqref{eq:pr_3}.
\end{proof}

\begin{theorem}[Existence of discrete solutions] \label{th2.2} 
	The 	problem 	\eqref{eq:full_disc_scheme} admits  at least one solution 
	\begin{align*} 
	(\vu_h^{n+1}, p_h^{n+1}, s_h^{n+1}, c_h^{n+1}) \in \vV_h\times \mathcal{Q}_h\times \mathcal{M}_{h,0} \times \mathcal{M}_{h,0}. \end{align*} 
\end{theorem}

The proof of Theorem~\ref{th2.2} makes use of Brouwer's fixed-point theorem in the following form 
 (given by \cite[Corollary 1.1, Chapter IV]{girault1986}): 

\begin{theorem}[Brouwer's fixed-point theorem] \label{th:brouwer}  Let $H$ be a finite-dimensional Hilbert space with scalar product  $(\cdot,\cdot )_H$ and corresponding norm $\norm{\cdot}_H$. Let $\Phi\colon H \to H$ be a continuous mapping  for which there exists $\mu > 0$ such that $(\Phi(u),u)_H \geq 0$ for all $u \in H$  with $\norm{u}_H  = \mu$. 		Then  there exists $u \in H$ such that $\Phi(u) = 0$ and $\norm{u}_H \leq \mu$. \end{theorem} 

\begin{proof}[Proof of Theorem~\ref{th2.2}]
	To simplify the proof we introduce the following constants:
	\begin{gather*} 
	C_u  \coloneqq  C_1 \bigl(\norm{s_h^{1}}^2_{0,\Omega} + \norm{2s_h^{1} - s_h^{0}}^2_{0,\Omega}+\norm{c_h^{1}}^2_{0,\Omega} + \norm{2c_h^{1} - c_h^{0}}^2_{0,\Omega}+\norm{\vu_h^1}^2_{0,\Omega} + \norm{2\vu_h^1 - \vu_h^0}^2_{0,\Omega}  \bigr), \\
	C_s  \coloneqq C_2 \bigl( \norm{s_h^{1}}^2_{0,\Omega} + \norm{2s_h^{1} - s_h^{0}}^2_{0,\Omega} \bigr), \qquad 
	C_c  \coloneqq C_3 \bigl( \norm{c_h^{1}}^2_{0,\Omega} + \norm{2c_h^{1} - c_h^{0}}^2_{0,\Omega} \bigr).
	\end{gather*}
		We proceed by induction on $n \geq 2$. We define the mapping 
	\begin{align*} 
	\Phi : \vV_h\times \mathcal{Q}_h \times \mathcal{M}_{h,0} \times \mathcal{M}_{h,0} \to \vV_h\times \mathcal{Q}_h \mathcal{M}_{h,0} \times \mathcal{M}_{h,0}  
	\end{align*} 
	using the relation
	\begin{align*}
	&\bigl( \Phi(\vu_h^{n+1}, p_h^{n+1},s_h^{n+1},c_h^{n+1}),(\vv_h, q_h, \varphi_h, \psi_h) \bigr)_{\Omega} \\
	& = \frac{( \mathcal{D} \vu_h^{n+1},\vv_h )_{\Omega}}{2\Dt}  + a_1^h \bigl(c_h^{n+1};\vu_h^{n+1},\vv_h \bigr) 
	+ c_1^h \bigl(\vu_h^{n+1};\vu_h^{n+1}, \vv_h \bigr) + b(\vv_h,p_h^{n+1}) -b(\vu_h^{n+1},q_h)  - F\bigl(s_h^{n+1},c_h^{n+1},\vv_h \bigr) 
	\\ & \quad+ \frac{( \mathcal{D} s_h^{n+1},\varphi_h)_{\Omega}}{2\Dt}   + a_2\bigl(s_h^{n+1},\varphi_h \bigr) 
	+ c_2 \bigl(\vu_h^{n+1};s^{n+1}_h,\varphi_h \bigr)+ \frac{(\mathcal{D} c_h^{n+1},\psi_h )_{\Omega}}{2\Dt}   + \frac{a_2 (c_h^{n+1},\psi_h  )}{\tau} \\ & \quad
	+ c_2 \bigl(\vu_h-v_\mathrm{p}\ve_z;c^{n+1}_h,\psi_h \bigr).
	\end{align*}

	Note this map is well-defined and continuous on $\vV_h \times\mathcal{Q}_h \times \mathcal{M}_{h,0} \times \mathcal{M}_{h,0}$. On the other hand, if we take 
	$$(\vv_h, q_h, \varphi_h, \psi_h)=(\vu_h^{n+1}, p_h^{n+1}, s_h^{n+1}, c_h^{n+1}),$$ 
	and employ \eqref{eq:chgeq0}, \eqref{eq:cTSeq0} and \eqref{eq:coerciv_ah}, we obtain
	\begin{align*}
	&\bigl( \Phi(\vu_h^{n+1},p_h^{n+1}, s_h^{n+1}, c_h^{n+1}),(\vu_h^{n+1}, p_h^{n+1}, s_h^{n+1}, c_h^{n+1}) \bigr)_{\Omega}\\
	& \geq -\frac{1}{2\Dt} \norm{4\vu_h^n-\vu_h^{n-1}}_{0,\Omega} \norm{\vu_h^{n+1}}_{0,\Omega} + \tilde{\alpha}_a\norm{\vu_h^{n+1}}^2_{\Thnorm}   - C_f  \bigl(\norm{s_h^{n+1}}_{0,\Omega}+\norm{c_h^{n+1}}_{0,\Omega} \bigr)\norm{\vu_h^{n+1}}_{0,\Omega} \\ & \quad - \frac{1}{2\Dt} \norm{4s_h^{n}-s_h^{n-1}}_{0,\Omega}\norm{s_h^{n+1}}_{0,\Omega}+ \hat{\alpha}_a\norm{s^{n+1}_h}_{1,\Omega}^2 - 
	\frac{1}{2\Dt}\norm{4c_h^{n}-c_h^{n-1}}_{0,\Omega}\norm{c_h^{n+1}}_{0,\Omega} 
	+ \frac{\hat{\alpha}_a}{\tau}\norm{c^{n+1}_h}_{1,\Omega}^2.
	\end{align*}
	Next,   using \eqref{eq:utcs_bounds} we deduce that
	\begin{align*}
	&\bigl( \Phi(\vu_h^{n+1}, p_h^{n+1}, s_h^{n+1}, c_h^{n+1}),(\vu_h^{n+1}, p_h^{n+1}, s_h^{n+1}, c_h^{n+1}) \bigr)_{\Omega} \\
	&  \geq \frac{\tilde{\alpha}_a}{2} \norm{\vu_h^{n+1}}^2_{0,\Omega} + \hat{\alpha}_a \norm{s_h^{n+1}}^2_{0,\Omega} + \frac{\hat{\alpha}_a}{\tau} \norm{c_h^{n+1}}^2_{0,\Omega}  - \frac{5}{2\Dt} C_u \norm{\vu_h^{n+1}}_{0,\Omega} - C_f(C_u+C_c)^{1/2}\norm{\vu_h^{n+1}}_{0,\Omega}\\
	& \quad - \frac{5}{2\Dt} C_s \norm{s_h^{n+1}}_{0,\Omega} - \frac{5}{2\Dt} C_c \norm{c_h^{n+1}}_{0,\Omega}.
	\end{align*}
	Then,	setting 
	\begin{align*} 
	C_R = \min\left \{ \tilde{\alpha}_a, \frac{\hat{\alpha}_a}{\tau}  \right\}, \quad  C_r = \sqrt{2} \max \left\{ \frac{5}{2\Dt} C_u, C_f(C_u+C_c)^{1/2}, \frac{5}{2\Dt} C_s, \frac{5}{2\Dt} C_c \right\},
	\end{align*}  
	we may apply the inequality $a+b \leq \sqrt{2} (a^2+b^2)^{1/2}$, valid for all $a,b\in \mathbb{R}$, to obtain
	\begin{align*}
	&	\bigl( \Phi(\vu_h^{n+1}, p_h^{n+1},  s_h^{n+1}, c_h^{n+1}),(\vu_h^{n+1}, p_h^{n+1}, s_h^{n+1}, c_h^{n+1}) \bigr)_{\Omega} \\ &\quad \geq  
	C_{R} \bigl(\norm{\vu_h^{n+1}}_{0,\Omega}^2 +  \norm{s_h^{n+1}}_{0,\Omega}^2  + \norm{c_h^{n+1}}_{0,\Omega}^2 \bigr)   
	- C_{r} \bigl(\norm{\vu_h^{n+1}}_{0,\Omega}^2 + \norm{s_h^{n+1}}_{0,\Omega}^2 + \norm{c_h^{n+1}}_{0,\Omega}^2 \bigr)^{1/2}.
	\end{align*} 
	Hence, the right-hand side is nonnegative on a sphere of radius $r  \coloneqq   C_r / C_R$. Consequently, by  
	Theorem~\ref{th:brouwer}, there exists a solution to the fixed-point problem 
	$\Phi  (\vu_h^{n+1},p_h^{n+1}, s_h^{n+1}, c_h^{n+1}  )=0$.  
\end{proof}


Let us denote by $\IN\colon C(\bar{\Omega}) \to \mathcal{M}_h$ the  nodal interpolator with respect to a unisolvent set of Lagrangian interpolation nodes associated with 
$\mathcal{M}_h$. $\IBDM \vu$ will  
denote the BDM projection of~$\vu$, and $\IL p$~is the $L^2$-projection of~$p$ onto~$\mathcal{Q}_h$. Under usual assumptions, the following approximation properties  hold (see \cite{konno2011}):  
\begin{align}   \label{eq:interp1}
&	\norm{\vu - \IBDM \vu}_{\Thnorm} \leq C^* h^{k+1} \norm{\vu}_{k+1,\Omega}, \quad  	\norm{c - \IN c}_{1,\Omega} \leq C^* h^k \norm{c}_{k+1,\Omega}, \quad  
\norm{p - \IL p}_{0,\Omega} \leq C^* h^k \norm{p}_{k,\Omega}. 
\end{align}

The following development follows the structure adopted in \cite{aldbaissy2018}.

\begin{lemma} \label{lemma:consistency}
	Assume that $\vu \in \vH^2(\Omega)$, $p\in L^2(\Omega)$ and $s,c \in H^1(\Omega)$. Then we have 
	\begin{subequations} 
 \label{lemm2.2eqns} \begin{align}
(\partial_t \vu, \vv)_{\Omega} + a^h_1(c;\vu,\vv) + c^h_1(\vu;\vu,\vv) + b(\vv,p) & = F(\vv) \quad \text{\em for all $\vv\in \vV_h$,} \label{lemm2.2eqnsa}\\
b(\vu,q) & = 0 \quad\text{\em for all $q \in \mathcal{Q}_h$,}  \label{lemm2.2eqnsb} \\ 
(\partial_t s, \varphi)_\Omega + a_2(s,\varphi) + c_2(\vu;s,\varphi) & = 0  \quad   \text{\em for all $\varphi\in \mathcal{M}_h$,} \label{lemm2.2eqnsc} \\
(\partial_t c, \psi)_\Omega +\frac{1}{\tau} a_2(c,\psi) + c_2(\vu-v_\mathrm{p}\ve_z;c,\psi) & = 0  \quad   \text{\em for all $\psi\in \mathcal{M}_h$.} \label{lemm2.2eqnsd}
	\end{align} \end{subequations} 
\end{lemma}
\begin{proof}
	Since we assume $\vu\in\vH^2(\Omega)$, integration by parts yields the required result. See also \cite{arnold2001}. The third and fourth equations are a straightforward result from the continuous form.
\end{proof}

Now we decompose the errors as follows: 
\begin{align*}
\vu_h - \vu  &= \eIu + \ePu = (\Iu - \vu) + (\vu_h - \Iu), \quad 
p_h - p = \eIp + \ePp = (\Ip - p) + (p_h - \Ip), \\
s_h - s  &= \eIs + \ePs = (\Is-s) + (s_h - \Is),\quad 
c_h - c  = \eIc + \ePc = (\Ic-c) + (c_h - \Ic).
\end{align*}
Assuming that $\vu_h^0 = \Iu(0)$, $s_h^0 = \Is(0)$ and $c_h^0 = \Ic(0)$, we will use also the notation $\eIu^n = (\vu(t_n) - \Iu(t_n))$ and  $\ePu^n = (\Iu(t_n) - \vu_h^n)$, and similar  notation for other variables.  Since for the first time iteration of system \eqref{eq:discr_var_mod} we adopt a  backward Euler scheme,  an error estimate is required  for this step.

\begin{theorem} \label{thm:ebound_h1}
	Let us assume that 
	\begin{align*} 
	& \vu  \in L^{\infty}(0,\Tend; \vH^{k+1}_0(\Omega) \cap \vC^1(\Th)), \quad 
	\vu' \in L^{\infty}(0,\Tend;\vH^k(\Omega)), \quad 
	 \vu''  \in L^{\infty}(0,\Tend;\vL^2(\Omega)), \\
	& p  \in L^{\infty}(0,\Tend;H^2(\Omega)), \quad 
	s  \in L^{\infty}(0,\Tend; H^{k+1}_{0}(\Omega)), \quad 
	 {s}^{\prime}  \in L^{\infty}(0,\Tend;H^k(\Omega)),  \quad 
	{s}^{ \prime \prime}   \in L^{\infty}(0,\Tend;L^2(\Omega)) , \\ 
	&c   \in L^{\infty}(0,\Tend; H_0^{k+1}(\Omega)), \quad 
	 {c}^{ \prime}   \in L^{\infty}(0,\Tend;H^k(\Omega)), \quad 
	{c}^{ \prime \prime}    \in L^{\infty}(0,\Tend;L^2(\Omega)), 
	\end{align*} 
	and also that $\smash{\norm{\vu}_{L^{\infty}(0,\Tend;\vW^{1,\infty}(\Omega))} < M}$  for a sufficiently small  $M>0$ (a precise condition can be found in Theorem~\ref{thm:errors_u}). Then there exist positive constants $C_u^1$, $C_{s}^1$, $C_c^1$, independently of~$h$ and~$\Dt$, such that
	\begin{align*}
	\frac{1}{4}\norm{\ePu^1}^2_{0,\Omega} + \frac{1}{2}\Dt \tilde{\alpha}_a \norm{\ePu}_{\Thnorm}^2 &\leq C_u^1(h^{2k}+\Dt^4),\\
	\frac{1}{4}\norm{\ePs^1}^2_{0,\Omega} + \frac{1}{2}\Dt \hat{\alpha}_a \norm{\ePs}_{1,\Omega}^2 &\leq C_s^1(h^{2k}+\Dt^4),\\
	\frac{1}{4}\norm{\ePc^1}^2_{0,\Omega} + \frac{1}{2} \Dt \hat{\alpha}_a \norm{\ePc^1}^2_{0,\Omega}  &\leq C_c^1(h^{2k}+\Dt^4). 
	\end{align*}
\end{theorem}
\begin{proof}
	As  in the continuous case, we define the discrete kernel of the bilinear form $b(\cdot,\cdot)$ as 
	\[
	\vX_h \coloneqq \Set{ \vv_h\in \vV_h}{b(\vv_h,q_h) = 0\; \forall q_h \in  \mathcal{Q}_h} = \Set{ \vv_h\in\vV_h}{\sdiv \vv_h = 0 \text{ in  $\Omega$}},
	\]
	and relying on the inf-sup condition \eqref{eq:infsup_disc}, we can 
	continue with an equivalent discrete problem without pressure.
	
	Taking into account the assumed regularity for $\vu$, we  have for all $\vx$, a $\gamma(\vx) \in (0,1)$  such that
	$$\vu(0) = \vu(\Dt) - \Dt\vu'(\Dt) + \frac{1}{2}\Dt^2 \vu''(\Dt \gamma), $$
	then $\vu$ satisfies the following error equation
	\begin{align*}
	\norm{\ePu^1}^2_{0,\Omega}+\Dt \tilde{\alpha}_a \norm{\ePu^{1}}^2_{\Thnorm} &\leq 
	\bigl( \Iu(\Dt)-\vu(\Dt) + \vu_h^0- \vu(0), \ePu^1 \bigr)_{\Omega}\\& \quad + \Dt 
	\bigl(a_1^h(c_h^1;\Iu(\Dt), \ePu^1)-a_1^h(c^1;\vu(\Dt), \ePu^1) \bigr) 		
	\\ & \quad  -\Dt \bigl(c_1^h(\vu_h^1;\vu_h^1,\ePu^1) - c_1^h(\vu(\Dt),\vu(\Dt),\ePu^1) \bigr) \\
	& \quad + \Dt \bigl( F(s_h^1,c_h^1,\ePu^1) - F(s(\Dt),c(\Dt),\ePu^1) \bigr),
	\end{align*}
	which results after choosing $\ePu^1$ as test function in the first equation of the reduced form of Lemma \ref{lemma:consistency} and system \eqref{eq:discr_var_mod}, performing an Euler scheme step, subtracting both equations, and adding $\pm a_1^h(c_h^1; \Iu(\Dt), \ePu^1)$.
	Now, invoking the  approximation estimates \eqref{eq:interp1}, Young's inequality, and the stability properties, we get 
	\begin{align} \label{eq:tmp1_u}	\begin{split}
	\frac{1}{4}\norm{\ePu^1}^2_{0,\Omega}+ \frac{1}{4}\Dt \tilde{\alpha}_a \norm{\ePu^{1}}^2_{\Thnorm} &\leq Ch^{2k} \Dt \Bigl(\norm{\vu(\Dt)}^2_{k+1,\Omega} + \norm{\vu(0)}^2_{k+1,\Omega} + \norm{c(\Dt)}^2_{k+1,\Omega} \Bigr) \\ &\quad 
	 + C\Dt^4(\norm{\vu''}^2_{L^{\infty}(0,\Tend;\vL^2(\Omega))}) + \frac{4 \tilde{C}_{\mathrm{Lip}}^2 M^2}{\tilde{\alpha}_a} \Dt \norm{\ePc^1}^2_{1,\Omega} \\
	 &\quad +\Dt C_f^2 \norm{\ePc^1}^2_{1,\Omega} + \Dt C_f^2 \norm{\ePs^1}^2_{1,\Omega}.
	\end{split}
	\end{align}
	Next, we choose $\ePs^1$ as test function in \eqref{lemm2.2eqnsc} and system \eqref{eq:discr_var_mod}; we follow the same steps as before, adding to the sum of both equations the term $\pm a_2( \Is^1, \ePc^1)$ to obtain
	\begin{align}  \label{eq:tmp1_s} \begin{split} 
	\frac{1}{4}\norm{\ePs^1}^2_{0,\Omega} + \frac{1}{2} \Dt \hat{\alpha}_a \norm{\ePs^1}^2_{1,\Omega} &\leq C\Dt h^{2k} \Bigl(\norm{\vu(\Dt)}^2_{k+1,\Omega} + \norm{s(\Dt)}^2_{k+1,\Omega} + \norm{s(0)}^2_{k+1,\Omega} \Bigr)  \\
	&\quad +C\Dt^4(\norm{s''}^2_{L^{\infty}(0,\Tend;L^2(\Omega))}) + \frac{4 \tilde{C}_c^2}{\hat{\alpha}_a} \Dt \norm{s}^2_{L^{\infty}(0,\Tend;H^1(\Omega))} \norm{\ePu}^2_{\Thnorm} . 
	\end{split} 
	\end{align}
	In the same way, choosing $\ePc^1$ as test function in \eqref{lemm2.2eqnsd}   and in \eqref{eq:discr_var_mod} we obtain
	\begin{align}  \label{eq:tmp1_c} \begin{split} 
	\frac{1}{4}\norm{\ePc^1}^2_{0,\Omega} + \frac{1}{2} \Dt \hat{\alpha}_a \norm{\ePc^1}^2_{1,\Omega} &\leq C\Dt h^{2k} \Bigl(\norm{\vu(\Dt)}^2_{k+1,\Omega} + \norm{c(\Dt)}^2_{k+1,\Omega} + \norm{c(0)}^2_{k+1,\Omega} + v_\mathrm{p}  \Bigr)  \\
	&\quad +C\Dt^4(\norm{c''}^2_{L^{\infty}(0,\Tend;L^2(\Omega))}) + \frac{4 \tilde{C}_c^2}{\hat{\alpha}_a} \Dt \norm{c}^2_{L^{\infty}(0,\Tend;H^1(\Omega))} \norm{\ePu}^2_{\Thnorm} . 
	\end{split} 
	\end{align}
	Now, from \eqref{eq:tmp1_u} we deduce  that
	\begin{align} \label{eq:par1}
	\Dt \norm{\ePu}^2_{\Thnorm} &\leq C(h^{2k} + \Dt^4) + \frac{16 \tilde{C}_{\mathrm{Lip}}^2 M^2}{ \tilde{\alpha}_a^2} \Dt \norm{\ePc^1}^2_{1,\Omega} + \Dt \frac{4C_f^2}{\tilde{\alpha}_a} \norm{\ePc^1}^2_{1,\Omega} + \Dt \frac{4C_f^2}{\tilde{\alpha}_a} \norm{\ePs^1}^2_{1,\Omega}.
	\end{align}
	We insert the previous identity into \eqref{eq:tmp1_c} and consider $M$ and $\Dt$ sufficiently small such that the terms multiplying $\norm{\ePc}^2_{1,\Omega}$, can be absorbed into the left-hand side of the inequality, to get
	\begin{align}
	\frac{1}{4}\norm{\ePc^1}^2_{0,\Omega} + \frac{1}{4}\Dt \hat{\alpha}_a \norm{\ePc}_{1,\Omega}^2 &\leq C_c^1(h^{2k}+\Dt^4) + \Dt \frac{4C_f^2}{\tilde{\alpha}_a} \norm{\ePs^1}^2_{1,\Omega}. \label{eq:res2}
	\end{align}
	Substituting this result back into \eqref{eq:par1} and then into \eqref{eq:tmp1_s}, get us the second estimate. The first estimate follows by directly substituting \eqref{eq:res2} into \eqref{eq:tmp1_u}.  	
\end{proof}

\begin{theorem} \label{thm:errors_u}
	Let $(\vu,p,s,c)$ be the solution of \eqref{eq:main_var_mod} under the assumptions of Section~\ref{sec:wellp}, and $(\vu_h, p_h,s_h,c_h)$ be the solution of \eqref{eq:full_disc_scheme}. Suppose that 
	\begin{align*} 	
	& \vu   \in L^{\infty}(0,\Tend; \vH^{k+1}(\Omega) \cap \vH^1_0(\Omega)), \quad c   \in L^{\infty}(0,\Tend; H^{k+1}(\Omega) \cap H^1_0(\Omega)), \\ 
	&\vu'  \in L^{\infty}(0,\Tend; \vH^{k}(\Omega)),  \quad   \vu^{(3)}  \in L^2(0,\Tend;\vL^2(\Omega))
	\end{align*} 
	and $\norm{\vu}_{L^{\infty}(0,\Tend;\vH^1(\Omega))} < M$  for a sufficiently small constant $M>0$. Then there exist positive constants $C,\gamma_1 \geq 0$ independent of $h$ and $\Dt$ such that for all $m+1 \leq N$, 
	\begin{align*}
	& \norm{\ePu^{m+1}}^2_{0,\Omega} + \norm{2\ePu^{m+1}-\ePu^{m}}^2_{0,\Omega} + \sum_{n=1}^{m}\norm{\Lambda \ePu^n}^2_{0,\Omega} 
	+ \sum_{n=1}^{m} \Dt \tilde{\alpha}_a \norm{\ePu^{n+1}}^2_{\Thnorm}  \leq C(\Dt^4 + h^{2k}) + \sum_{n=1}^{m} \gamma_1 \Dt \norm{\ePc^{n+1}}^2_{1,\Omega}.
	\end{align*}
\end{theorem}

\begin{proof}
	We appeal to the reduced form of the problem again, taking $\vu_h \in \vX_h$ and $\vu\in \vX$, then we choose as test function $\vv_h=\ePu^{n+1}$ in the first equation of \eqref{eq:full_disc_scheme} and insert the terms 
	\begin{align*} 
	 \pm \frac{(\mathcal{D}  \vu(t_{n+1}) , \ePu^{n+1})_{\Omega}}{2 \Delta t}, \quad  
	  \pm \frac{( \mathcal{D} \Iu(t_{n+1}) , \ePu^{n+1})_{\Omega}}{2\Dt},  \quad 
	\pm a_1^h \bigl(c_h^{n+1};\Iu(t_{n+1}), \ePu^{n+1} \bigr)_{\Omega}.  \end{align*} 
	Hence we get
	\begin{align} \label{eq:tmp2_ud} \begin{split} 
	&-\frac{ (\mathcal{D} \ePu^{n+1}, \ePu^{n+1})_{\Omega}}{2\Dt}  - \frac{(\mathcal{D} \eIu^{n+1} , \ePu^{n+1} )_{\Omega}}{2\Dt}
	+ \frac{(\mathcal{D} \vu(t_{n+1}) , \ePu^{n+1})_{\Omega}}{2\Dt}  + a_1^h \bigl(c_h^{n+1};\ePu^{n+1},\ePu^{n+1} \bigr)   \\
	&+ a_1^h  \bigl(c_h^{n+1};\Iu(t_{n+1}),\ePu^{n+1} \bigr)  + c_1^h \bigl(\vu_h^{n+1},\vu_h^{n+1},\ePu^{n+1} \bigr)    = F \bigl(s_h^{n+1},c_h^{n+1},\ePu^{n+1} \bigr).  \end{split} 
	\end{align}
	We consider  \eqref{lemm2.2eqnsa} (see Lemma \ref{lemma:consistency}) at $t=t_{n+1}$ and $\vv = \ePu^{n+1}$.  Inserting the term 
	$\pm  (\mathcal{D} \vu(t_{n+1}), \ePu^{n+1})_{\Omega}/(2\Dt)$,  we readily deduce the 
	identity 
	\begin{align} \label{eq:tmp2_uc} 
	&\frac{( \mathcal{D} \vu(t_{n+1}) , \ePu^{n+1})_{1,\Omega}}{2\Dt} + a_1^h \bigl(c^{n+1};\vu(t_n),\ePu^{n+1} \bigr)   
	+ c_1^h \bigl(\vu(t_{n+1}),\vu(t_{n+1}),\ePu^{n+1} \bigr) 
	= F \bigl(s(t_{n+1}),c(t_{n+1}),\ePu^{n+1} \bigr). 
	\end{align}
	We can then subtract \eqref{eq:tmp2_uc} from \eqref{eq:tmp2_ud} and multiply both sides by $4\Dt$, yielding  $I_1 + I_2 + \dots + I_7 =0$, with 
	\begin{align*}
	&I_1 \coloneqq  2\bigl( \mathcal{D} \ePu^{n+1} , \ePu^{n+1}\bigr)_{\Omega},   \quad 
	I_2 \coloneqq 4\Dt a_1^h(c_h^{n+1};\ePu^{n+1},\ePu^{n+1})_{\Omega} , \quad 
	 I_3 \coloneqq   4\Dt \left(\vu'(t_{n+1}) - \frac{1}{2 \Dt}  \mathcal{D}\vu(t_{n+1}) , \ePu^{n+1} \right)_{\Omega},  \\
	&I_4 \coloneqq  2 \bigl(\mathcal{D} \eIu^{n+1}, \ePu^{n+1}\bigr)_{\Omega}, \quad 
	 I_5 \coloneqq  4\Dt \left(a_1^h \bigl(c^{n+1};\vu^{n+1},\ePu^{n+1}\bigr) - a_1^h \bigl(c_h^{n+1};\Iu^{n+1},\ePu^{n+1}\bigr) \right),  \\ 
	& I_6 \coloneqq 4\Dt \bigl(c_1^h \bigl(\vu(t_{n+1}),\vu(t_{n+1}),\ePu^{n+1}\bigr)-c_1^h(\vu_h^{n+1},\vu_h^{n+1},\ePu^{n+1})\bigr),   \\
	&   I_7 \coloneqq 4\Dt \bigl( F(s(t_{n+1}),c(t_{n+1}),\ePu^{n+1}) - F(s_h^{n+1},c_h^{n+1},\ePu^{n+1})  \bigr).
	\end{align*}
	For the first term, using \eqref{eq:alg_id} we can assert that 
	\begin{align*} 
	I_1&  = \norm{\ePu^{n+1}}^2_{0,\Omega} + \norm{2\ePu^{n+1}-\ePu^{n}}^2_{0,\Omega} + \norm{\Lambda \ePu^{n+1}}^2_{0,\Omega} 
	- \norm{\ePu^{n}}^2_{0,\Omega} - \norm{2 \ePu^n - \ePu^{n-1}}^2_{0,\Omega}. \end{align*} 
	Using the ellipticity stated in \eqref{eq:coerciv_ah}, we readily get
	$$I_2 \geq 4\Dt \tilde{\alpha_a} \norm{\ePu^{n+1}}^2_{\Thnorm}.$$
	By using Taylor's formula with integral remainder we have
	\begin{align*}
	\abs*{\vu'(t_{n+1}) - \frac{1}{2\Dt} \mathcal{D}  \vu(t_{n+1})} &= \frac{\Dt^{3/2}}{2\sqrt{3}} \norm{\vu^{(3)}}_{L^2(t^{n-1},t^{n+1};\vL^2(\Omega))},
	\end{align*}
	then by combining Cauchy-Schwarz and Young's inequality, we obtain the bound 
	\begin{align*}
	\abs{I_3} \leq \frac{\Dt^4}{24 \varepsilon_1} \norm{\vu^{(3)}}^2_{L^2(t_{n-1},t_{n+1};\vL^2(\Omega))} + \frac{\Dt \varepsilon_1}{2} \norm{\ePu^{n+1}}^2_{\Thnorm}.
	\end{align*}
	Now we insert $\pm 4\Dt \eIu'(t_{n+1})$ into the fourth term, which leads to 
	\begin{align*}
	I_4 & = -4\Dt  \bigl(\eIu'(t_{n+1}),\ePu^{n+1} \bigr)_{\Omega}  + \left(\eIu'(t_{n+1}) -  \frac{\mathcal{D} \eIu^{n+1} }{2\Dt}, \ePu^{n+1}\right)_{\Omega}.
	\end{align*}
	Proceeding as before and using \eqref{eq:interp1} on the first term of $I_4$, we get
	\begin{align*}
	\abs{I_4} & \leq \frac{C}{2\varepsilon_2} h^{2k} \norm{\vu'}^2_{\vL^{\infty}(0,\Tend;\vH^{k}(\Omega))} + \frac{\Dt \varepsilon_2}{2} \norm{\ePu^{n+1}}^2_{\Thnorm} \\ & \quad  + \frac{\Dt^4 C}{2 \varepsilon_3} \norm{\vu^{(3)}}^2_{L^2(0,\Tend;\vL^2(\Omega))} + \frac{\Dt \varepsilon_3}{2}\norm{\ePu^{n+1}}_{\Thnorm}^2.
	\end{align*}
	Again we insert $\pm a_1^h(\Ic^{n+1};\vu(t_{n+1}),\ePu^{n+1})$ and  $\pm a_1^h(c_h^{n+1};\vu(t_{n+1}),\ePu^{n+1})$. Then  by \eqref{eq:interp1}, \eqref{eq:a1lipschitz}, Lemma \ref{lemma4:ab_bounds} and Young's inequality we immediately have
	\begin{align*}
	\abs{I_5} &\leq \frac{8 \tilde{C}_{\mathrm{Lip}}^2 M^2 C^*\Dt}{\varepsilon_4} h^{2k} \norm{c}^2_{L^{\infty}(0,\Tend;H^{k+1}(\Omega))} + \frac{\varepsilon_4}{2}\Dt \norm{\ePu^{n+1}}_{\Thnorm} +\frac{8\tilde{C}_{\mathrm{Lip}}^2 M^2 \Dt}{\varepsilon_5} \norm{\ePc^{n+1}}_{1,\Omega}^2 \\
	& \quad + \frac{\varepsilon_5}{2}\Dt \norm{\ePu^{n+1}}_{\Thnorm} + \frac{8 \tilde{C}_{a}^2 C^*\Dt}{\varepsilon_6} h^{2k} \norm{c}^2_{L^{\infty}(0,\Tend;H^{k+1}(\Omega))} + \frac{\varepsilon_6}{2}\Dt \norm{\ePu^{n+1}}_{\Thnorm}.
	\end{align*}
	Adding and subtracting suitable terms within~$I_6$ 
	yields 
	\begin{align*} 
	I_6 = \tilde{I}_6 - 4 \delta t c_1^h \bigl(\vu_h^{n+1},\ePu^{n+1},\ePu^{n+1} \bigr), 
	\end{align*} 
	where we define  
	\begin{align*}
	\tilde{I}_6 \coloneqq  -4\Dt\bigl( & c_1^h(\vu(t_{n+1}),\Iu(t_{n+1}),\ePu^{n+1}) - c_1^h(\Iu(t_{n+1}),\Iu(t_{n+1}),\ePu^{n+1})   \\
	&+  c_1^h(\Iu(t_{n+1}),\Iu(t_{n+1}),\ePu^{n+1})- c_1^h(\Iu(t_{n+1}),\vu(t_{n+1}),\ePu^{n+1})  \\
	&\ + c_1^h(\Iu(t_{n+1}),\vu(t_{n+1}),\ePu^{n+1}) - c_1^h(\vu(t_{n+1}),\vu(t_{n+1}), \ePu^{n+1}) 
	\bigr).
	\end{align*}
	The bound \eqref{eq:bound_ch1} and \eqref{eq:interp1} imply that 
	\begin{align*}
	\abs{\tilde{I}_6}  & \leq 4\Dt \tilde{C}_c \bigl( \norm{\ePu^{n+1}}^2_{\Thnorm} \norm{\Iu(t_{n+1})}_{\Thnorm} +
	\norm{\Iu(t_{n+1})}_{\Thnorm} \norm{\eIu^{n+1}}_{\Thnorm} \norm{\ePu^{n+1}}_{\Thnorm} \\& \qquad + \norm{\eIu^{n+1}}_{\Thnorm}\norm{\vu(t_{n+1})}_{\Thnorm}\norm{\ePu^{n+1}}_{\Thnorm}\bigr)  \\
	& \leq 4\Dt \biggl( \tilde{C}_c C^{*} \norm{\ePu^{n+1}}^2_{\Thnorm} \norm{\vu}^2_{L^{\infty}(0, \Tend; \vH^1(\Omega))} \\
	& \qquad  +
	\frac{2 h^{2k}C\tilde{C}_c^2}{\varepsilon_7}\norm{\vu}^2_{L^{\infty}(0,T;\vH^1(\Omega))} \norm{\vu}^2_{L^{\infty}(0,\Tend;\vH_1^{k+1}(\Omega))}   + \frac{\varepsilon_7}{8} \norm{\ePu^{n+1}}^2_{\Thnorm} \\ & \qquad  +\frac{2 Ch^{2k} \tilde{C}_c^2}{\varepsilon_8}\norm{\vu}^2_{L^{\infty}(0,\Tend;\vH^{k+1}(\Omega))} \norm{\vu}^2_{L^{\infty}(0,\Tend;\vH^1(\Omega))}  
	+ \frac{\varepsilon_8}{8} \norm{\ePu^{n+1}}^2_{\Thnorm}\biggr) \\
	& \leq 4\Dt \biggl( C^{*} \tilde{C}_c M \norm{\ePu^{n+1}}^2_{\Thnorm} +
	\frac{h^{2k}C}{2\varepsilon_7}\norm{\vu}^2_{L^{\infty}(0,T;\vH^1(\Omega))} \norm{\vu}^2_{L^{\infty}(0,\Tend;\vH^{k+1}(\Omega))}  \\&\qquad\qquad + \frac{\varepsilon_7}{8} \norm{\ePu^{n+1}}^2_{\Thnorm} + \frac{2Ch^{2k}}{\varepsilon_8}\norm{\vu}^2_{L^{\infty}(0,\Tend;\vH^{k+1}(\Omega))} \norm{\vu}^2_{L^{\infty}(0,\Tend;\vH^1(\Omega))}  + \frac{\varepsilon_8}{8} \norm{\ePu^{n+1}}^2_{\Thnorm}\biggr),
	\end{align*}
	where $C^*$ is a positive constant coming from \eqref{eq:interp1}.
	We also have
	\begin{align*}
	\abs{\tilde{I_7}} & \leq 4\Dt \biggl( \frac{2 C_f^2}{\varepsilon_9} (\norm{\ePs^{n+1}}_{0,\Omega}^2 + \norm{\eIs^{n+1}}_{0,\Omega}^2) + \frac{2 C_f^2}{\varepsilon_{10}} (\norm{\ePc^{n+1}}_{0,\Omega}^2 + \norm{\eIc^{n+1}}_{0,\Omega}^2) \\
	& \qquad +\frac{\varepsilon_9}{8} \norm{\ePu^{n+1}}_{1,\Th}^2 + \frac{\varepsilon_{10}}{8} \norm{\ePu^{n+1}}_{1,\Th}^2 \biggr) \\
	&\leq 4\Dt \biggl( \frac{2 C_f^2}{\varepsilon_9} (\norm{\ePs^{n+1}}_{0,\Omega}^2 + C^* h^{2k} \norm{s}_{L^{\infty}(0,\Tend,H^{k+1}(\Omega)}^2) + \frac{2 C_f^2}{\varepsilon_{10}} (\norm{\ePc^{n+1}}_{0,\Omega}^2 + + C^* h^{2k} \norm{c}_{L^{\infty}(0,\Tend,H^{k+1}(\Omega)}^2)) \\
	& \qquad +\frac{\varepsilon_9}{8} \norm{\ePu^{n+1}}_{1,\Th}^2 + \frac{\varepsilon_{10}}{8} \norm{\ePu^{n+1}}_{1,\Th}^2 \biggr).
	\end{align*}
	Hence, by choosing $\varepsilon_i= 3\tilde{\alpha}_a / 5$ for $i=1, \dots, 10$, collecting the above estimates, and summing over $1\leq n \leq m$ for all $m+1\leq N$, we get
	\begin{align*}
	& \norm{\ePu^{m+1}}^2_{0,\Omega} + \norm{2\ePu^{m+1}-\ePu^{m}}^2_{0,\Omega} + \sum_{n=1}^{m} \norm{\Lambda \ePu^{n}}^2_{0,\Omega} - 3\norm{\ePu^1}^2_{0,\Omega} \\ & + \sum_{n=1}^{m} \Dt \tilde{\alpha}_a \norm{\ePu^{n+1}}^2_{\Thnorm}  \leq C (\Dt^4 + h^{2k}) +  \frac{16 M^2 C_{\mathrm{Lip}^2} \Dt}{\tilde{\alpha}_a}\sum_{n=1}^{m}  \norm{\ePc^{n+1}}^2_{1,\Omega}, 
	\end{align*}
	where 
	\begin{align*} 
	 M \leq \min \biggl\{\frac{\tilde{\alpha}_a}{4 \tilde{C}_c C^{*}}, \frac{\sqrt{\hat{\alpha}_a}}{4 C_{\mathrm{Lip}} \sqrt{2}} \biggr\} ,  
	  \quad \gamma_1 = \frac{16 M^2 C_{\mathrm{Lip}^2}}{\tilde{\alpha}_a} \leq  \frac{\hat{\alpha}_a\tilde{\alpha}_a}{2}. 
	  \end{align*} 
	  Finally, Theorem~\ref{thm:ebound_h1} yields the desired result.
\end{proof}


\begin{theorem} \label{thm:errors_s}
	Let $(\vu,p,s,c)$ be the solution of \eqref{eq:main_var_mod} under the assumptions of Section~\ref{sec:wellp}, and $(\vu_h, p_h,s_h,c_h)$ be the solution of \eqref{eq:full_disc_scheme}. If 
	\begin{align*} 	
	& \vu \in L^{\infty}(0,\Tend; \vH_1^{k+1}(\Omega) \cap \vH^1_0(\Omega)),  \quad  s \in L^{\infty}(0,\Tend; H^{k+1}(\Omega) \cap H^1_0(\Omega)),  \\  & s^{ \prime} \in L^{\infty}(0,\Tend; H^{k}(\Omega)), \quad s^{(3)} \in L^2(0,\Tend;L^2(\Omega)),  
	\end{align*} 
	then there exist  constants $C, \gamma_2 >0$, independent of $h$ and $\Dt$, such that for all $m+1 \leq N$
	\begin{align*}
	&\norm{\ePs^{m+1}}^2_{0,\Omega} + \norm{2\ePs^{m+1}-\ePs^{m}}^2_{0,\Omega} + \sum_{n=1}^{m}\norm{\Lambda \ePs^{n+1}}^2_{0,\Omega} 
	+ \sum_{n=1}^{m} \Dt \hat{\alpha}_a \norm{\ePs^{n+1}}^2_{1,\Omega} \\
	&\qquad \qquad \leq C(\Dt^4 + h^{2k}) + \sum_{n=1}^{m} \gamma_2 \Dt \norm{\ePu^{n+1}}^2_{\Thnorm}.
	\end{align*}
\end{theorem}
\begin{proof}
	Proceeding similarly as for Theorem~\ref{thm:errors_u}, we choose as test function $\smash{\varphi_h=\ePs^{n+1}}$ in the second equation of \eqref{eq:full_disc_scheme} and insert suitable additional terms 	
	to obtain the following identity (analogous to \eqref{eq:tmp2_ud}) 
	\begin{align}\label{eq:tmp2_Td} \begin{split} 
	&-\frac{(\mathcal{D} \ePs^{n+1}, \ePs^{n+1})_{\Omega} }{2\Dt}- \frac{ (\mathcal{D} \eIs^{n+1} , \ePs^{n+1})_{\Omega}}{2\Dt} - a_2(\ePs^{n+1},\ePs^{n+1})  + \frac{( \mathcal{D} s(t_{n+1}) , \ePs^{n+1})_{\Omega}}{2\Dt}  \\ & + a_2(\Is(t_{n+1}),\ePs^{n+1}) + c_2(\vu_h^{n+1},s_h^{n+1},\ePs^{n+1}) = 0. \end{split} 
	\end{align} 
	From \eqref{lemm2.2eqnsb}, focusing 
	on $t=t_{n+1}$, using $\smash{\varphi = \ePs^{n+1}}$ and proceeding as in 
	the derivation of \eqref{eq:tmp2_uc}, we obtain 
	\begin{align} \label{eq:tmp2_Tc}  
	\frac{( \mathcal{D} s(t_{n+1}), \ePs^{n+1} )_{\Omega}}{2\Dt} + a_2
	\bigl(s(t_{n+1}),\ePs^{n+1} \bigr)   	+ c_2 \bigl(\vu(t_{n+1}),s(t_{n+1}),\ePs^{n+1} \bigr)   
	  = - \biggl(s'(t_{n+1}) - \frac{ \mathcal{D} s(t_{n+1}) }{2\Dt}, \ePs^{n+1} \biggr)_{\Omega}. 
	\end{align}
	Next we  subtract \eqref{eq:tmp2_Tc} from \eqref{eq:tmp2_Td}  and   multiply both sides by $4\Dt$. This yields 
	 $\hat{I}_1 + \dots + \hat{I}_6 =0$, where 
	\begin{align*} 
	& \hat{I}_1 \coloneqq 2 ( \mathcal{D} \ePs^{n+1} , \ePs^{n+1} )_{\Omega}, 
	\quad \hat{I}_2 \coloneqq   4\Dt a_2  (\ePs^{n+1},\ePs^{n+1} ), \quad 
	  \hat{I}_3 \coloneqq  4\Dt \biggl(s'(t_{n+1}) - \frac{ \mathcal{D} s(t_{n+1}) }{2\Dt}, \ePs^{n+1} \biggr)_{\Omega}, \\ &  \hat{I}_4 \coloneqq  2 \bigl( \mathcal{D} \eIs^{n+1}, \ePs^{n+1}\bigr)_{\Omega}, \quad 
	\hat{I}_5 \coloneqq 4\Dt a_2 \bigl(\eIs^{n+1},\ePs^{n+1} \bigr),  \\ & 
	\hat{I}_6 \coloneqq 4\Dt \bigl(c_2 \bigl(\vu_h^{n+1},s_h^{n+1},\ePu^{n+1} \bigr)-c_2 \bigl(\vu(t_{n+1}),s(t_{n+1}),\ePu^{n+1} \bigr) \bigr). 
	\end{align*}
	For $\hat{I}_1 $, $\hat{I}_2$ and $\hat{I}_3$  we use \eqref{eq:alg_id}, \eqref{eq:coerciv_am}, and Taylor expansion along with Young's inequality, respectively, to obtain
	\begin{align*}
	\hat{I}_1 & = \norm{\ePs^{n+1}}^2_{0,\Omega} + \norm{2\ePs^{n+1}-\ePs^{n}}^2_{0,\Omega} + \norm{\Lambda \ePs^{n+1}}^2_{0,\Omega} 
	- \norm{\ePs^{n}}^2_{0,\Omega} - \norm{2 \ePs^n - \ePs^{n-1}}^2_{0,\Omega},\\
	\hat{I}_2 & \geq 4\Dt \hat{\alpha_a} \norm{\ePs^{n+1}}^2_{1,\Omega}, \qquad 
	\abs{\hat{I}_3}  \leq \frac{\Dt^4}{24 \varepsilon_1} \norm{s^{(3)}}^2_{L^2(t_{n-1},t_{n+1};L^2(\Omega))} + \frac{\Dt \varepsilon_1}{2} \norm{\ePs^{n+1}}^2_{1,\Omega}.
	\end{align*}
	Inserting $\pm 4\Dt \eIs'(t_{n+1})$ into $\hat{I}_4$  and using \eqref{eq:interp1} leads to the bound 
	\begin{align*}
	\abs{\hat{I}_4} &\leq \frac{C}{2\varepsilon_2} h^{2k} \norm{s^{ \prime}}^2_{L^{\infty}(0,\Tend;H^{k}(\Omega))} + \frac{\Dt \varepsilon_2}{2} \norm{\ePs^{n+1}}^2_{1,\Omega} 
	+ \frac{\Dt^4 C}{2 \varepsilon_3} \norm{s^{(3)}}^2_{L^2(0,\Tend;L^2(\Omega))} + \frac{\Dt \varepsilon_3}{2}\norm{\ePs^{n+1}}_{1,\Omega}^2. 
	\end{align*}
	Employing again \eqref{eq:interp1} in combination with \eqref{eq:cont_a2} we have
	\begin{align*}
	\abs{\hat{I}_5} &\leq \frac{8\hat{C}_a^2 \Dt }{\varepsilon_4} h^{2k} \norm{s}^2_{L^{\infty}(0,\Tend;H^{k+1}(\Omega))}+\frac{\Dt \varepsilon_4}{2}\norm{\ePs^{n+1}}^2_{1,\Omega}.
	\end{align*}
	In order to derive a bound for  
	$\hat{I}_6$  we proceed as for the bound
	on $I_7$ in the proof of Theorem~\ref{thm:errors_u}; namely 
	adding and subtracting suitable terms 
	in the definition of~$\hat{I}_6$,  
	defining $\tilde{I}_6$ in this case by 
	\begin{align*} 
	\hat{I}_6 = \tilde{I}_6 + 4 \delta_t c_2 \bigl(\vu_h^{n+1},\ePs^{n+1},\ePs^{n+1} \bigr), 
	\end{align*} 
	and applying \eqref{eq:cTSeq0}, \eqref{eq:bound_Cm0}, \eqref{eq:interp1} and Young's Inequality to the result, we get
	\begin{align*}
	\abs{\tilde{I}_6} 
	\leq & 4\Dt \biggl( \frac{2 \tilde{C}_c^2 C^{*2}}{\varepsilon_5} \norm{\ePu^{n+1}}^2_{\Thnorm} \norm{s}^2_{L^{\infty}(0, \Tend; H^1(\Omega))} + \frac{1}{8\varepsilon_5} \norm{\ePs}^2_{1,\Omega} \\
	&\qquad  + \frac{2h^{2k}\tilde{C}_c^2C^{*2}}{\varepsilon_6}\norm{\vu}^2_{L^{\infty}(0,T;H^1(\Omega))} \norm{s}^2_{L^{\infty}(0,\Tend;H^{k+1}(\Omega))} + \frac{\varepsilon_6}{8} \norm{\ePs^{n+1}}^2_{1,\Omega} \\
	&\qquad + \frac{2 h^{2k} \tilde{C}^2 C^{*2}}{\varepsilon_7}\norm{\vu}^2_{L^{\infty}(0,\Tend;\vH^{k+1}(\Omega))} \norm{s}^2_{L^{\infty}(0,\Tend;H^1(\Omega))} + \frac{\varepsilon_7}{8} \norm{\ePs^{n+1}}^2_{1,\Omega}\biggr).
	\end{align*}
	In this manner, and after choosing $\varepsilon_i= 6 \hat{\alpha}_a /7$ for $i=1, \dots, 7$, we can collect the above estimates and sum over $1\leq n \leq m$, for all $m+1\leq N$, to get 
	\begin{align*}
	& \norm{\ePs^{m+1}}^2_{0,\Omega} + \norm{2\ePs^{m+1}-\ePs^{m}}^2_{0,\Omega} + \sum_{n=1}^{m}\norm{\Lambda \ePs^{n}}^2_{0,\Omega} 
	+ \sum_{n=1}^{m} \Dt \hat{\alpha}_a \norm{\ePs^{n+1}}^2_{1,\Omega}  -3  \norm{\ePs^{1}}^2_{0,\Omega} \\
	& \leq C(\Dt^4 + h^{2k})  +  \gamma_2 \Dt \sum_{n=1}^{m} \norm{\ePu^{n+1}}^2_{\Thnorm}, \quad \text{where} \quad \gamma_2 \coloneqq \frac{28 \tilde{C}_c^2 C^{*2}}{3\hat{\alpha}_a} \norm{s}^2_{L^{\infty}(0,\Tend; H^1(\Omega))}. 
	\end{align*}
	 This concludes the proof. \end{proof}

\begin{theorem} \label{thm:errors_c}
	Let $(\vu,p,s,c)$ be the solution of \eqref{eq:main_var_mod} under the assumptions of Section~\ref{sec:wellp}, and $(\vu_h, p_h,s_h,c_h)$ be the solution of \eqref{eq:full_disc_scheme}. If 
	\begin{align*} 	
	& \vu \in L^{\infty}(0,\Tend; \vH^{k+1}(\Omega) \cap \vH^1_0(\Omega)),  \quad  c \in L^{\infty}(0,\Tend; H^{k+1}(\Omega) \cap H^1_0(\Omega)),  \\  & c^{ \prime} \in L^{\infty}(0,\Tend; H^{k}(\Omega)), \quad c^{(3)} \in L^2(0,\Tend;L^2(\Omega)),  
	\end{align*} 
	then there exist  constants $C, \gamma_3 >0$ that are  independent of $h$ and $\Dt$, such that for all $m+1 \leq N$
	\begin{align*}
	&\norm{\ePc^{m+1}}^2_{0,\Omega} + \norm{2\ePc^{m+1}-\ePc^{m}}^2_{0,\Omega} + \sum_{n=1}^{m}\norm{\Lambda \ePc^{n+1}}^2_{0,\Omega} 
	+ \sum_{n=1}^{m} \Dt \hat{\alpha}_a \norm{\ePc^{n+1}}^2_{1,\Omega} \\
	&\qquad \qquad \leq C(\Dt^4 + h^{2k}) + \sum_{n=1}^{m} \gamma_3 \Dt \norm{\ePu^{n+1}}^2_{\Thnorm}.
	\end{align*}
\end{theorem}
\begin{proof}
	It follows along the same lines of the proof of Theorem \ref{thm:errors_s}, with constant $\gamma_3$ given by
	\[
	\gamma_3 = \frac{28 \tilde{C}_c^2 C^{*2}}{3\hat{\alpha}_a} \norm{c}^2_{L^{\infty}(0,\Tend; H^1(\Omega))}.
	\]
\end{proof}

\begin{theorem} \label{thm:errors_ucs}
	Under the same assumptions of Theorems \ref{thm:errors_u}-\ref{thm:errors_c}, there exist positive constants $\hat{\gamma}_u$, $\gamma_{s}$ and $\gamma_c$ independent of $\Dt$ and $h$, such that, for a sufficiently small $\Dt$ and all $m + 1 \leq N$, there hold
	\begin{align*}
	& \Biggl(\norm{\ePu^{m+1}}^2_{0,\Omega} + \norm{2\ePu^{m+1}-\ePu^{m}}^2_{0,\Omega} + \sum_{n=1}^{m} \bigl( \norm{\Lambda \ePu^n}^2_{0,\Omega} 
	+  \Dt \tilde{\alpha}_a \norm{\ePu^{n+1}}^2_{\Thnorm} \bigr) \Biggr)^{1/2} \leq \gamma_u(\Dt^2 + h^{k}), \\
	& \Biggl(\norm{\ePs^{m+1}}^2_{0,\Omega} + \norm{2\ePs^{m+1}-\ePs^{m}}^2_{0,\Omega} + \sum_{n=1}^{m}  \bigl( \norm{\Lambda \ePs^n}^2_{0,\Omega} 
	+   \Dt \hat{\alpha}_a \norm{\ePs^{n+1}}^2_{1,\Omega}\bigr)  \Biggr)^{1/2}  \leq \gamma_s(\Dt^2 + h^{k}), \\
	& \Biggl(\norm{\ePc^{m+1}}^2_{0,\Omega} + \norm{2\ePc^{m+1}-\ePs^{m}}^2_{0,\Omega} + \sum_{n=1}^{m}  \bigl( \norm{\Lambda \ePc^n}^2_{0,\Omega} 
	+   \Dt \hat{\alpha}_a \norm{\ePc^{n+1}}^2_{1,\Omega}\bigr)  \Biggr)^{1/2}  \leq \gamma_c(\Dt^2 + h^{k}).
	\end{align*}
\end{theorem}

\begin{proof}
	From Theorems \ref{thm:errors_u} and \ref{thm:errors_c}, since $\gamma_1 \leq  \frac{\hat{\alpha}_a\tilde{\alpha}_a}{2}$ we have the estimate 
	\[
	\sum_{n=1}^{m} \Dt \norm{\ePu^{n+1}}^2_{\Thnorm} \leq C(\Dt^4+h^{2k}) + \sum_{n=1}^{m} \Dt \frac{\hat{\alpha}_a}{2}\norm{\ePc^{n+1}}^2_{\Thnorm},
	\]
    which, after substitution back into Theorem~\ref{thm:errors_c}, yields
    \begin{equation} \label{eq:tmp1}
    \sum_{n=1}^{m} \Dt \norm{\ePc^{n+1}}^2_{1,\Omega} \leq C(\Dt^4+h^{2k}). 
    \end{equation}
	The first bound follow by combining \eqref{eq:tmp1} and Theorem \ref{thm:errors_u}. The second and third bounds follow directly from the first bound and Theorems \ref{thm:errors_s} and \ref{thm:errors_c}.
\end{proof}

\begin{lemma}
	Under the same assumptions of Theorem \ref{thm:errors_ucs}, we have
	\begin{align*}
	\left(\sum_{n=1}^{m} \Dt \norm{p(t_{n+1})-p_h^{n+1}}^2_{0,\Omega}\right)^{1/2} \leq \hat{\gamma}_p (\Dt^2 + h^{k}).
	\end{align*}
\end{lemma}
\begin{proof}
	Owing to the inf-sup condition \eqref{eq:infsup_disc}, there exists  $\vw_h \in \vX_h^{\perp}$ such that
	\begin{align}
	b \bigl(\vw_h, p(t_{n+1})-p_h^{n+1} \bigr) &= \norm{p(t_{n+1})-p_h^{n+1}}^2_{0,\Omega}, \quad  
	\norm{\vw_h}_{\Thnorm} \leq \frac{1}{\tilde{\beta}} \norm{p(t_{n+1})-p_h^{n+1}}_{0,\Omega}. \label{eq:bp2}
	\end{align}
	From \eqref{eq:full_disc_scheme} and Lemma \ref{lemma:consistency}, proceeding as in the proof of Theorem \ref{thm:errors_u}, we obtain
	\begin{align*}
	& \Dt b(\vw_h,p(t_{n+1})-p_h^{n+1}) \\
	&= -\Dt\left(\vu'(t_{n+1}) - \frac{1}{2\Delta t} \mathcal{D} \vu_h^{n+1},\vw_h\right)_{\Omega}  + \Dt \left(a_1^h(c_h^{n+1};\vu_h^{n+1}, \vw_h) - a_1^h(c(t_{n+1});\vu(t_{n+1}), \vw_h) \right) \\& \quad + \Dt \bigl(c_1^h(\vu_h^{n+1};\vu_h^{n+1},\vw_h) - c_1^h(\vu(t_{n+1});\vu(t_{n+1}),\vw_h) \bigr) 	\\
	& \quad + \Dt \bigl( F(s_h^{n+1},c_h^{n+1},\vw_h) - F(s(t_{n+1}),c(t_{n+1}),\vw_h) \bigr)\\
	&\leq \frac{\Dt^2}{2\sqrt{3}} \norm{\vu^{(3)}}_{L^2(t_{n-1},t_{n+1},\vL^2(\Omega))} \sqrt{\Dt} \norm{\vw_h}_{\Thnorm} + \Dt  C_f \norm{\ePs^{n+1}}_{0,\Omega} \norm{\vw_h}_{\Thnorm}
	\\ & \quad  + \tilde{C}_a C^* h^k \Dt \norm{\vu}_{L^{\infty}(0,\Tend;\vH^{k+1}(\Omega))} \norm{\vw_h}_{\Thnorm} 
	+ \tilde{C}_{\mathrm{lip}} M \Dt \norm{\ePc^{n+1}}_{1,\Omega} \norm{\vw_h}_{\Thnorm} \\
	& \quad +  \tilde{C}_{\mathrm{lip}} \Dt M \norm{\ePc^{n+1}}_{1,\Omega} \norm{\vw_h}_{\Thnorm}  + \Dt C \tilde{C}_c C^* C_{\infty} M \norm{\ePc^{n+1}}_{1,\Omega}  \norm{\vw_h}_{\Thnorm}\\ 
	& \quad + 2\Dt C^* \tilde{C}_c h^k \norm{\vu}_{L^{\infty}(0,\Tend;\vH^1(\Omega))} \norm{\vu}_{L^{\infty}(0,\Tend,\vH^{k+1}(\Omega))} \norm{\vw_h}_{\Thnorm}  + \Dt C_f \norm{\ePc^{n+1}}_{0,\Omega} \norm{\vw_h}_{\Thnorm}\\
	& \quad + \Dt C^* C_f h^k \norm{s}_{L^{\infty}(0,\Tend;H^1(\Omega))} \norm{\vw_h}_{\Thnorm} + \Dt C^* C_f h^k \norm{c}_{L^{\infty}(0,\Tend;H^1(\Omega))} \norm{\vw_h}_{\Thnorm}. 
	\end{align*}
	Summing over $1\leq n \leq m$ for all $m+1 \leq N$ and substituting back into  \eqref{eq:bp2}, we obtain
	\begin{align*}
 \left(\sum_{n=1}^{m} \Dt \norm{p(t_{n+1})-p_h^{n+1}}^2_{0,\Omega}\right)^{1/2}\!\!
  \leq \frac{C}{\tilde{\beta}}\Biggl( \Dt^2 + h^{k} + \Biggl( \sum_{n=1}^{m} \Dt \norm{\ePc^{n+1}}^2_{0,\Omega} \Biggr)^{1/2}   + \Biggl( \sum_{n=1}^{m} \Dt \norm{\ePu^{n+1}}^2_{\Thnorm} \Biggr)^{1/2}  \Biggr),
	\end{align*}
	and the desired result readily follows from Theorem \ref{thm:errors_ucs}.
\end{proof}

\bigskip 
We next proceed to derive and analyse {\it a posteriori} error estimators. We split the presentation into three cases of increasing complexity, starting with an estimator focusing on the steady coupled problem.

\section{{\it A posteriori} error estimation for the stationary problem}
\label{sec:aposteriori1}
Let us define the following nonlinear coupled problem in weak form, 
associated with the stationary version of the 
model equations. Find $({\bu}, {p},{s},{c})\in \bH^1_0(\Omega)\times L^2_0(\Omega)\times H^1_0(\Omega)\times H^1_0(\Omega)$ such that
\begin{subequations}\label{elliptic_cons111sta}
\begin{align}
a_1(c,{\bu},\bv)+c_1({\bu};{\bu},\bv)+b(\bv,{p})&=({\ff},\bv)_{0,\Omega}\quad\forall \bv\in\bH^1_0(\Omega)\\
b({\bu},q)&=0\quad \forall q\in L^2_0(\Omega),\\
a_2({s},\phi)+c_2(\bu;{s},\phi)&=(f_1,\phi)_{0,\Omega}\quad \forall\phi\in H^1_0(\Omega),\\
\frac{1}{\tau}a_2({c},\psi)+c_2(\bu-v_p\boldsymbol{e}_z;{c},\psi)&=(f_2,\psi)_{0,\Omega}\quad \forall\psi\in H^1_0(\Omega),
\end{align}
\end{subequations}
where  $\ff=  (\rho / \rho_{\mathrm{m}})\vg=(\alpha s+\beta c)\vg$, $f_1=0$, and $f_2=0$.
Let us also consider its discrete counterpart: Find $({\bu}_h, {p}_h,s_h,c_h)\in \bV_h\times \mathcal{Q}_h\times \mathcal{M}_h\times \mathcal{M}_h$ such that
\begin{subequations}\label{elliptic_cons112sta}
\begin{align}
a_1(c_h,{\bu}_h,\bv)+c_1({\bu}_h;{\bu}_h,\bv)+b(\bv,{p})&=({\ff}_h,\bv)_{0,\Omega}\quad\forall \bv\in \bV_h\\
b({\bu}_h,q)&=0\quad \forall q\in \mathcal{Q}_h,\\
a_2({s}_h,\phi)+c_2(\bu_h;{s}_h,\phi)&=(f_1,\phi)_{0,\Omega}\quad \forall\phi\in \mathcal{M}_h\\
\frac{1}{\tau}a_2({c}_h,\psi)+c_2(\bu_h-v_p\boldsymbol{e}_z;{c}_h,\psi)&=(f_2,\psi)_{0,\Omega}\quad \forall\psi\in \mathcal{M}_h,
\end{align}
\end{subequations}
where $\ff_h=(\alpha s_h+\beta c_h)\vg$, $f_1=0$, and $f_2=0$.

For each $K\in\mathcal{T}_h$ and each $e\in \mathcal{E}_h$ 
we define element-wise and edge-wise residuals  as follows:
\begin{align*}
& \textbf{R}_K \coloneqq \{\ff_h+\nabla\cdot(\nu(c_h)\bnabla\bu_h)-\bu_h\cdot\bnabla\bu_h-(\rho_{\mathrm{m}})^{-1}\nabla p_h\}|_K, \\
& R_{1,K} \coloneqq  \{f_{1}+ \mathrm{Sc}^{-1}\nabla^2 s_h-\bu_h\cdot\nabla s_h\}|_K,\quad 
R_{2,K} \coloneqq \{f_{2}+ (\tau \mathrm{Sc})^{-1}\nabla^2 c_h-(\bu_h-v_p\ve_z)\cdot\nabla c_h\}|_K,\nonumber\\
& \label{eq:RK-Re} \textbf{R}_e \coloneqq \begin{cases}
\frac{1}{2} \djump{((\rho_{\mathrm{m}})^{-1}p_h\vI -\nu(c_h)\bnabla \bu_h)\vn} & \text{for $e\in\mathcal{E}_h\setminus \Gamma$,} \\
0 & \text{for $ e\in \Gamma$,} 
\end{cases}\\
& R_{1,e} \coloneqq \begin{cases}
\frac{1}{2} \djump{(\mathrm{Sc}^{-1}\nabla s_h)\cdot\vn} & \text{for $e\in\mathcal{E}_h\setminus \Gamma$,} \\
0 & \text{for $e\in \Gamma$,} 
\end{cases}, \qquad 
R_{2,e} \coloneqq  \begin{cases}
\frac{1}{2} \djump{((\tau \mathrm{Sc})^{-1}\nabla c_h)\cdot\vn} &  \text{for $e\in\mathcal{E}_h\setminus \Gamma$,} \\
0 & \text{for $e\in \Gamma$.} 
\end{cases} \nonumber
\end{align*}
Then we introduce the element-wise error estimator $\Psi_K^2=\Psi_{R_K}^2+\Psi_{e_K}^2+\Psi_{J_K}^2$ with contributions defined as 
\begin{gather*}
\Psi_{R_K}^2 \coloneqq h_K^2 \bigl(\|\textbf{R}_K\|^2_{0,K}+\|R_{1,K}\|_{0,K}^2+\|R_{2,K}\|_{0,K}^2 \bigr),\\
\Psi_{e_K}^2 \coloneqq \sum_{e\in\partial K}h_e \bigl(\|\textbf{R}_e\|_{0,e}^2+\|R_{1,e}\|^2_{0,e}+\|R_{2,e}\|^2_{0,e}\bigr),\quad 
\Psi_{J_K}^2 \coloneqq \sum_{e\in\partial K}h_e^{-1}\|\djump{\bu_h}\|_{0,e}^2,
\end{gather*}
so a global {\it a posteriori} error estimator for the nonlinear coupled and steady problem \eqref{elliptic_cons112sta} is 
\begin{align}\label{eststa}
\Psi= \left(\sum_{K\in\mathcal{T}_h}\Psi_K^2 \right)^{1/2}. 
\end{align}

\subsection{Reliability}
Let us introduce the space 
$$ \tilde{X}(\mathcal{T}_h)=\{\bv\in\textbf{H}_0(\bdiv,\Omega): \bv\in \bH^1_0(K)\ \forall K\in\mathcal{T}_h\}.
$$
Then, for a fixed $(\tilde{\bu},\tilde{c})\in \tilde{X}(\mathcal{T}_h)\times H^1_0(\Omega)$, we define the bilinear form $\mathcal{A}^{(\tilde{\bu},\tilde{c})}_h(\cdot,\cdot)$ as
\begin{align*}
\mathcal{A}_h^{(\tilde{\bu},\tilde{c})}\bigl((\bu,p,s,c),(\bv,q,\phi,\psi)\bigr) &= \tilde{a}_1^h(\tilde{c},{\bu},\bv)+c_1^h(\tilde{\bu};{\bu},\bv)+b(\bv,p)+b(\bu,q)+a_2({s},\phi)\\
&\quad+c_2(\tilde{\bu};{s},\phi)+\frac{1}{\tau}a_2({c},\psi)+c_2(\tilde{\bu}-v_p\boldsymbol{e}_z;{c},\psi),
\end{align*}
for all  $(\bu,p,s,c),(\bv,q,\phi,\psi)\in \bV_h\times \mathcal{Q}_h\times \mathcal{M}_h\times \mathcal{M}_h$, 
where 
\[
\tilde{a}_1^h(\tilde{c},{\bu},\bv) \coloneqq 
 \int_{\Omega} \bigl(\nu(\tilde{c}) \bnabla_h(\bu) : \bnabla_h(\bv)  \bigr) +
\sum_{e\in\mathcal{E}_h} \int_e \biggl( 
 \frac{a_0}{h_e}\nu(\tilde{c}) \djump{\bu } \cdot \djump{\bv } \biggr).
\]
Note that
${a}_1^h(\tilde{c},{\bu},\bv)=\tilde{a}_1^h(\tilde{c},{\bu},\bv)+K_h(\tilde{c},\bu,\bv)$, where 
\[ 
K_h(\tilde{c},\bu,\bv) \coloneqq \sum_{e\in\mathcal{E}_h} \int_e \biggl( 
-\dmean{\nu(\tilde{c}) \bnabla_h(\bu) \vn_e} \cdot \djump{\bv } - \dmean{\nu(\tilde{c}) \bnabla_h(\bv) \vn_e} \cdot \djump{\bu}\biggr),
\]
and we point out that $\mathcal{A}_h^{(\cdot,\cdot)}\bigl((\bu,p,s,c),(\bv,q,\phi,\psi)\bigr)$ is well-defined also for every $(\bu,p,s,c), (\bv,q,\phi,\psi)\in \bH^1_0(\Omega)\times L^2_0(\Omega)\times H^1_0(\Omega)\times H^1_0(\Omega)$.

\begin{theorem}[Global inf-sup stability]\label{stabspa11}
Let the pair $(\tilde{\bu},\tilde{c})\in \tilde{X}(\mathcal{T}_h)\times H^1_0(\Omega)$ satisfy
$\|\tilde{\bu}\|_{\Thnorm} <M$, for a sufficiently small 
$M>0$. 
For any $(\bu,p,s,c)\in\bH^1_0(\Omega)\times L^2_0(\Omega)\times H^1_0(\Omega)\times H^1_0(\Omega)$, there exists $(\bv,q,\phi,\psi)\in \bH^1_0(\Omega)\times L^2_0(\Omega)\times H^1_0(\Omega)\times H^1_0(\Omega)$ with $\tnorm{(\bv,q,\phi,\psi)}\le 1$ such that
\[
\mathcal{A}_h^{(\tilde{\bu},\tilde{c})}\bigl((\bu,p,s,c),(\bv,q,\phi,\psi)\bigr)\ge C\tnorm{(\bu,p,s,c)},
\]
where  we define 
$\tnorm{(\bv,q,\phi,\psi)}^2:=\|\bv\|_{\Thnorm}^2+\|q\|_{0,\Omega}^2+\|\phi\|_{1,\Omega}^2+\|\psi\|_{1,\Omega}^2$.
\end{theorem}
\begin{proof}
For any $(\bu,p,s,c)\in\bH^1_0(\Omega)\times L^2_0(\Omega)\times H^1_0(\Omega)\times H^1_0(\Omega)$  there holds 
\[
\mathcal{A}_h^{(\tilde{\bu},\tilde{c})}\bigl((\bu,p,s,c),(\bu,-p,s,c)\bigr)\ge \alpha_a \|\bu\|_{1,\Omega}^2+\tilde{\alpha}_a\|s\|_{1,\Omega}^2+\frac{\tilde{\alpha}_{a}}{\tau}\|c\|_{1,\Omega}^2.
\]
Applying the inf-sup condition, we get that for any $p\in L^2_0(\Omega)$,  there exists a $\bv\in\bH^1_0(\Omega)$ such that 
$b(\bv,p)\ge \beta \|p\|_{0,\Omega}^2$ and $\|\bv\|_{1,\Omega}\le\|p\|_{0,\Omega}$,
where $\beta>0$ is the inf-sup constant depending only on $\Omega$. Then, we have
\begin{align*}
& \mathcal{A}_h(\tilde{u},\tilde{c})(\bu,p,s,c;v,0,0,0)\\
&=a_1(\tilde{c};\bu,\bv)+c_1(\tilde{\bu};\bu,\bv)+b(\bv,p) \ge \beta\|p\|_{0,\Omega}^2-|a_1(\tilde{c};\bu,\bv)|-|c_1(\tilde{\bu};\bu,\bv)|\\
&\ge \beta\|p\|_{0,\Omega}^2-C_a\|\bu\|_{1,\Omega} \|\bv\|_{1,\Omega}-C_c \|\tilde{\bu}\|_{1,h}\|\bu\|_{1,\Omega}\|\bv\|_{1,\Omega} 
 \ge \beta\|p\|_{0,\Omega}^2-2C_a\|\bu\|_{1,\Omega} \|\bv\|_{1,\Omega}\\
&\ge \beta \|p\|_{0,\Omega}^2-2C_a\|\bu\|_{1,\Omega} \|p\|_{0,\Omega} 
\ge \left(\beta-\frac{1}{\epsilon}\right)\|p\|_{0,\Omega}^2-\epsilon C_a^2 \|\bu\|_{1,\Omega}^2,
\end{align*}
where $\epsilon >0$. Now, we introduce a $\delta>0$ such that
\begin{align*}
\mathcal{A}_h^{(\tilde{\bu},\tilde{c})}\bigl((\bu,p,s,c),(\bu+\delta \bv,-p,s,c)\bigr)&=\mathcal{A}_h^{(\tilde{\bu},\tilde{c})}(\bu,p,s,c,\bu,-p,s,c)+\delta\mathcal{A}_h^{(\tilde{\bu},\tilde{c})}(\bu,p,s,c,\bv,0,0,0)\\
&\ge (\alpha_a-\delta\epsilon C_a^2)\|\bu\|_{1,\Omega}^2+\delta \left(\beta-\frac{1}{\epsilon}\right)\|p\|_{0,\Omega}^2+\tilde{\alpha}\|s\|_{1,\Omega}^2+\frac{\tilde{\alpha}}{\tau}\|c\|_{1,\Omega}^2.
\end{align*}
Choosing $\epsilon=2/ \beta$ and $\delta= \alpha_a/(2\epsilon C_a^2)$, we obtain
\begin{align*}
\mathcal{A}_h^{(\tilde{\bu},\tilde{c})}(\bu,p,s,c,\bu+\delta \bv,-p,s,c)&\ge \frac{\alpha_a}{2} \|\bu\|_{1,\Omega}^2+\frac{\beta}{2}\|p\|_{0,\Omega}^2+\tilde{\alpha}\|s\|_{1,\Omega}^2+\frac{\tilde{\alpha}}{\tau}\|c\|_{1,\Omega}^2\\
&\ge \min \biggl\{\frac{\alpha_a}{2},\frac{\beta}{2},\tilde{\alpha},\frac{\tilde{\alpha}}{\tau} \biggr\} \bigl(\|\bu\|_{1,\Omega}^2+\|p\|_{0,\Omega}^2+\|s\|_{1,\Omega}^2+\|c\|_{1,\Omega}^2 \bigr).
\end{align*}
Finally, using triangle inequality, the following relations hold: 
\begin{align*}
\tnorm{(\bu+\delta \bv,-p,s,c)}^2&=\|\bu+\delta \bv\|^2_{1,\Omega}+\|p\|_{0,\Omega}^2+\|s\|_{1,\Omega}^2+\|c\|_{1,\Omega}^2\\
&\le 2 \bigl(\|\bu\|^2_{1,\Omega}+\delta^2\|\bv\|^2_{1,\Omega} \bigr)+\|p\|_{0,\Omega}^2+\|s\|_{1,\Omega}^2+\|c\|_{1,\Omega}^2\\
&\le \max\{2,(1+2\delta^2)\} \bigl(\|\bu\|^2_{1,\Omega}+\|p\|_{0,\Omega}^2+\|s\|_{1,\Omega}^2+\|c\|_{1,\Omega}^2 \bigr).
\end{align*}
This concludes the proof. 
\end{proof}

Next, we decompose the $\bH(\div)$-conforming velocity approximation uniquely into
$\bu_h=\bu_h^c+\bu_h^r$, 
where $\bu_h^c\in\bV_h^c$ and $\bu_h^r\in(\bV_h^c)^{\bot}$, and we note that $\bu^r_h=\bu_h-\bu^c_h\in\bV_h$.
\begin{lemma}
There holds
\[
\|\bu^r_h\|_{\Thnorm}\le {C}_r \left(\sum_{K\in\mathcal{T}_h}\Psi_{J_K}^2 \right)^{1/2}.
\]
\end{lemma}
\begin{proof} It follows straightforwardly from the decomposition $\bu_h=\bu_h^c+\bu_h^r$   and 
from the edge residual.\end{proof}

\begin{lemma}\label{rel1.1}
If $\|\bu\|_{1,\infty}< M$, $\|s\|_{\infty}<M$ and $\|c\|_{\infty}<M$, then the following estimate holds:
\begin{align*}
\frac{C}{2}\tnorm{(\boldsymbol{e}^{\bu},e^p,e^s,e^c)} &\le \int_{\Omega} (\ff-\ff_h)\cdot\bv+\int_{\Omega} \ff_h\cdot(\bv-\bv_h)+\int_{\Omega} f_1(\phi-\phi_h)+\int_{\Omega}f_2(\psi-\psi_h)+K_h(\bu_h,\bv_h)\\
&\quad-\mathcal{A}_h^{({\bu}_h,{c}_h)}(\bu_h,p_h,s_h,c_h,\bv-\bv_h,q,\phi-\phi_h,\psi-\psi_h)
  +(1+C)C_r \left(\sum_{K\in\mathcal{T}_h}\Psi_{j_K}^2 \right)^{1/2}.
\end{align*}
\end{lemma}
\begin{proof}
Using $\bu_h=\bu_h^c+\bu_h^r$ and the triangle inequality imply
\begin{align*}
\tnorm{(\boldsymbol{e}^{\bu},e^p,e^s,e^c)}&\le \tnorm{(\boldsymbol{e}^{\bu}_c,e^p,e^s,e^c)} +\|{\bu}_h^r\|_{\Thnorm}
\le\tnorm{(\boldsymbol{e}^{\bu}_c,e^p,e^s,e^c)}+{C}_r \left(\sum_{K\in \mathcal{T}_h}\Psi_{J_K}^2 \right)^{1/2}.
\end{align*}
Then, Theorem \ref{stabspa11} gives
\begin{align*}
C\tnorm{(\boldsymbol{e}_{c}^{\bu},e^p,e^s,e^c)} 
&\le \mathcal{A}_h^{({\bu_h},{c}_h)}(\boldsymbol{e}^{\bu},e^p,e^s,e^c; \bv,q,\phi,\psi)+\mathcal{A}_h^{({\bu_h},{c}_h)}({\bu}_h^r,0,0,0; \bv,q,\phi,\psi)\\
&\le \mathcal{A}_h^{({\bu_h},{c}_h)}(\boldsymbol{e}^{\bu},e^p,e^s,e^c; \bv,q,\phi,\psi)+C_r \left(\sum_{K\in\mathcal{T}_h}\Psi_{J_K}^2 \right)^{1/2}.
\end{align*}
Owing to the relation 
\begin{align*}
 \mathcal{A}_h^{({\bu_h},{c}_h)}(\bu,p,s,c; \bv,q,\phi,\psi)&= \mathcal{A}_h^{({\bu},{c})}(\bu,p,s,c; \bv,q,\phi,\psi)-a_{1}(c;\bu,\bv)+a_1(c_h;\bu,\bv)\\
 &\quad-c_1(e^{\bu};\bu,\bv)-c_2(e^{\bu};s,\phi)-c_2(e^{\bu};c,\psi),
\end{align*}
we then have
\begin{align*}
 C\tnorm{(\boldsymbol{e}^{\bu},e^p,e^s,e^c)} & \le C\tnorm{(\boldsymbol{e}_{c}^{\bu},e^p,e^s,e^c)}+C_r 
 \left(\sum_{K\in\mathcal{T}_h}\Psi_{J_K}^2 \right)^{1/2}\\
& \le\mathcal{A}_h^{({\bu},{c})}(\bu,p,s,c; \bv,q,\phi,\psi)-a_{1}({c};\bu,\bv)+a_1(c_h,\bu,\bv)-c_1(e^{\bu};\bu,\bv)-c_2(e^{\bu};s,\phi)\\
 &\quad-c_2(e^{\bu};c,\psi)-\mathcal{A}_h({\bu}_h,{c}_h)(\bu_h,p_h,s_h,c_h; \bv,q,\phi,\psi) +(1+C)C_r
  \left(\sum_{K\in\mathcal{T}_h}\Psi_{J_K}^2 \right)^{1/2},
\end{align*}
while using the properties 
\begin{gather*}
\bigl|a_1(c;\bu,\bv)-a_1(c_h;\bu,\bv) \bigr|\le C_1\|c-c_h\|_{1} \|\bu\|_{1,\infty}\|\bv\|_1 \le C_1M\|e^c\|_1,\\
c_1(e^{\bu};\bu,\bv)\le \ C_2M\|e^{\bu}\|_{\Thnorm},\quad 
c_2(e^{\bu};s,\phi)\le C_3 M \|e^{\bu}\|_{\Thnorm},\quad 
c_2(e^{\bu};c,\psi)\le C_4 M \|e^{\bu}\|_{\Thnorm} 
\end{gather*}
yields the bound
\begin{align*}
C\tnorm{(\boldsymbol{e}^{\bu},e^p,e^s,e^c)} &\le \mathcal{A}_h^{({\bu},{c})}(\bu,p,s,c; \bv,q,\phi,\psi)-\mathcal{A}_h({\bu}_h,{c}_h)(\bu_h,p_h,s_h,c_h; \bv,q,\phi,\psi)\\
&\quad+(1+C)C_r \left(\sum_{K\in\mathcal{T}_h}\Psi_{J_K}^2 \right)^{1/2} -(C_1+C_2+C_3+C_4)M \tnorm{(\boldsymbol{e}^{\bu},e^p,e^s,e^c)}.
\end{align*}
Moreover, we have
\begin{align}\label{eqstab11}
\frac{C}{2}\tnorm{(\boldsymbol{e}^{\bu},e^p,e^s,e^c)}&\le \int_{\Omega}\ff \cdot\bv+\int_{\Omega}f_1\phi +\int_{\Omega}f_2\psi-\mathcal{A}_h^{({\bu}_h,{c}_h)}(\bu_h,p_h,s_h,c_h; \bv,q,\phi,\psi)\nonumber\\
&\quad+(1+C)C_r \left(\sum_{K\in\mathcal{T}_h}\Psi_{j_K}^2 \right)^{1/2},
\end{align}
and we readily see that after stating the discrete problem as 
\begin{align*}
a_1^h(c_h;\bu_h,\bv_h)+c_1^h(\bu_h;\bu_h,\bv_h)+b(\bv_h,p_h)-\int_{\Omega}\ff_h \cdot\bv_h&=0,\quad \forall\bv_h\in\textbf{V}_h,\\
a_2(s_h,\phi_h)+c_2(\bu_h;s_h,\phi_h)-\int_{\Omega}f_1\phi_h&=0\quad\forall \phi_h\in\mathcal{M}_h,\\
\frac{1}{\tau}a_2(c_h,\psi_h)+c_2(\bu_h-v_pe_z;c_h,\psi_h)- \int_{\Omega}f_2\psi_h&=0\quad\forall \psi_h\in\mathcal{M}_h,
\end{align*}
and employing~\eqref{eqstab11}, the sought results follow.
\end{proof}

\begin{lemma}\label{rel1.2}
For $(\bv,q,s,c)\in \bH^1_0(\Omega)\times L^2_0(\Omega)\times H^1_0(\Omega)\times H^1_0(\Omega)$, there are $\bv_h\in \textbf{V}_h$, $s_h\in \mathcal{M}_h$ and $c_h\in \mathcal{M}_h$ such that
\begin{align}
\int_{\Omega}&(\ff-\ff_h)\cdot\bv+ \int_{\Omega} \ff_h\cdot(\bv-\bv_h)+\int_{\Omega} f_1 (\phi-\phi_h)+\int_{\Omega}f_2(\psi-\psi_h)\nonumber\\
&-\mathcal{A}_h^{({u}_h,{c}_h)}(\bu_h,p_h,s_h,c_h,\bv-\bv_h,q,\phi-\phi_h,\psi-\psi_h)\le C \bigl(\Psi+\|\ff-\ff_h\|_{0,\Omega} \bigr) 
 \tnorm{(\bv,q,s,c)}. \label{eq:aux33}
\end{align}
 \end{lemma}
 \begin{proof}
 Using integration by parts gives
 \begin{align}
 \int_{\Omega}(\ff-\ff_h)\cdot\bv&+\int_{\Omega} \ff_h\cdot(\bv-\bv_h)+\int_{\Omega} f_1 (\phi-\phi_h)+\int_{\Omega}f_2(\psi-\psi_h)\nonumber\\
&-\mathcal{A}_h^{({u}_h,{c}_h)}(\bu_h,p_h,s_h,c_h,\bv-\bv_h,q,\phi-\phi_h,\psi-\psi_h)=T_1+ \cdots +T_5,
 \end{align}
 where we define the terms 
 \begin{align*}
 T_1& \coloneqq  \sum_{K\in\mathcal{T}_h}\int_{K} \biggl(\ff_h+\nabla\cdot \bigl(\nu(c_h)\bnabla\bu_h \bigr)-\bu_h\cdot\bnabla\bu_h-
 \frac{1}{\rho_{\mathrm{m}}} \nabla p_h \biggr)\cdot(\bv-\bv_h) \, \mathrm{d} \bx+\int_{\Omega}(\ff-\ff_h)\cdot\bv \, \mathrm{d} \bx,\\
 T_2&\coloneqq \sum_{K\in\mathcal{T}_h}\int_{\partial K} \biggl( \biggl(  \frac{1}{\rho_{\mathrm{m}}}p_h\vI-\nu(c_h)\bnabla\bu_h\biggr)\cdot\vn_K \biggr)\cdot(\bv-\bv_h) \, \mathrm{d}S ,\\
 T_3&\coloneqq \sum_{K\in\mathcal{T}_h}\int_{\partial K_{in}\setminus \Gamma}\bu_h\cdot\vn_K(\bu_h-\bu_h^e)\cdot(\bv-\bv_h) \,
  \mathrm{d} S,\\
 T_4&\coloneqq\sum_{K\in\mathcal{T}_h}\int_{K} \biggl(f_1+\frac{1}{\mathrm{Sc}}\nabla^2 s_h-\bu_h\cdot\nabla s_h \biggr)(\phi-\phi_h)
  \, \mathrm{d} \bx+\sum_{K\in\mathcal{T}_h}\int_{\partial K} \biggl(\frac{1}{\mathrm{Sc}}\nabla s_h\cdot\vn_K\biggr) (\phi-\phi_h)
    \, \mathrm{d} S ,\\
 T_5&\coloneqq\sum_{K\in\mathcal{T}_h}\int_{K} \biggl(f_2+\frac{1}{\tau \mathrm{Sc}}\nabla^2 c_h-(\bu_h-v_pe^z)\cdot\nabla c_h\biggr)(\psi-\psi_h) \, \mathrm{d} \bx+\sum_{K\in\mathcal{T}_h}\int_{\partial K} \biggl(\frac{1}{\tau \mathrm{Sc}}\nabla c_h\cdot\vn_K \biggr) (\psi-\psi_h) \, \mathrm{d}S.  
 \end{align*}
 Applying the Cauchy-Schwarz inequality to  $T_1$ implies
 \begin{align}
 T_1&\le \left(\sum_{K\in\mathcal{T}_h}h_K^2\|\textbf{R}_K\|^2_{0,K} \right)^{1/2}
 \left(\sum_{K\in\mathcal{T}_h}h_K^{-2}\|\bv-\bv_h\|^2_{0,K} \right)^{1/2}+\|\ff-\ff_h\|_{0,\Omega}\|\bv\|_{0,\Omega}\nonumber\\
& \le \left(\sum_{K\in\mathcal{T}_h}h_K^2\|\textbf{R}_K\|^2_{0,K} \right)^{1/2} 
\bar{C}\|\bnabla \bv\|_{0,\Omega} +\|\ff-\ff_h\|_{0,\Omega}\|\bv\|_{0,\Omega}.\nonumber
 \end{align} 
 Next, we rewrite $T_2$ in terms of a sum over interior edges and apply again the  Cauchy-Schwarz inequality. Then 
 \begin{align}
  T_2&=\sum_{e\in\mathcal{E}_h}\int_{e} \djump{((\rho_{\mathrm{m}})^{-1}p_h\vI-\nu(c_h)\bnabla\bu_h)\vn}\cdot(\bv-\bv_h) \, 
   \mathrm{d} S \nonumber\\
  &\le \left(\sum_{e\in\mathcal{E}_h}h_e\|\textbf{R}_e\|_{0,e}^2 \right)^{1/2} \left(\sum_{e\in\mathcal{E}_h}h_e^{-1}\|\bv-\bv_h\|_{0,e}^2 \right)^{1/2}\le \left(\sum_{e\in\mathcal{E}_h}h_e\|\textbf{R}_e\|_{0,e}^2 \right)^{1/2}\bar{C}\|\bnabla \bv\|_{0,\Omega}.
 \end{align}
 Then, owing to the Cauchy-Schwarz inequality, it follows that 
 \begin{align}
 T_3\le \left(\sum_{e\in\mathcal{E}_h}\|\djump{\bu_h}\|_{0,e}^2 \right)^{1/2}\bar{C}\|\bnabla \bv\|_{0,\Omega}.
 \end{align}
 Proceeding similarly, we may establish the following bounds  for $T_4$ and $T_5$:
 \begin{align}
 T_4 &\le \Biggl( \left(\sum_{K\in\mathcal{T}_h}h_K^2\|{R}_{1,K}\|^2_{0,K} \right)^{1/2}
 + \left(\sum_{e\in\mathcal{E}_h}h_e\|{R}_{1,e}\|_{0,e}^2 \right)^{1/2} \Biggr)\bar{C}\|\nabla \phi\|_{0,\Omega},\nonumber\\
 T_5 &\le \Biggl( \left(\sum_{K\in\mathcal{T}_h}h_K^2\|{R}_{2,K}\|^2_{0,K} \right)^{1/2}
 + \left(\sum_{e\in\mathcal{E}_h}h_e\|{R}_{2,e}\|_{0,e}^2\right)^{1/2} \Biggr)\bar{C}\|\nabla \psi\|_{0,\Omega}.\nonumber
 \end{align} 
Finally, \eqref{eq:aux33} results as a combination of the bounds derived for $T_1$, $T_2$, $T_3$, $T_4$ and $T_5$.
 \end{proof}
 \begin{theorem}\label{starelth11}
 Let $(\bu,p,s,c)$ and $(\bu_h,p_h,s_h,c_h)$ be the unique solutions to~\eqref{elliptic_cons111sta} and~\eqref{elliptic_cons112sta}, respectively. Let $\Psi$ be the {\it a posteriori} error estimator defined in~\eqref{eststa}. Then 
 the  following estimate holds:
 \begin{align}
 \tnorm{(\bu-\bu_h,p-p_h,s-s_h,c-c_h)}\le C(\Psi+\|\ff-\ff_h\|_0),
 \end{align}
 where $C>0$ is a constant independent of $h$.
 \end{theorem}
 \begin{proof}
It suffices to apply Lemmas~\ref{rel1.1} and~\ref{rel1.2}.
 \end{proof}
 
\subsection{Efficiency}
For each $K\in\mathcal{T}_h$, we can define the standard polynomial bubble function $b_K$. Then, for any polynomial function $\bv$ on $K$, the following results hold:
\begin{subequations}\label{bubleele11}
\begin{align}
\|b_K\bv\|_{0,K}&\le C \|\bv\|_{0,K},\qquad 
\|\bv\|_{0,K} \le C \|b_{K}^{1/2}\bv\|_{0,K},\\
\|\nabla(b_K\bv)\|_{0,K}&\le Ch_K^{-1}\|\bv\|_{0,K},\qquad 
\|b_K\bv\|_{\infty,K}\le C h_K^{-1}\|\bv\|_{0,K},
\end{align}  
\end{subequations}
where $C$ is a positive constant independent of $K$ and $\bv$.
\begin{lemma}\label{effires11}
The following estimates hold, where $C$ is a positive constant: 
\begin{align*}
h_K\|\textbf{R}_K\|_{0,K}&\le C \bigl(\|c-c_h\|_{1,K}+\|\bu-\bu_h\|_{1,K}+\|p-p_h\|_{0,K}+h_K\|\ff-\ff_h\|_{0,K} \bigr),\\
h_K\|R_{1,K}\|_{0,K}&\le C \bigl(\|s-s_h\|_{1,K}+\|\bu-\bu_h\|_{1,K} \bigr),\qquad 
h_K\|R_{2,K}\|_{0,K} \le C \bigl(\|c-c_h\|_{1,K}+\|\bu-\bu_h\|_{1,K} \bigr),
\end{align*}
 Moreover, it also follows that 
\begin{align*}
\Psi_{K}\le C\tnorm{(\bu-\bu_h,p-p_h,s-s_h,c-c_h)}_K.
\end{align*}
\end{lemma}
\begin{proof}
For each $K\in\mathcal{T}_h$, we define $\textbf{W}_b=b_K\textbf{R}_K$. Then, using \eqref{bubleele11}, we have
\begin{align*}
\|\textbf{R}_K\|_{0,K}^2&\le \|b_K^{1/2}\textbf{R}_K\|^2_{0,K}=\int_K \textbf{R}_K\cdot \textbf{W}_b\\
&=\int_K \biggl(\ff_h+\nabla\cdot \bigl(\nu(c_h)\bnabla \bu_h \bigr)-(\bu_h\cdot\nabla)\bu_h-\frac{1}{\rho_{\mathrm{m}}} \nabla p_h \biggr)\cdot \textbf{W}_b    =T_1+T_2, 
\end{align*}
where
\begin{align*}
T_1&= \int_{K} \biggl( \bigl((\nu(c)-\nu(c_h)\bigr)\bnabla\bu+\nu(c_h)\bnabla(\bu-\bu_h)):\nabla \textbf{W}_b- \frac{1}{\rho_{\mathrm{m}}}(p-p_h)\nabla\cdot\textbf{W}_b \biggr) +\int_{K}(\ff_h-\ff)\cdot\textbf{W}_b,\\
T_2&=\int_K \bigl(((\bu-\bu_h)\cdot\nabla)\bu+(\bu_h\cdot\nabla)(\bu-\bu_h) \bigr)\cdot \textbf{W}_b.
\end{align*}
Using  the Cauchy-Schwarz inequality and  \eqref{bubleele11}  we obtain 
\begin{align*}
T_1&\le C_1 \bigl(\|c-c_h\|_{1,K}+\|\bu-\bu_h\|_{1,K}+\|p-p_h\|_{0,K}+h_K\|\ff-\ff_h\| \bigr)h_K^{-1}\|\textbf{R}_K\|_{0,K},\\
T_2&\le C_2  \|\bu-\bu_h\|_{1,K} h_K^{-1}\|\textbf{R}_K\|_{0,K},
\end{align*}
and combining  these bounds 
leads to the first stated result. The other two bounds follow similarly.  
\end{proof}

Let $e$ denote an interior edge that is shared by two elements $K$ and $K'$. Let $\omega_e$ be the patch which is the union of $K$ and 
$K'$.
Next, we define the edge bubble function $\zeta_e$ on $e$ with the property that it is positive in the interior of the patch $\omega_e$ and zero on
the boundary of the patch. From \cite{verfurth96}, the following results hold:
\begin{subequations}\label{newedgebub11}
\begin{align}
\|q\|_{0,e}&\le C\|\zeta_e^{1/2}q\|_{0,e},\\
\|\zeta_eq\|_{0,K}&\le Ch_e^{1/2}\|q\|_{0,e},  \quad 
\|\nabla(\zeta_{e}q)\|_{0,K}\le Ch_e^{-1/2}\|q\|_{0,e}
\qquad \forall K\in \omega_e.
\end{align}
\end{subequations}

\begin{lemma}\label{edgeff}
The following estimates hold: 
\begin{align*}
h_e\|\textbf{R}_e\|_{0,e}^2&\le C\sum_{K\in\omega_e}
\bigl(\|\bu-\bu_h\|_{1,K}^2+\|c-c_h\|_{1,K}^2+\|p-p_h\|_{0,K}^2+h_K^2\|\ff-\ff_h\|_{0,K}^2 \bigr),\\
h_e\|{R}_{1,e}\|_{0,e}^2&\le C\sum_{K\in\omega_e}
 \bigl(\|\bu-\bu_h\|_{1,K}^2+\|s-s_h\|_{1,K}^2 \bigr),\\
h_e\|{R}_{2,e}\|_{0,e}^2&\le C\sum_{K\in\omega_e}
\bigl(\|\bu-\bu_h\|_{1,K}^2+\|c-c_h\|_{1,K}^2 \bigr). 
\end{align*}
Moreover,  we  also have 
\begin{align*}
\Psi_{e_K}^2\le C\sum_{e\in\partial K}\sum_{K\in\omega_e}(\|\!|(\bu-\bu_h,p-p_h,s-s_h,c-c_h)\|\!|_{K}^2+h_K^2\|\ff-\ff_h\|_{0,K}^2).
\end{align*}
\end{lemma}
\begin{proof}
Let $e$ be an interior edge and let us define a rescaling of the edge bubble function in the form 
\begin{align*}
\bvartheta_e=\sum_{e\in\partial K}\frac{h_e}{2}\textbf{R}_e \zeta_e. 
\end{align*}
Using \eqref{newedgebub11} gives 
\begin{align}
h_e\|\textbf{R}_e\|^2_{0,e}&\le {C}
 \bigl(\djump{(\rho_{\mathrm{m}})^{-1}p_h\vI-\nu(c_h)\bnabla \bu_h} ,  \bvartheta_e\bigr)_e\nonumber\\
&\le {C}
 \bigl(\djump{(\rho_{\mathrm{m}})^{-1}p_h\vI-\nu(c_h)\bnabla \bu_h}-\djump{(\rho_{\mathrm{m}})^{-1}p\vI-\nu(c)\bnabla\bu} ,  \bvartheta_e \bigr)_e.\label{eleedge12}
\end{align}
Using integration by parts on  each element of patch~$\omega_e$ implies 
\begin{align*}
 \bigl(\djump{(\rho_{\mathrm{m}})^{-1}p_h\vI-\nu(c_h)\bnabla \bu_h} ,  \bvartheta_e \bigr)_e&=\sum_{K\in\omega_e}\int_K
  \biggl(\nabla\cdot \bigl(\nu(c_h)\bnabla\bu_h \bigr)-\nabla\cdot  \bigl(\nu(c)\bnabla\bu \bigr)
   + \frac{1}{\rho_{\mathrm{m}}} \nabla (p-p_h)\biggr)\cdot \bvartheta_e \\
&\quad+\int_K\biggl( \frac{1}{\rho_{\mathrm{m}}}(p-p_h)\vI+\nu(c_h)\bnabla\bu_h-\nu(c)\bnabla\bu \biggr):\bnabla \bvartheta_e. 
\end{align*}
Note that $(\bu,p,s,c)$ solves the underlying problem, so we then have 
\begin{align}\label{effedge11}
\bigl(\djump{(\rho_{\mathrm{m}})^{-1}p_h\vI-\nu(c_h)\bnabla \bu_h} ,  \bvartheta_e \bigr)_e&=\sum_{K\in\omega_e}\int_K
 \biggl(\ff+\nabla\cdot \bigl(\nu(c_h)\bnabla\bu_h \bigr)-\bu_h\cdot\bnabla\bu_h- \frac{1}{\rho_{\mathrm{m}}}\nabla p_h \biggr)
\cdot \bvartheta_e\nonumber\\
&\quad +\sum_{K\in\omega_e}\int_K(\bu\cdot\bnabla\bu-\bu_h\cdot\bnabla\bu_h)
\cdot\bvartheta_e\nonumber\\
&\quad +\sum_{K\in\omega_e}\int_K \biggl( \frac{p-p_h}{\rho_{\mathrm{m}}}\vI- \bigl(\nu(c)-\nu(c_h) \bigr)\bnabla\bu_h-\nu(c)\bnabla(\bu-\bu_h)\biggr):\bnabla \bvartheta_e\nonumber\\
&=:T_1+T_2+T_3. 
\end{align}
Next, applying  the Cauchy-Schwarz inequality together with Lemma~\ref{effires11} and \eqref{newedgebub11}  gives
\begin{align*}
T_1&\le C_1 \left(\sum_{K\in\omega_e}h_K^2\|\textbf{R}_K\|_{0,K}^2+h_K^2\|\ff-\ff_h\|_{0,K}^2 \right)^{1/2} \left(\sum_{K\in\omega_e}h_K^{-2}\|\bvartheta_e\|_{0,K}^2 \right)^{1/2} \\
&\le C_1 \left(\sum_{K\in\omega_e} \tnorm{(\bu-\bu_h,p-p_h,s-s_h,c-c_h)}_K^2 \right)^{1/2} h_e^{1/2}\|\textbf{R}_e\|_{0,e}, \\
T_2&\le C_2 \left(\sum_{K\in\omega_e}\|\bu-\bu_h\|_{1,K}^2
\right)^{1/2}h_e^{1/2}\|\textbf{R}_e\|_{0,e}, \\
T_3&\le C_3 \left(\sum_{K\in\omega_e}\tnorm{(\bu-\bu_h,p-p_h,s-s_h,c-c_h)}_{K}^2 \right)^{1/2}h_e^{1/2}\|\textbf{R}_e\|_{0,e}. 
\end{align*}
Combining the bounds of $T_1$, $T_2$ and $T_3$ with~\eqref{eleedge12} and~\eqref{effedge11}  implies the first stated result. Similarly, we can prove the other two bounds. 
\end{proof}

\begin{theorem}\label{staefflth11}
 Let $(\bu,p,s,c)$ and $(\bu_h,p_h,s_h,c_h)$ be the unique solutions of problems  \eqref{elliptic_cons111sta} and \eqref{elliptic_cons112sta}, respectively. Let $\Psi$ be defined as in~\eqref{eststa}. Then 
 there exists a constant $C>0$ that is independent of $h$ such that 
 \begin{align*}
 \Psi\le C \Biggl(\tnorm{(\bu-u_h,p-p_h,s-s_h,c-c_h)}+ \left(\sum_{K\in\mathcal{T}_h}h_K^2\|\ff-\ff_h\|_{0,K}^2  \right)^{1/2} \Biggr).
 \end{align*}
 \end{theorem}
 \begin{proof}
 Combining Lemmas~\ref{effires11} and~\ref{edgeff} implies the stated result.
 \end{proof}
 
 \section{{\it A posteriori} error bound for the semidiscrete method}
 \label{sec:aposteriori2}
For each $t\in(0,T]$, let us consider the problem: find $(\tilde{\bu},\tilde{p},\tilde{c},\tilde{s})\in \bH^1_0(\Omega)\times L^2_0(\Omega)\times H^1_0(\Omega)\times H^1_0(\Omega)$ such that
\begin{align*}
a_1(c_h,\tilde{\bu},\bv)+c_1({\bu}_h;\tilde{\bu},\bv)+b(\bv,\tilde{p})&=({\ff},\bv),\quad\forall \bv\in\bH^1_0(\Omega),\\
b(\tilde{\bu},q)&=0,\quad \forall q\in L^2_0(\Omega),\\
a_2(\tilde{s},\phi)+c_2(\bu_h;\tilde{s},\phi)&=(f_1,\phi),\quad \forall\phi\in H^1_0(\Omega),\\
\frac{1}{\tau}a_2(\tilde{c},\psi)+c_2(\bu_h-v_p\boldsymbol{e}_z;\tilde{c},\psi)&=(f_2,\psi)\quad \forall\psi\in H^1_0(\Omega),
\end{align*}
where 
\begin{equation}\label{def:fs-sec6}
\ff=
(\alpha s_h+\beta c_h)\textbf{g}
- \partial_t \bu_{h},\quad f_1=-  \partial_t s_h, \quad f_2=- \partial_t c_h.
\end{equation}

Also, for each $t\in(0,T]$, we write the discrete weak formulation: find $(\tilde{\bu}_h,\tilde{p}_h,\tilde{c}_h,\tilde{s}_h)\in C^{0,1}(0,T; \bV_h)\times C^{0,0}(0,T;\mathcal{Q}_h)\times C^{0,1}(0,T; \mathcal{M}_h)\times C^{0,1}(0,T;\mathcal{M}_h)$ such that
\begin{subequations}\label{elliptic_cons112}
\begin{align}
a_1(c_h,\tilde{\bu}_h,\bv)+c_1({\bu}_h;\tilde{\bu}_h,\bv)+b(\bv,\tilde{p})&=({\ff},\bv),\quad\forall \bv\in \bV_h,\\
b(\tilde{\bu}_h,q)&=0,\quad \forall q\in \mathcal{Q}_h,\\
a_2(\tilde{s}_h,\phi)+c_2(\bu_h;\tilde{s}_h,\phi)&=(f_1,\phi),\quad \forall\phi\in \mathcal{M}_h\\
\frac{1}{\tau}a_2(\tilde{c}_h,\psi)+c_2(\bu_h-v_p\boldsymbol{e}_z;\tilde{c}_h,\psi)&=(f_2,\psi)\quad \forall\psi\in \mathcal{M}_h,
\end{align}
\end{subequations}
where \eqref{def:fs-sec6} remains in effect. 

\begin{lemma}\label{semdislem11}
For each $t\in(0,T]$ and for all $(\bv,q,\phi,\psi)\in\bH^1_0(\Omega)\times L^2_0(\Omega)\times H^1_0(\Omega)\times H^1_0(\Omega)$ we have
\begin{align*}
\left(\partial_t e_{\bu}, \bv\right)+a_1(c,\rho_{\bu},\bv)+c_1({\bu};\rho_{\bu},\bv)+b(\bv,p-\tilde{p})&=a_1(c_h,\tilde{\bu},\bv)-a_1(c,\tilde{\bu},\bv)-c_1(e_{\bu};\tilde{\bu},\bv),\quad\forall \bv\in\bH^1_0(\Omega),\\
b(\bu-\tilde{\bu},q)&=0,\quad \forall q\in L^2_0(\Omega),\\
\left(  \partial_t e_s, \phi\right)+a_2(\rho_s,\phi)+c_2(\bu;\rho_s,\phi)&=-c_2(e_{\bu};\tilde{s},\phi),\quad \forall\phi\in H^1_0(\Omega),\\
\left(\partial_t e_c, \psi\right)+\frac{1}{\tau}a_2(\rho_c,\psi)+c_2(\bu-v_p\boldsymbol{e}_z;\rho_c,\psi)&=-c_2(e_{\bu};\tilde{c},\psi)\quad \forall\psi\in H^1_0(\Omega),
\end{align*}  
where $e_{\bu}=\bu-\bu_h$, $e_s=s-s_h$, $e_c=c-c_h$, $\rho_{\bu}=\bu-\tilde{u}$, $\rho_s=s-\tilde{s}$ and $\rho_c=c-\tilde{c}$.
\end{lemma}
Next we introduce the error indicator $\Theta$ as
\begin{align}\label{semidisest}
\Theta^2=\|e_{\bu}(0)\|^2_{0,\Omega}+\|e_c(0)\|^2_{0,\Omega}+\|e_s(0)\|^2_{0,\Omega}+\int_{0}^T\Psi^2+\int_0^T\Theta_2^2+\max_{0\le t\le T}\Theta_3^2,
\end{align}
where
\begin{align*}
\Theta_2^2=\sum_{e\in\mathcal{E}_h}h_e\|\djump{{\partial_t\bu_h}}\|_{0,e}^2, \qquad 
\Theta_3^2=\sum_{e\in\mathcal{E}_h}h_e\|\djump{\bu_h}\|_{0,e}^2,
\end{align*}
whereas $\Psi$ is the global {\it a posteriori} error estimator for the steady problem
with  element and edge residual contributions defined in 
\eqref{eq:RK-Re}. In this case we now replace $\ff$ and $f_1,f_2$ by \eqref{def:fs-sec6}.

\begin{theorem}\label{semdisrel}
 Let $(\bu,p,s,c)$ and $(\bu_h,p_h,s_h,c_h)$ be the solutions to~\eqref{eq:main_var_mod} and~\eqref{elliptic_cons112}, respectively. Let $\Theta$ be the {\it a posteriori} error estimator defined in~\eqref{semidisest}. Then 
 there exists $C>0$, independent of $h$, such that 
\begin{align*}
\bigl(\|e_{\bu}\|^2_{\star}+\|e_s\|^2_{\star}+\|e_c\|^2_{\star} \bigr)^{1/2}&\le C\,\Theta,\\
\|\partial_t e_{\bu}+\nabla(p-p_h)\|_{L^2(0,T;\bH^{-1}(\Omega))}+\|\partial_t e_s\|_{L^2(0,T;H^{-1}(\Omega))}+\|\partial_t e_c\|_{L^2(0,T;H^{-1}(\Omega))}&\le C\,\Theta, 
\end{align*}
where 
\begin{align*}
\|\bv\|_{\star}^2=\|\bv\|_{L^{\infty}(0,T;L^2(\Omega))}^2+\int_{0}^T\|\bv\|_{\Thnorm}^2dt, \quad 
\|\phi\|_{\star}^2=\|\phi\|_{L^{\infty}(0,T;L^2(\Omega))}^2+\int_{0}^T\|\phi\|_{1}^2dt.
\end{align*}
\end{theorem}
\begin{proof}
Choosing $\bv=e_{\bu}^c$, $q=p-\tilde{p}$, $\phi=e_s$ and $\psi=e_c$ in Lemma~\ref{semdislem11} gives
\begin{align*}
\left( \partial_t e_{\bu}, e_{\bu}^c\right)+a_1(c,\rho_{\bu},e_{\bu}^c)+c_1({\bu};\rho_{\bu},e_{\bu}^c)&=a_1(c_h,\tilde{\bu},e_{\bu}^c)-a_1(c,\tilde{\bu},e_{\bu}^c)-c_1(e_{\bu};\tilde{\bu},e_{\bu}^c),\\
\left( \partial_t e_s , e_s\right)+a_2(\rho_s,e_s)+c_2(\bu;\rho_s,e_s)&=-c_2(e_{\bu};\tilde{s},e_s),\\
\left( \partial_t  e_c, e_c\right)+\frac{1}{\tau}a_2(\rho_c,e_c)+c_2(\bu-v_p\boldsymbol{e}_z;\rho_c,e_c)&=-c_2(e_{\bu};\tilde{c},e_c).
\end{align*}
Moreover,  there also holds 
\begin{align*}
\left( \partial_t e_{\bu}^c, e_{\bu}^c\right)+a_1(c,e_{\bu}^c,e_{\bu}^c)+c_1({\bu};e_{\bu}^c,e_{\bu}^c)&=\left(\partial_t \bu_{h,r} , e_{\bu}^c\right)+a_1(c_h,\tilde{\bu},e_{\bu}^c)-a_1(c,\tilde{\bu},e_{\bu}^c)-c_1(e_{\bu};\tilde{\bu},e_{\bu}^c)\\
&\quad+a_1(c,\theta_{\bu}^c,e_{\bu}^c)+c_1({\bu};\theta_{\bu}^c,e_{\bu}^c)\nonumber\\
\left( \partial_t e_s, e_s\right)+a_2(e_s,e_s)+c_2(\bu;e_s,e_s)&=-c_2(e_{\bu};\tilde{s},e_s)+a_2(\theta_s,e_s)+c_2(\bu;\theta_s,e_s),\\
\left( \partial_t e_c, e_c\right)+\frac{1}{\tau}a_2(e_c,e_c)+c_2(\bu-v_p\boldsymbol{e}_z;e_c,e_c)&=-c_2(e_{\bu};\tilde{c},e_c)+\frac{1}{\tau}a_2(\theta_c,e_c)+c_2(\bu-v_p\boldsymbol{e}_z;\theta_c,e_s),
\end{align*}
where $\theta_{\bu}^c=\tilde{\bu}-\bu_h^c$.  
Using the Cauchy-Schwarz inequality, we have
\begin{align*}
\frac{\mathrm{d}}{\mathrm{d}t}\|e_{\bu}^c\|^2_{0,\Omega}+\alpha_a\|e_{\bu}^c\|^2_{\Thnorm}&\le \bigl(C_1\|\theta_{\bu}^c\|_{\Thnorm} +M\|\bu_{h,r}\|_{\Thnorm} \bigr)\|e_{\bu}^c\|_{\Thnorm}\nonumber\\&\quad+C_2M\|e_s\|_{1,\Omega}\|e_{\bu}^c\|_{\Thnorm}+C_3M\|e_{\bu}^c\|_{\Thnorm}^2+\|{\partial_t \bu_{h,r}}{}\|_{0,\Omega}\|e_{\bu}^c\|_{0,\Omega},\\
\frac{\mathrm{d}}{\mathrm{d}t}\|e_s\|^2_{0,\Omega}+\tilde{\alpha}_a\|e_s\|^2_{1,\Omega}&\le 
\bigl(C_4\|\theta_s\|_{1,\Omega}+M\|\bu_{h,r}\|_{\Thnorm} \bigr)\|e_s\|_{1,\Omega}+C_5M\|e_s\|_{1,\Omega}\|e_{\bu}^c\|_{1,\Omega},\\
\frac{\mathrm{d}}{\mathrm{d}t}\|e_c\|^2_{0,\Omega}+\tilde{\alpha}_a\|e_c\|^2_{1,\Omega}&\le \bigl(C_6\|\theta_c\|_{1,\Omega} +M\|\bu_{h,r}\|_{1,\Omega} \bigr)\|e_c\|_{1,\Omega}+C_7M\|e_c\|_{1,\Omega}\|e_{\bu}^c\|_{1,\Omega}.
\end{align*}
Let us now suppose that $E_c:=\|e_{\bu}^c(T_0)\|=\|e_{\bu}\|_{L^{\infty}(0,T;L^2(\Omega))}$, for some $T_0\in[0,T]$. Then
using Poincar\'{e}-Friedrichs's inequality, Young's inequality and then combining the three equations implies 
\begin{align*}
\frac{\mathrm{d}}{\mathrm{d}t} \| e_{\bu}^c\|^2_{0,\Omega}+\|e_{\bu}^c\|^2_{\Thnorm}+&\frac{\mathrm{d}}{\mathrm{d}t}\|e_c\|^2_{0,\Omega}+\|e_c\|^2_{1}+\frac{\mathrm{d}}{\mathrm{d}t}\|e_s\|^2_{0,\Omega}+\|e_s\|^2_{1,\Omega}\nonumber\\
&\le C \bigl(\|\theta_{\bu}^c\|_{\Thnorm}^2 + \|\theta_s\|_{1,\Omega}^2+\|\theta_c\|_{1,\Omega}^2+\|{\partial_t \bu_{h,r}}{}\|_{0,\Omega}^2 +M\|\bu_{h,r}\|^2_{\Thnorm} \bigr).
\end{align*}
Integrating with respect to $t$ on $[0,T]$ and $[0,T_0]$ yields
\begin{align*}
\|e_{\bu}^c\|^2_{\star}+\|e_c\|^2_{\star}+\|e_s\|^2_{\star}&\le \|e_{\bu}^c(0)\|^2_{0,\Omega}+\|e_c(0)\|^2_{0,\Omega}+\|e_s(0)\|^2_{0,\Omega} \\
& \quad +C \left (\int_{0}^{T} \bigl(\|\theta_{\bu}^c\|_{\Thnorm}^2 + \|\theta_s\|_{1,\Omega}^2+\|\theta_c\|_{1,\Omega}^2 \bigr)+\int_{0}^T\|{\partial_t\bu_{h,r}}{}\|_{0,\Omega}^2 +M\int_{0}^T\|\bu_{h,r}\|^2_{\Thnorm} \right),
\end{align*}
and we moreover  have
\begin{align}\label{sderr11}
\|e_{\bu}\|^2_{\star}+\|e_c\|^2_{\star}+\|e_s\|^2_{\star}&\le \|e_{\bu}(0)\|^2_{0,\Omega}+\|e_c(0)\|^2_{0,\Omega}+\|e_s(0)\|^2_{0,\Omega}\nonumber\\
& \quad +C \left(\int_{0}^{T} \bigl(\|\theta_{\bu}\|_{\Thnorm}^2 + \|\theta_s\|_{1,\Omega}^2+\|\theta_c\|_{1,\Omega}^2 \bigr)+\int_{0}^T\|{\partial_t \bu_{h,r}}{}\|^2_{0,\Omega} +\|\bu_{h,r}\|_{\star}^2 \right),
\end{align}
and as a result we can combine Theorem~\ref{starelth11} and~\eqref{sderr11}  to readily obtain the first stated result. 

On the other hand, integrating by parts in Lemma \ref{semdislem11} yields
\begin{align*}
\left( \partial_t  e_{\bu} +\nabla(p-p_h), \bv\right)&=-a_1(c,\rho_{\bu},\bv)-c_1({\bu};\rho_{\bu},\bv)-b(\bv,p_h-\tilde{p})+a_1(c_h,\tilde{\bu},\bv)\\
&\quad-a_1(c,\tilde{\bu},\bv)-c_1(e_{\bu};\tilde{\bu},\bv),\quad\forall \bv\in\bH^1_0(\Omega)\\
\left( \partial_t e_s, \phi\right)&=-a_2(\rho_s,\phi)-c_2(\bu;\rho_s,\phi)-c_2(e_{\bu};\tilde{s},\phi),\quad \forall\phi\in H^1_0(\Omega),\\
\left( \partial_t e_c, \psi\right)&=-\frac{1}{\tau}a_2(\rho_c,\psi)-c_2(\bu-v_p\boldsymbol{e}_z;\rho_c,\psi)-c_2(e_{\bu};\tilde{c},\psi)\quad \forall\psi\in H^1_0(\Omega).
\end{align*}
We apply Young's inequality and the definition of the dual norm. Then, we integrate in time 
the resulting expression. Finally, the second result is a consequence of 
Theorem~\ref{starelth11} and~\eqref{sderr11}.
\end{proof}

\section{{\it A posteriori} error analysis for the fully discrete method}
\label{sec:aposteriori3}
In this section, we develop an {\it a posteriori} error estimator for the fully discrete problem and focus the presentation on the simpler case of a time discretisation by the backward Euler method.
For each time step $k$ ($1\le k\le N$), we define the (global in space) 
time indicator $\Xi_k$ as
\begin{align*}
\Xi_k=(\Xi_{k,1}^2 +\Xi_{k,2}^2 +\Xi_{k,3}^2 )^{1/2}, 
\end{align*}
where 
\begin{align*}
\Xi_{k,1}^2& \coloneqq \tilde{\tau}_k \bigl(\|\bu^k_h-I^k\bu^{k-1}_h\|_{1,\mathcal{T}_{h,k}}^2+h_e\tilde{\tau}^{-2}_k\| \djump{ I^k\bu_{h}^{n-1}-\bu_{h}^{n-1}} \|_{0,e}^2+h_e\tilde{\tau}^{-2}_k\|\djump{\bu_{h}^{n}-I^k\bu_{h}^{n-1}}\|_{0,e}^2 \bigr),\\
\Xi_{k,2}^2&\coloneqq\tilde{\tau}_k\|s^k_h-s^{k-1}_h\|^2_1,\qquad 
\Xi_{k,3}^2\coloneqq \tilde{\tau}_k\|c^k_h-c^{k-1}_h\|^2_1.
\end{align*}
Here
\begin{align*}
\|\bu^{k}_{h}\|_{1,\mathcal{T}_{h,k}}^2=\sum_{K\in\mathcal{T}_{h,k}}\|\bnabla_h\bu^k_h\|_{0,K}+\sum_{e\in\mathcal{E}_h}\frac{1}{h_e}\| \djump{\bu^k_h} \|_{0,e}.
\end{align*}
Next we define the accumulated time and spatial error indicators as 
\begin{equation}\label{fullydisest}
\Xi^2 = \sum_{k=1}^{N} \Xi_k^2, \qquad 
\Upsilon^2=\sum_{k=1}^{N}\tilde{\tau}_k \bigl(\Upsilon^2_k(\bu^k_h,p^k_h,s^k_h,c^k_h)+\Upsilon^2_k(I^k\bu_h^{k-1},I^k_pp_h^{k-1},s^{k-1}_h,c^{k-1}_h) \bigr),
\end{equation}
where the terms $\Upsilon^2_k$ are constructed with the {\it a posteriori} error estimator contributions defined as in the 
steady case \eqref{eq:RK-Re}, but at a given time step $k$. That is, 
\begin{align*}
\Upsilon^2_k(\bu_h^{k},p_h^{k},s^{k}_h,c^{k}_h)=\Upsilon^2_{K,k}+\Upsilon^2_{e,k}+\Upsilon^2_{J,k},
\end{align*}
with 
\begin{gather*}
\Upsilon^2_{K,k}\coloneqq h_K^2 \bigl(\|\textbf{R}_{K}^k\|^2_{0,K}+\|R_{1,K}^k\|_{0,K}^2+\|R_{2,K}^k\|_{0,K}^2 \bigr), \\
\Upsilon^2_{e,k} \coloneqq \sum_{e\in\partial K}h_e \bigl(\|\textbf{R}_{e}^{k}\|_{0,e}^2+\|R_{1,e}^k\|^2_{0,e}+\|R_{2,e}^k\|^2_{0,e} \bigr),\quad 
\Upsilon^2_{J,k} \coloneqq \sum_{e\in\partial K}h_e^{-1}\|\djump{\bu_h^k}\|_{0,e}^2,
\end{gather*}
and  
\begin{align*}
& \textbf{R}^k_K \coloneqq \biggl\{ 
- \frac{1}{\tilde{\tau}_k} ( \bu_h^k-I^k\bu^{k-1}_h) 
 + \frac{\rho}{\rho_{\mathrm{m}}} \vg +\nabla\cdot(\nu(c_h)\bnabla\bu_h) -\bu_h\cdot\bnabla\bu_h- 	\frac{1}{\rho_{\mathrm{m}}}\nabla p_h \biggr\}  \biggr|_K,\\
& R^k_{1,K} \coloneqq \biggl\{-\frac{s^k_h-s_h^{k-1}}{\tilde{\tau}_k} +\frac{1}{\mathrm{Sc}}\nabla^2 s_h-\bu_h\cdot\nabla s_h \biggr\} \biggr|_K,\quad R^k_{1,e} \coloneqq  \begin{cases}
\frac{1}{2} \djump{(\mathrm{Sc}^{-1}\nabla s_h)\cdot\vn} &\text{for $e\in\mathcal{E}_h\setminus \Gamma$,} \\
0 & \text{for $e\in \Gamma$,} 
\end{cases},\\
& R^k_{2,K}  \coloneqq \biggl\{- \frac{c^k_h-c^{k-1}_h}{\tilde{\tau}_k} +\frac{1}{\tau \mathrm{Sc}}\nabla^2 c_h-(\bu_h-v_p\ve_z)\cdot\nabla c_h \biggr\} \biggr|_K, \quad R^k_{2,e} \coloneqq \begin{cases}
\frac{1}{2} \djump{((\tau \mathrm{Sc})^{-1}\nabla c_h)\cdot\vn} & \text{for $e\in\mathcal{E}_h\setminus \Gamma$,} \\
0 & \text{for $e\in \Gamma$.} 
\end{cases} 
\end{align*}

For each time step $k$, we can split again the $\bH(\div)$-conforming discrete solution $\bu^k_h$ into a conforming part $\bu^k_{hc}$ and a non-conforming part $\bu^k_{hr}$
such that $\bu^k_h=\bu^k_{h,c}+\bu^k_{h,r}$. For each $t\in(t_{k-1},t_{k}]$, we introduce a linear interpolant $\bu_{h}(t)$ in terms of $t$ as
\begin{align*}
\bu_{h}(t):= \frac{t_{k}-t}{\tilde{\tau}_k} I^k\bu^{k-1}_h+\frac{t-t_{k}}{\tilde{\tau}_k}\bu^{k}_h, 
\end{align*} 
where $\{l_k,l_{k+1}\}$ is the standard linear interpolation basis defined on $[t^k,t^{k+1}]$. Similarly, we may introduce $\bu_{h,c}(t)$ and $\bu_{h,r}(t)$. 
Then, setting $\boldsymbol{e}_{\bu_c}=\bu-\bu_{h,c}$, we have 
$\boldsymbol{e}^{\bu}=\bu-\bu_h=\boldsymbol{e}^{\bu_c}-\bu_{h,r}$. 
For $t\in(t_{k-1},t_k)$, we define 
\begin{align*}
\partial_t\bu_{\tilde{\tau}}(t) \coloneqq \frac{1}{\tilde{\tau}_k}(\bu^k-\bu^{k-1}),
\end{align*}
and for all $t\in(t_{k-1},t_k)$, we consider the problem of finding  $(\tilde{\bu}^k, \tilde{p}^k,\tilde{s}^k,\tilde{c}^k)\in \bH^1_0(\Omega)\times L^2_0(\Omega)\times H^1_0(\Omega)\times H^1_0(\Omega)$ such that
\begin{subequations}\label{elliptic_cons11}
\begin{align}
 \bigl(\partial_t\bu_{h}(t),\bv \bigr)+a_1(c_h,\tilde{\bu}^k,\bv)+c_1({\bu}_h;\tilde{\bu}^k,\bv)+b(\bv,\tilde{p}^k)&=(\ff^k,\bv),\quad\forall \bv\in\bH^1_0(\Omega),\\
b(\tilde{\bu}^k,q)&=0,\quad \forall q\in L^2_0(\Omega),\\
\bigl(\partial_t {s}_h(t),\phi \bigr)+a_2(\tilde{s}^k,\phi)+c_2(\bu_h;\tilde{s}^k,\phi)&=0,\quad \forall\phi\in H^1_0(\Omega),\\
(\partial_t{c}_h,\psi)+\frac{1}{\tau}a_2(\tilde{c}^k,\psi)+c_2(\bu_h-v_p\boldsymbol{e}_z;\tilde{c}^k,\psi)&=0\quad \forall\psi\in H^1_0(\Omega).
\end{align}
\end{subequations}

\begin{lemma}\label{lemma1.1}
The following estimates hold
\begin{align}
&\frac{1}{2}\|\boldsymbol{e}^{\bu_c}_{\tilde{{{\tau}}}}(t_n)\|^2_{0,\Omega}+\alpha_1\int_{0}^{t_n}\|\boldsymbol{e}^{\bu_c}_{\tilde{\tau}}(t)\|^2_{1,\Omega}+\frac{1}{2}\|{e}^{s}_{\tilde{\tau}}(t_n)\|^2_{0,\Omega}+{\alpha}_2\int_{0}^{t_n}\|{e}^{s}_{\tilde{\tau}}(t)\|^2_{1,\Omega}+\frac{1}{2}\|{e}^{c}_{\tilde{\tau}}(t_n)\|^2_{0,\Omega}+{{\alpha}_3}{\tau}\int_0^{t_n}\|{e}^{c}_{\tilde{\tau}}(t)\|^2_{1,\Omega}\nonumber\\
&\le C(\Xi^2+\Upsilon^2)+\frac{1}{2}\|\boldsymbol{e}^{\bu_c}_{\tilde{\tau}}(0)\|^2_{0,\Omega}+\frac{1}{2}\|{e}^{s}_{\tilde{\tau}}(0)\|^2_{0,\Omega}+\frac{1}{2}\|{e}^{c}_{\tilde{\tau}}(0)\|^2_{0,\Omega}+\frac{1}{2}\sum_{k=1}^{n-1} \bigl(\|\bu(t_k)-I^{k+1}\tilde{\bu}_{h,c}^k\|^2_{0,\Omega}-\|\boldsymbol{e}^{\bu_c}_{\tilde{\tau}}(t_n)\|^2_{0,\Omega} \bigr),
\nonumber\\ 
&\sum_{k=1}^{n} \bigl(\|\partial_t \boldsymbol{e}^{\bu}_{\tilde{\tau}}+\nabla(p-p_h)\|_{L^2(t_{k-1},t_k;\bH^{-1}(\Omega))}+\|\partial_t e^s_{\tilde{\tau}}\|_{L^2(t_{k-1},t_k;H^{-1}(\Omega))}+\|\partial_t e^c_{\tilde{\tau}}\|_{L^2(t_{k-1},t_k;H^{-1}(\Omega))}\bigr)\nonumber\\
&\le C(\Xi^2+\Upsilon^2)+\frac{1}{2}\|\boldsymbol{e}^{\bu_c}_{\tilde{\tau}}(0)\|^2_{0,\Omega}+\frac{1}{2}\|{e}^{s}_{\tilde{\tau}}(0)\|^2_{0,\Omega}+\frac{1}{2}\|{e}^{c}_{\tilde{\tau}}(0)\|^2_{0,\Omega}+\frac{1}{2}\sum_{k=1}^{n-1} \bigl(\|\bu(t_k)-I^{k+1}\tilde{\bu}_{h,c}^k\|^2_{0,\Omega}-\|\boldsymbol{e}^{\bu_c}_{\tilde{\tau}}(t_n)\|^2_{0,\Omega} \bigr).\nonumber
\end{align}
\end{lemma}
\begin{proof}
Combining~\eqref{eq:main_var_mod} and~\eqref{elliptic_cons11}  implies
\begin{subequations}
\begin{align}\label{lemfuldiseq1}
\nonumber \bigl(\partial_t(\boldsymbol{e}_{\tilde{\tau}}^{\bu}),\bv \bigr)+a_1(c,\bu-\tilde{\bu}^k,\bv)+c_1(\bu,\bu-\tilde{\bu}^k,\bv)+b(\bv,p-\tilde{p}^k)\\
+a_1(c-c_h,\tilde{\bu}^k,\bv)+c_1(\bu-\bu_h,\tilde{\bu}^k,\bv)&= 0, \quad \forall \bv\in \bH^1_0(\Omega),\\
b(\bu-\tilde{\bu}^k,q)&=0,\quad \forall q\in L^2_0(\Omega),\\
 \bigl(\partial_t e^s_{\tilde{\tau}}(t),\phi \bigr)-a_2(s-\tilde{s}^k,\phi)+c_2(\bu;s-\tilde{s}^k,\phi)&=c_2(\bu-\bu_h;\tilde{s}^k,\phi),\quad \forall\phi\in H^1_0(\Omega),\\
(\partial_te^c_{\tilde{\tau}},\psi)-\frac{1}{{\tau}}a_2(c-\tilde{c}^k,\psi)+c_2(\bu-v_p\boldsymbol{e}_z;c-\tilde{c}^k,\psi)&=c_2(\bu-\bu_h;\tilde{c}^k,\psi)\quad \forall\psi\in H^1_0(\Omega).
\end{align}\end{subequations}
Moreover, we have
\begin{align*}
(\partial_t\boldsymbol{e}_{{\tilde{\tau}}}^{\bu_c},\bv)+a_1(c,\boldsymbol{e}^{\bu_c}_{{\tilde{\tau}}},\bv)+c_1(\bu,\boldsymbol{e}^{\bu_c}_{{\tilde{\tau}}},\bv)+b(\bv,p-p^k)&=(\partial_t\tilde{\bu}_{h,r},\bv)-a_1(c-c_h,\tilde{\bu}^k,\bv)
-c_1(\bu-\bu_h,{\bu}^k,\bv)\nonumber\\
&\quad -a_1(c,\tilde{\bu}_{h,c}-\tilde{\bu}^k,\bv)-c_1(\bu,\tilde{\bu}_{h,c}-\tilde{\bu}^k,\bv),\\
b(\boldsymbol{e}^{\bu}_{{\tilde{\tau}}},q)&=0,\nonumber\\
\bigl(\partial_t e^s_{\tilde{\tau}}(t),\phi \bigr)-a_2(e^{s}_k,\phi)+c_2(\bu;e^{s}_k,\phi)&=c_2(\bu-\bu_h;\tilde{s}^k,\phi)-a_2(s_h-\tilde{s}^k,\phi)+c_2(\bu;s_h-\tilde{s}^k,\phi),\nonumber\\
(\partial_te^c_{\tilde{\tau}},\psi)-\frac{1}{{\tau}}a_2(e^{c}_k,\psi)+c_2(\bu-v_p\boldsymbol{e}_z;e^{c}_k,\psi)&=c_2(\bu-\bu_h;\tilde{c}^k,\psi)\\
&-\frac{1}{{\tau}}a_2(c_h-\tilde{c}^k,\psi)+c_2(\bu-v_p\boldsymbol{e}_z;c_h-\tilde{c}^k,\psi).\nonumber
\end{align*}
Choosing $\bv=\boldsymbol{e}^{\bu_c}_{{\tilde{\tau}}}$, $q=p-p^k$, $\phi=e^s_{\tilde{\tau}}$, $\psi=e^c_{\tilde{\tau}}$ and then combining the first two equations, we have
\begin{align*}
& (\partial_t\boldsymbol{e}_{{\tilde{\tau}}}^{\bu_c},\boldsymbol{e}^{\bu_c}_{{\tilde{\tau}}})+a_1(c,\boldsymbol{e}^{\bu_c}_{{\tilde{\tau}}},\boldsymbol{e}^{\bu_c}_{{\tilde{\tau}}})+c_1(\bu,\boldsymbol{e}^{\bu_c}_{{\tilde{\tau}}},\boldsymbol{e}^{\bu_c}_{{\tilde{\tau}}})
-(\partial_t\tilde{\bu}_{h,r},\boldsymbol{e}^{\bu_c}_{{\tilde{\tau}}})+a_1(c-c_h,\tilde{\bu},\boldsymbol{e}^{\bu_c}_{{\tilde{\tau}}})\\
&\qquad +c_1(\bu-\bu_h,{\bu},\boldsymbol{e}^{\bu_c}_{{\tilde{\tau}}})
 + a_1(c,\tilde{\bu}_{h,c}-\tilde{\bu}^{k},\boldsymbol{e}^{\bu_c}_{{\tilde{\tau}}})+ c_1(\bu,\tilde{\bu}_{h,c}-\tilde{\bu}^k,\boldsymbol{e}^{\bu_c}_{{\tilde{\tau}}}) = 0,\\
& (\partial_t e^s_{\tilde{\tau}},e^s_{\tilde{\tau}})+a_2(e^{s}_{\tilde{\tau}},e^s_{\tilde{\tau}})+c_2(\bu;e^{s}_{\tilde{\tau}},e^s_{\tilde{\tau}})-c_2(\bu-\bu_h;\tilde{s}^k,e^s_{\tilde{\tau}})+a_2(s_h-\tilde{s}^k,e^s_{\tilde{\tau}})-c_2(\bu;s_h-\tilde{s}^k,e^s_{\tilde{\tau}})=0,\\
&(\partial_te^c_{\tilde{\tau}},e^c_{\tilde{\tau}})+\frac{1}{{\tau}}a_2(e^{c}_{\tilde{\tau}},e^c_{\tilde{\tau}})+c_2(\bu-v_p\boldsymbol{e}_z;e^{c}_{\tilde{\tau}},e^c_{\tilde{\tau}})-c_2(\bu-\bu_h;\tilde{c}^k,e^c_{\tilde{\tau}}) 
+ \frac{1}{{\tau}}a_2(c_h-\tilde{c}^k,e^c_{\tilde{\tau}})\\
&\qquad -c_2(\bu-v_p\boldsymbol{e}_z;c_h-\tilde{c}^k,e^c_{\tilde{\tau}}) = 0.
\end{align*}
These identities  readily allow us to derive the following bounds: 
\begin{align*}
\frac{1}{2}\frac{\mathrm{d}}{\mathrm{d}t}\|\boldsymbol{e}^{\bu_c}_{{\tilde{\tau}}}\|_{0,\Omega}^2+\alpha_a\|\boldsymbol{e}^{\bu_c}_{{\tilde{\tau}}}\|^2_{1,\Omega}&\le \bigl(\|\partial_t\tilde{\bu}_{h,r}\|_{0,\Omega}+M\|c-c_h\|_{1,\Omega}+(1+2M)\|\bu_{hc}-\tilde{\bu}^k\|_{1,\Omega}\bigr)\|\boldsymbol{e}^{\bu_c}_{{\tilde{\tau}}}\|_{1,\Omega}+M\|\boldsymbol{e}^{\bu_c}_{\tilde{\tau}}\|^2_{1,\Omega}, \\
\frac{1}{2}\frac{\mathrm{d}}{\mathrm{d}t}\|{e}^{s}_{{\tilde{\tau}}}\|^2_{0,\Omega}+\hat{\alpha}_a\|{e}^{s}_{{\tilde{\tau}}}\|^2_{1,\Omega}&\le \bigl(M\|\bu-\bu_h\|_{\Thnorm}+(1+M)\|s_{h}-\tilde{s}^k\|_{1,\Omega}\bigr)\|{e}^{s}_{{\tilde{\tau}}}\|_{1,\Omega},\\
\frac{1}{2}\frac{\mathrm{d}}{\mathrm{d}t}\|{e}^{c}_{{\tilde{\tau}}}\|^2_{0,\Omega}+\frac{\hat{\alpha}_a}{{\tau}}\|{e}^{c}_{{\tilde{\tau}}}\|^2_{1,\Omega}&\le \bigl(M\|\bu-\bu_h\|_{\Thnorm}+(1+M)\|c_{h}-\tilde{c}^k\|_{1,\Omega}\bigr)\|{e}^{c}_{{\tilde{\tau}}}\|_{1,\Omega}.
\end{align*}
And owing to Young's inequality, we obtain 
\begin{align}\label{newfuleqn11}
\frac{1}{2}\frac{\mathrm{d}}{\mathrm{d}t}\|\boldsymbol{e}^{\bu_c}_{{\tilde{\tau}}}\|^2_{0,\Omega}&+\alpha_1\|\boldsymbol{e}^{\bu_c}_{{\tilde{\tau}}}\|^2_{1,\Omega}+\frac{1}{2}\frac{\mathrm{d}}{\mathrm{d}t}\|{e}^{s}_{{\tilde{\tau}}}\|^2_{0,\Omega}+{\alpha}_2\|{e}^{s}_{{\tilde{\tau}}}\|^2_{1,\Omega}+\frac{1}{2}\frac{\mathrm{d}}{\mathrm{d}t}\|{e}^{c}_{{\tilde{\tau}}}\|^2_{0,\Omega}+{{\alpha}_3}\|{e}^{c}_{{\tilde{\tau}}}\|^2_{1,\Omega}\nonumber\\
&\le C_1 \bigl(\|\tilde{\bu}_{hc}-\tilde{\bu}^k\|_{\Thnorm}^2+\|s_{h}-\tilde{s}^k\|_{1,\Omega}^2+\|c_{h}-\tilde{c}^k\|_{1,\Omega}^2\bigr)+\|\partial_t\tilde{\bu}_{h,r}\|^2_{0,\Omega}\nonumber\\
&\le 2C_1 \bigl(\|\tilde{\bu}_{hc}-\tilde{\bu}_{{\tilde{\tau}}}\|_{\Thnorm}^2+\|s_{h}-\tilde{s}_{{\tilde{\tau}}}\|_{1,\Omega}^2+\|c_{h}-\tilde{c}_{{\tilde{\tau}}}\|_{1,\Omega}^2\nonumber\\
&\quad+\|\tilde{\bu}_{{\tilde{\tau}}}-\tilde{\bu}^k\|_{\Thnorm}^2+\|\tilde{s}_{{\tilde{\tau}}}-\tilde{s}^k\|_{1,\Omega}^2+\|\tilde{c}_{{\tilde{\tau}}}-\tilde{c}^k\|_{1,\Omega}^2 \bigr)+\|\partial_t\tilde{\bu}_{h,r}\|^2_{0,\Omega},
\end{align}
where $C_1=\max\{\frac{1}{2\alpha_a}(1+2M)^2,\frac{1}{2\hat{\alpha}_a}(1+M)^2, \frac{\tau}{2\hat{\alpha}_a}(1+M)^2\}$, $\alpha_1=\frac{\alpha}{2}-M-2M^2$, $\alpha_2=  \hat{\alpha}_a / 2$ and $\alpha_3=\frac{\hat{\alpha}_a}{2\tau}-\frac{3M^2}{2}$. 
Moreover, we have
\begin{align*}
\frac{1}{2}\|\boldsymbol{e}^{\bu_c}_{{\tilde{\tau}}}(t_n)\|^2_{0,\Omega}&+\alpha_1\int_{0}^{t_n}\|\boldsymbol{e}^{\bu_c}_{{\tilde{\tau}}}(t)\|^2_{1,\Omega}+\frac{1}{2}\|{e}^{s}_{{\tilde{\tau}}}(t_n)\|^2_{0,\Omega}+{\alpha}_2\int_{0}^{t_n}\|{e}^{s}_{{\tilde{\tau}}}(t)\|^2_{1,\Omega}+\frac{1}{2}\|{e}^{c}_{{\tilde{\tau}}}(t_n)\|_{0,\Omega}^2+{{\alpha}_3}\int_0^{t_n}\|{e}^{c}_{{\tilde{\tau}}}(t)\|^2_{1,\Omega} \\
&\le \frac{1}{2}\|\boldsymbol{e}^{\bu_c}_{{\tilde{\tau}}}(0)\|_{0,\Omega}^2+C \left(\sum_{k=1}^{n}\int_{t_{k-1}}^{t_k}(\|(\tilde{\bu}_{{\tilde{\tau}}}-\tilde{\bu}^k)\|^2_{\Thnorm}+\|(\tilde{s}_{{\tilde{\tau}}}-\tilde{s}^k)\|^2_{1,\Omega}+\|(\tilde{c}_{{\tilde{\tau}}}-\tilde{c}^k)\|^2_{1,\Omega} \right) \\
&\quad +\sum_{k=1}^{n}\int_{t_{k-1}}^{t_k} \bigl(\| \tilde{\bu}_{hc}-\tilde{\bu}_{{\tilde{\tau}}} \|^2_{\Thnorm}+\| s_{h}-\tilde{s}^k \|^2_{1,\Omega}+\| c_{h}-\tilde{c}_{{\tilde{\tau}}} \|^2_{1,\Omega} \bigr) \\
&\quad+\frac{1}{2}\sum_{k=1}^{n-1} \bigl(\|\bu(t_k)-I^{k+1}\tilde{\bu}_{h,c}^k\|^2_{0,\Omega}-\|\boldsymbol{e}^{\bu_c}_{{\tilde{\tau}}}(t_n)\|^2_{0,\Omega} \bigr)+\|\partial_t\tilde{\bu}_{h,r}\|^2_{0,\Omega}. 
\end{align*}
In light of  the definition of $\bu_{{\tilde{\tau}}}$, $s_{{\tilde{\tau}}}$ and $c_{{\tilde{\tau}}}$, we get
\begin{align}\label{newfuleq12} \begin{split} 
&\int_{t_{k-1}}^{t_k}(\|\bnabla(\bu_{{\tilde{\tau}}}-\bu^k)\|^2_{0,\mathcal{T}_h}+\|\nabla(s_{{\tilde{\tau}}}-s^k)\|_{0,\Omega}^2+\|\nabla(c_{{\tilde{\tau}}}-c^k)\|^2_{0,\Omega})\\
&\quad \le \tilde{\tau}_k \bigl(\|\bnabla(\bu^{k}-\bu^{k-1})\|^2_{0,\mathcal{T}_h}+\|\nabla(s^{k}-s^{k-1})\|^2_{0,\Omega}+\|\nabla(c^k-c^{k-1})\|^2_{0,\Omega} \bigr). \end{split} 
\end{align}
Then we can apply triangle inequality, which gives
\begin{align}\label{newfuleq13} \begin{split} 
& \tilde{\tau}_k \bigl( \| \bu^{k}-\bu^{k-1} \|^2_{1,\Omega}+\| s^{k}-s^{k-1} \|^2_{1,\Omega}+\| c^k-c^{k-1} \|^2_{1,\Omega} \bigr)\\
& \le \Upsilon_k^2 
    +\tilde{\tau}_k \bigl(\| \bu^{k}-\bu^{k}_h \|^2_{\Thnorm}+\| s^{k}-s^{k}_h \|^2_{1,\Omega}+\| c^k-c^{k}_h \|^2_{1,\Omega}  \bigr) \\
   &\quad +\tilde{\tau}_k \bigl(\| \bu^{k-1}-I_k\bu^{k-1}_h \|^2_{\Thnorm}+\| s^{k-1}-s^{k-1}_h \|^2_{1,\Omega}+\| c^{k-1}-c^{k-1}_h \|^2_{1,\Omega} \bigr).
   \end{split} 
\end{align}
Combining the results with Theorem \ref{starelth11} implies that
\begin{align}\label{newfuleq14}
&\frac{1}{2}\|\boldsymbol{e}^{\bu_c}_{{\tilde{\tau}}}(t_n)\|^2_{0,\Omega}+\alpha_1\int_{0}^{t_n}\|\boldsymbol{e}^{\bu_c}_{{\tilde{\tau}}}(t)\|^2_{1,\Omega}+\frac{1}{2}\|{e}^{s}_{{\tilde{\tau}}}(t_n)\|^2+{\alpha}_2\int_{0}^{t_n}\|{e}^{s}_{{\tilde{\tau}}}(t)\|^2_{1,\Omega}+\frac{1}{2}\|{e}^{c}_{{\tilde{\tau}}}(t_n)\|^2+{{\alpha}_2}\int_0^{t_n}\|{e}^{c}_{{\tilde{\tau}}}(t)\|^2_{1,\Omega}\nonumber\\
&\le C(\Xi^2+\Upsilon^2)+\frac{1}{2}\|\boldsymbol{e}^{\bu_c}_{{\tilde{\tau}}}(0)\|^2_{0,\Omega}+\frac{1}{2}\|{e}^{s}_{\tilde{\tau}}(0)\|^2_{0,\Omega}+\frac{1}{2}\|{e}^{t}_{\tilde{\tau}}(0)\|^2_{0,\Omega}+\frac{1}{2}\sum_{k=1}^{n-1} \bigl(\|\bu(t_k)-I^{k+1}\tilde{\bu}_{h,c}^k\|^2_{0,\Omega}-\|\boldsymbol{e}^{\bu_c}_{{\tilde{\tau}}}(t_n)\|^2_{0,\Omega} \bigr).
\end{align}
Finally, applying integration by parts in \eqref{lemfuldiseq1}  yields 
\begin{align*}
\bigl(\partial_t \boldsymbol{e}_{\tilde{\tau}}^{\bu} +\nabla(p-p_h),\bv\bigr)&=-a_1(c,\bu-\tilde{\bu}^k,\bv)-c_1(\bu,\bu-\tilde{\bu}^k,\bv)-b(\bv,p_h-\tilde{p}^k)\\
&\qquad-a_1(c-c_h,\tilde{\bu}^k,\bv)-c_1(\bu-\bu_h,\tilde{\bu}^k,\bv)\quad \forall \bv\in\bH^1_0(\Omega),\\
 \bigl(\partial_t e^s_{\tilde{\tau}}(t),\phi\bigr)&=a_2(s-\tilde{s}^k,\phi)-c_2(\bu;s-\tilde{s}^k,\phi)+c_2(\bu-\bu_h;\tilde{s}^k,\phi),\quad \forall\phi\in H^1_0(\Omega),\\
(\partial_te^c_{\tilde{\tau}},\psi)&=\frac{1}{{\tau}}a_2(c-\tilde{c}^k,\psi)-c_2(\bu-v_p\boldsymbol{e}_z;c-\tilde{c}^k,\psi)+c_2(\bu-\bu_h;\tilde{c}^k,\psi)\quad \forall\psi\in H^1_0(\Omega).
\end{align*}
Next we apply Young's inequality and the definition of the dual norm. Then, we integrate the whole expression in time between $t_{k-1}$ and $t_{k}$ for each $k=1,2,\ldots,n$ and sum the expression for each $k$. Finally, we
 use \eqref{newfuleqn11}, \eqref{newfuleq12}, \eqref{newfuleq13} and~\eqref{newfuleq14}  to establish the second stated result.
\end{proof}

\begin{theorem}\label{fullydisrel}
 Let $(\bu,p,s,c)$ be the solution of \eqref{eq:main_var_mod}, and   $(\bu_h,p_h,s_h,c_h)$   the corresponding discrete solution. Let $\Xi$, $\Upsilon$ be the {\it a posteriori} error estimators defined in \eqref{fullydisest}. Then 
 the  following reliability estimate holds:
 \begin{align*}
& \bigl(\|\boldsymbol{e}^{\bu}_{\tau}\|_{\star}^2+\|e^c_{\tau}\|^2_{\star}+\|e^s_{\tau}\|_{\star}^2
\bigr)^{1/2}\\
& \qquad \le C \left(\Xi^2+\Upsilon^2+\frac{1}{2}\|\boldsymbol{e}^{\bu_c}_{{\tilde{\tau}}}(0)\|^2_{0,\Omega}+\frac{1}{2}\|{e}^{s}_{\tilde{\tau}}(0)\|^2_{0,\Omega}+\frac{1}{2}\|{e}^{t}_{\tilde{\tau}}(0)\|^2_{0,\Omega}+\sum_{k=1}^{N-1}\|u^{k}_{h,r}-I^{k+1}u^{k}_{h,r}\|_{0,\Omega}^2 \right)^{1/2},\\
&\sum_{k=1}^{N} \bigl(\|\partial_t \boldsymbol{e}^{\bu}_{\tilde{\tau}}+\nabla(p-p_h)\|_{L^2(t_{k-1},t_k;\bH^{-1}(\Omega))}+\|\partial_t e^s_{\tilde{\tau}}\|_{L^2(t_{k-1},t_k;H^{-1}(\Omega))}+\|\partial_t e^c_{\tilde{\tau}}\|_{L^2(t_{k-1},t_k;H^{-1}(\Omega))} \bigr)\nonumber\\
&\qquad\le C \left(\Xi^2+\Upsilon^2+\frac{1}{2}\|\boldsymbol{e}^{\bu_c}_{{\tilde{\tau}}}(0)\|^2_{0,\Omega}+\frac{1}{2}\|{e}^{s}_{\tilde{\tau}}(0)\|^2_{0,\Omega}+\frac{1}{2}\|{e}^{t}_{\tilde{\tau}}(0)\|^2_{0,\Omega}+\sum_{k=1}^{N-1}\|u^{k}_{h,r}-I^{k+1}u^{k}_{h,r}\|_{0,\Omega}^2 \right)^{1/2},
\end{align*}
where we define 
\begin{align*}
\|\bv\|_{\star}^2& \coloneqq 
 \int_{0}^T \|\bv\|_{\Thnorm}^2 \, \mathrm{d}t,\quad  
\|\phi\|_{\star}^2 \coloneqq 
 \int_{0}^T \|\phi\|_{1,\Omega}^2 \, \mathrm{d}t.
\end{align*}
\end{theorem}
\begin{proof}
Using $\bu^k_h=\bu^k_{h,c}+\bu^k_{h,r}$ together with the identity in \cite[(5.59)-(5.60)]{georgoulis11} that in our context reads
\begin{align*}
 \|\bu(t_k)-I^{k+1}\tilde{\bu}_{h,c}^k\|^2_{0,\Omega}-\|\boldsymbol{e}^{\bu_c}_{{\tilde{\tau}}}(t_k)\|^2_{0,\Omega}=\norm{\vu^{k}_{h,r}-I^{k+1}\vu^{k}_{h,r}}_{0,\Omega}^2+\langle \vu^{k}_{h,r}-I^{k+1}\vu^{k}_{h,r},e_{\tilde{\tau}}^{\bu_c}(t_{k})\rangle,
\end{align*}
we can invoke Lemma \ref{lemma1.1} and reuse the strategy applied in  Theorem \ref{semdisrel} to complete the proof. 
\end{proof}

\section{Numerical tests}\label{sec:numer}
We now present computational examples illustrating the properties of the numerical schemes. All numerical routines have been realised using the open-source finite element libraries FEniCS \cite{alnaes} and FreeFem++ \cite{freefem}. 

\subsection{Example 1: accuracy verification against smooth solutions}

\begin{table}[t]
	\setlength{\tabcolsep}{4pt}
	\renewcommand{\arraystretch}{1.3}
	\centering 
	{\small\begin{tabular}{|r|ccccccccc|}
		\hline 
		1/h & $\texttt{e}_{\vu}$ & \texttt{rate} & $\texttt{e}_{p}$ & \texttt{rate} & $\texttt{e}_{s}$ & \texttt{rate} & $\texttt{e}_{c}$ & \texttt{rate} & $\norm{\sdiv\vu_h}_{\infty,\Omega}$ \tabularnewline
		\hline
		\hline
		$\sqrt{2}$ &  1.3970 &        -- & 5.0910 &        -- & 0.03723 &        -- & 0.02511 &        -- & 2.19e-11\tabularnewline
		2$\sqrt{2}$ & 0.5651 &    1.306 & 1.9920 &    1.354 & 0.01098 &    1.762 & 0.00679 &    1.887 & 4.08e-12\tabularnewline
		4$\sqrt{2}$ & 0.1719 &    1.717 & 0.6402 &    1.637 & 0.00298 &    1.882 & 0.00171 &     1.990 & 1.00e-12\tabularnewline
		8$\sqrt{2}$ & 0.0456 &    1.914 & 0.1695 &    1.917 & 0.00080 &    1.904 & 0.00046 &    1.903 & 5.20e-13\tabularnewline
		16$\sqrt{2}$ & 0.0115 &    1.994 & 0.0412 &    2.039 & 0.00021 &    1.941 & 0.00012 &    1.962 & 2.23e-13 \tabularnewline
		\hline 
	\end{tabular} }
	\medskip\caption{Example 1. Experimental errors and convergence rates for the approximate solutions 
		$\vu_h$, $p_h$, $s_h$ and $c_h$. The $\ell^{\infty}$-norm of the vector formed 	by the divergence of the discrete velocity computed at time $\Tend$ for each discretisation is shown in the last column.} \label{tab:ex1_err}
\end{table}

A known analytical solution example is used to verify theoretical convergence rates of the scheme. We choose  
$\Tend = 2.0$ and  $\Omega = (0,1)^2$. We take the parameter values $\nu = 1.0$, $\rho = 1.0$, $\rho_{\mathrm{m}} = 1.5$, $\vg = (0,-1)^T$, $\Sc =  1.0$, $\tau  = 0.5$, $v_p = 1.0$, $a_0 = 50$. Following 
the approach of manufactured solutions, we prescribe boundary data and additional external 
forces and adequate source terms so that the closed-form solutions to \eqref{eq:model}  are 
given by the  smooth functions
\begin{gather*}
\vu(x,y,t) = \left(\begin{array}{c}
(\sin(\pi x))^2 (\sin(\pi y))^2 \cos(\pi y) \sin(t)\\
-1/3 \sin(2 \pi x) (\sin(\pi y))^3 \sin(t)
\end{array}\right), \quad 
p(x,y,t) = (x^4-y^4) \sin(t), \\
c(x,y,t) = \frac{1}{2}(1+\cos(\pi/4 (xy))) \exp(-t), \quad s(x,y,t) = \frac{1}{2}(1+\sin(\pi/2 (xy))) \exp(-t).
\end{gather*}
As $\vu$ is prescribed everywhere on $\partial\Omega$, for sake of uniqueness we impose $p \in L_{0}^2(\Omega)$ through a real Lagrange multiplier approach. To verify the {\it a priori} error estimates, we introduce the  
discrete norms
\begin{align*}
\tnorm{\vu}_{0,\Th} \coloneqq  \left(\Dt \sum_{n=1}^N \norm{\vu_h^n}_{1,\mathcal{T}_h}^2\right)^{1 / 2}, \quad \text{and} \quad 
\tnorm{\chi}_{0,k} \coloneqq  \left(\Dt \sum_{n=1}^N \norm{\chi_h^n}_{k,\Omega}^2\right)^{1/ 2}. 
\end{align*}
The corresponding individual errors and convergence rates are computed as 
\begin{gather}
\texttt{e}_{\vu} = \tnorm{\vu - \vu_h}_{0,\Th}, \quad \texttt{e}_{p} = \tnorm{p - p_h}_{0,0}, \quad 
\texttt{e}_{s} = \tnorm{s - s_h}_{0,1}, \quad  \texttt{e}_{c} = \tnorm{c - c_h}_{0,1}, \nonumber\\ 
\texttt{rate} =\log(e_{(\cdot)}/\tilde{e}_{(\cdot)})[\log(\xi/\tilde{\xi})]^{-1}, \xi = \{h,\Dt\}, \label{eq:error01}
\end{gather}
where $e,\tilde{e}$ denote errors generated on two
consecutive pairs of mesh size and time step ~$(h,\Dt)$, and~$(\tilde{h},\tilde{\Dt})$, respectively.
Choosing $\Dt = \sqrt{2} h$ and using scheme \eqref{eq:full_disc_scheme},  the results in Table \ref{tab:ex1_err}  confirm that the rates of convergence are optimal, coinciding with the theoretical bounds anticipated in Theorem~\ref{thm:errors_ucs}. 

\subsection{Example 2: adaptive mesh refinement for the stationary problem}
The classical D\"orfler strategy \cite{dorfler94} is employed for the adaptive algorithm 
based on the steps of solving, estimating, marking, and refining. Estimation is performed 
by computing the error indicators and using them to select/mark elements that 
contribute the most to the error \cite{larson08}. 
The marking is done following the bulk criterion of selecting sufficiently many 
elements so that they represent a given fraction of the total estimated 
error. That is, one refines all elements $K\in\Th$ for which  
\begin{equation*}
\Psi_K \geq \gamma_{\text{ratio}} \max_{L\in\Th} \Psi_L,
\end{equation*}
where $0<\gamma_{\text{ratio}}<1$ is a user-defined constant (that we tune 
in order to generate a similar number of 
degrees of freedom, or comparable errors, as those obtained under uniform refinement).  
And then the algorithm  aims for 
equidistribution of the local error indicator in the updated mesh. 

In the adaptive case, instead of \eqref{eq:error01} the convergence rates (for the spatial errors) are computed as 
$$
\texttt{rate} = -2\log(e_{(\cdot)}/\tilde{e}_{(\cdot)})[\log({\tt DoF}/\widetilde{{\tt DoF}})]^{-1}, 
$$
where {\tt DoF} and $\widetilde{{\tt DoF}}$ are the number of degrees of freedom associated with each refinement level. The robustness 
of the global estimators is measured using the effectivity index (ratio between the total error and the indicator) 
$$ {\tt eff}(\Psi) = \frac{\bigl\{ {\tt e}_{\vu}^2+{\tt e}_{p}^2+{\tt e}_{s}^2 +{\tt e}_{c}^2 \bigr\}^{1/2}}{\Psi}.
$$ 

\begin{figure}[!t]
\begin{center}
\includegraphics[width=0.2425\textwidth]{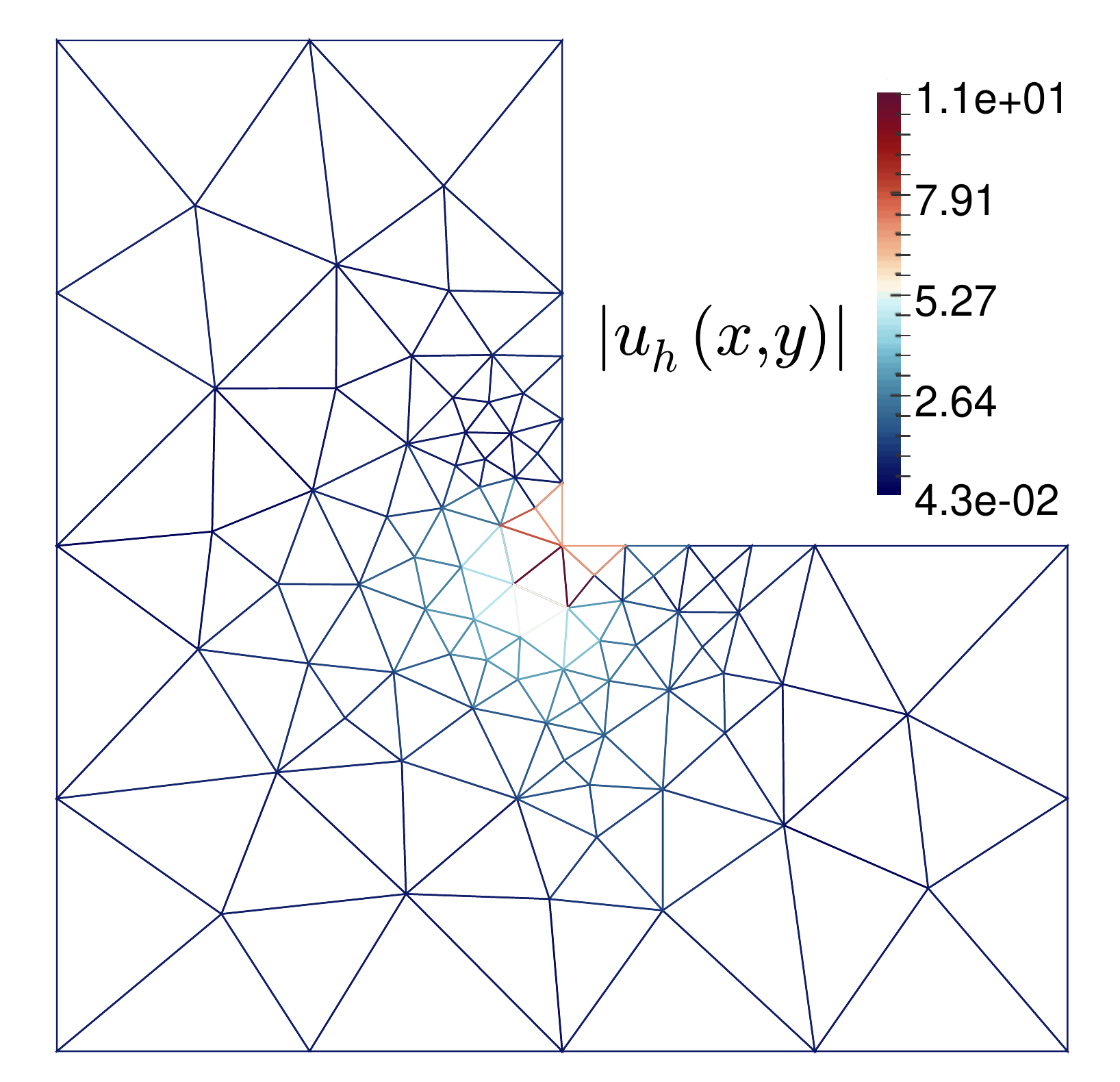}
\includegraphics[width=0.2425\textwidth]{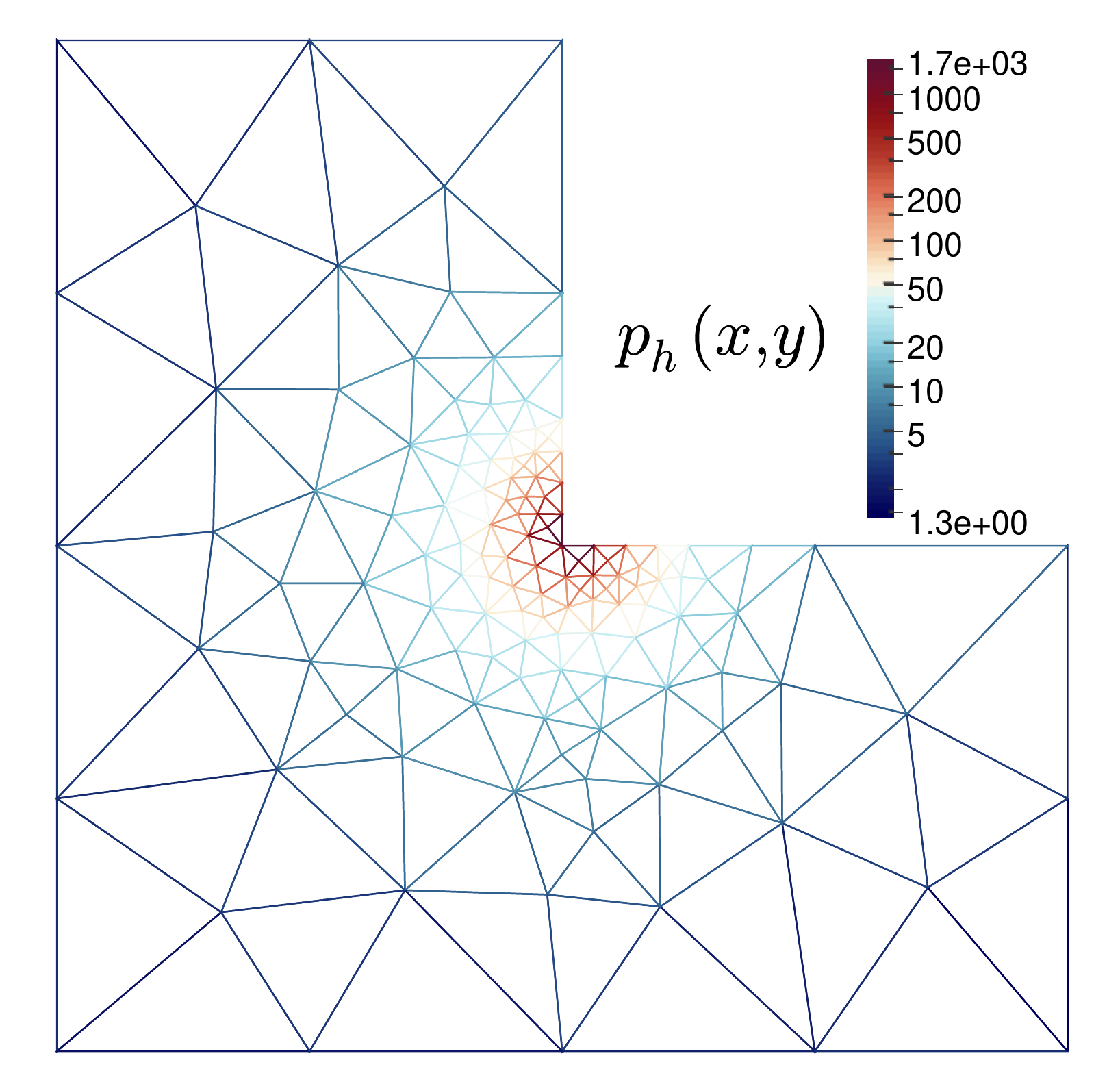}
\includegraphics[width=0.2425\textwidth]{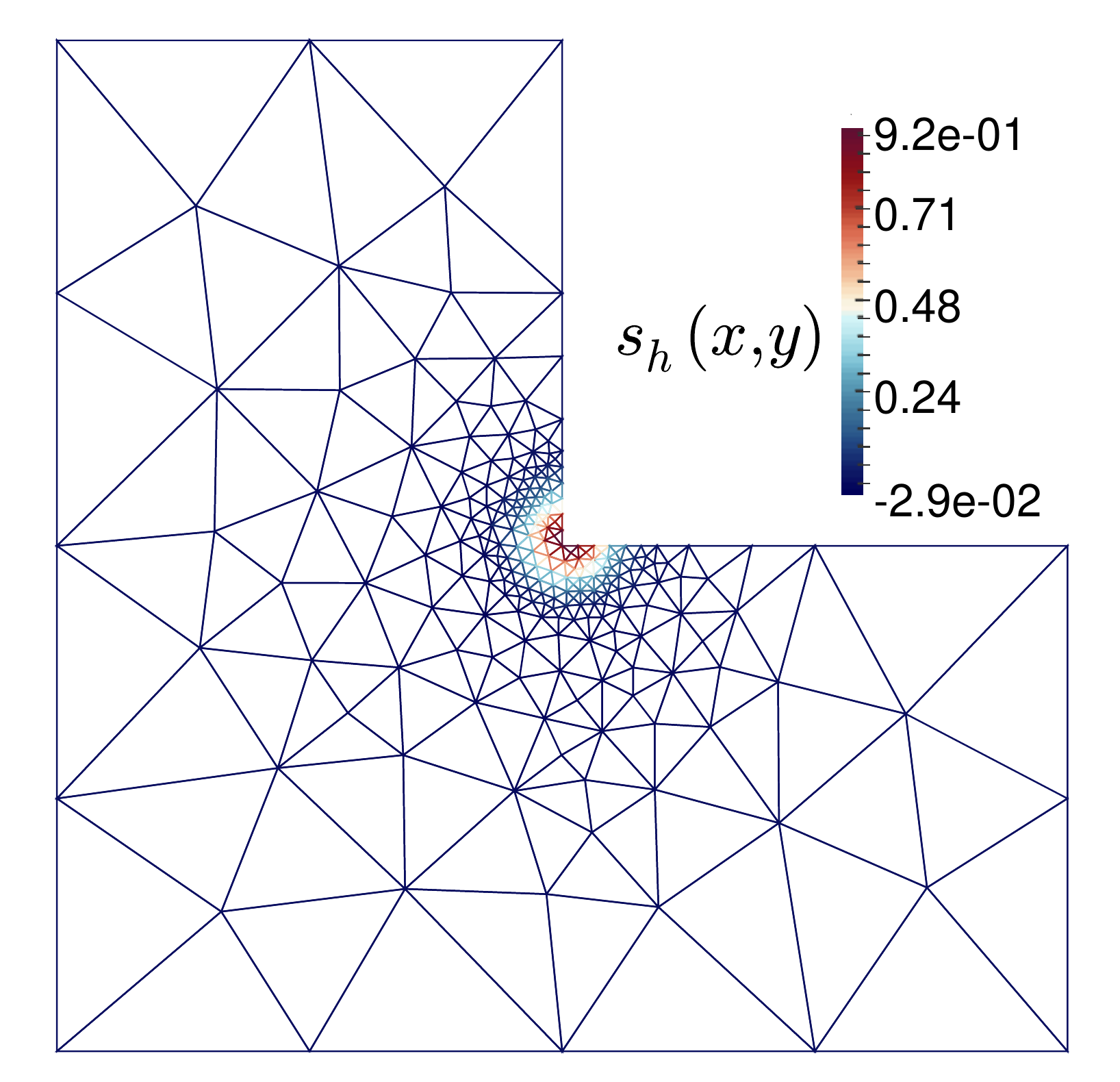}
\includegraphics[width=0.2425\textwidth]{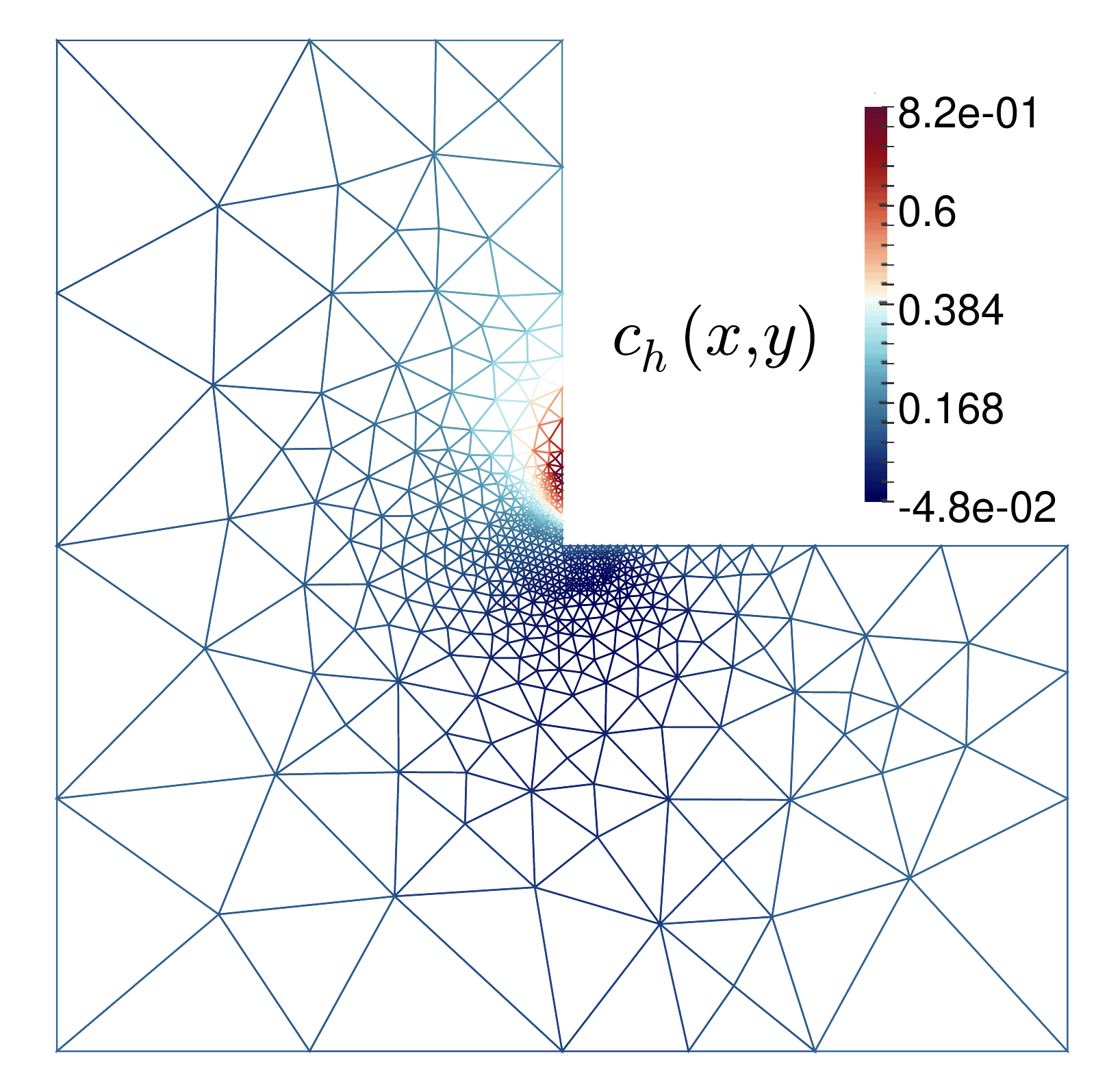}
\end{center}

\vspace{-4mm}
\caption{Example 2. Approximate velocity magnitude (after 3 refinement steps), pressure (after 4 refinement steps), concentration $s$ (after 5 refinement steps), and distribution of $c$ after 6 steps of adaptive refinement.}\label{fig:ex02}  
\end{figure}

We start with verifying the robustness of the {\it a posteriori} error estimator $\Psi$, and construct closed-form solutions to the 
stationary counterpart of the coupled problem \eqref{eq:model}. We consider concentration-dependent 
viscosity, model parameter values, and stabilisation constant as 
\begin{gather*}
\nu(c) = \frac{1}{10}(1+ \exp(-1/4\,c)), \quad \rho = 1, \quad \rho_{\mathrm{m}} = 1.5, \quad \vg = (0,-1)^T, \quad 
\Sc =  1, \\ 
\tau  = 0.5, \quad v_p = 1, \quad \alpha = 0.5, \quad \beta = 0.5, \quad a_0 = 5.\end{gather*}
The considered exact solutions are defined on the L-shaped domain  $\Omega = (-1,1)^2\setminus(0,1)^2$ 
\begin{gather*}
\vu(x,y) = \begin{pmatrix}
\cos(\pi x)\sin(\pi y)\\
-\sin(\pi x)\cos(\pi y)\end{pmatrix}, \quad 
p(x,y) = \frac{2+\sin(xy)}{(x-0.02)^2+(y-0.02)^2}, \\
s(x,y) = \exp(-150(x-0.01)^2-150(y-0.01)^2), \quad 
c(x,y) = \frac{1}{10}+\frac{\cos(\pi x)\sin(\frac{\pi}{2}y)}{25((x-0.1)^2+(y-0.1)^2)}.
\end{gather*}
These solutions exhibit a generic singularity towards the reentrant corner, and therefore one expects that the error decay is suboptimal when applying uniform mesh refinement. After solving the coupled stationary problem on sequences of uniformly and adaptively refined meshes, the aforementioned behaviour is indeed observed in Table~\ref{tab:ex02}, where the first part of the table shows deterioration of the convergence due to the high gradients of the exact solutions on the non-convex domain. The results shown in the bottom block of the table confirm that as more degrees of freedom are added, a restored error reduction rate is observed due to adaptive mesh refinement guided by the {\it a posteriori} error estimator $\Psi$. The second-last column of the table also indicates that the effectivity index oscillates under uniform refinement, while it is much more steady in the adaptive case. We tabulate as well the Newton-Raphson iteration count (needed to reach the relative residual tolerance of 1e-6), and this number is also systematically smaller for the adaptive case (about four steps in all instances) than for the uniform refinement case (up to seven nonlinear steps for certain refinement levels). 
As an example we plot in Figure~\ref{fig:ex02} solutions on relative coarse meshes and display meshes generated with the adaptive algorithm, indicating significant 
refinement near the reentrant corner. 
Let us also remark that the boundary conditions for velocity have been imposed (here and in all other tests) essentially for the normal component, while the tangent component is fixed through a Nitsche's penalisation. For this example we use a constant $a_{\text{Nitsche}} =10^3$. 

\begin{table}[!t]
	\setlength{\tabcolsep}{3pt}
	\centering 
	{\small\begin{tabular}{|rr|ccccccccccc|}
		\hline 
		{\tt DoF} & $h$ & $\texttt{e}_{\vu}$ & \texttt{rate} & $\texttt{e}_{p}$ & \texttt{rate} & $\texttt{e}_{s}$ & \texttt{rate} & $\texttt{e}_{c}$ & \texttt{rate} & $\norm{\sdiv\vu_h}_{\infty,\Omega}$ & $\texttt{eff}(\Psi)$ & \texttt{iter} \tabularnewline
		\hline
\multicolumn{13}{|c|}{Error history under uniform mesh refinement}\tabularnewline
		\hline
  79 &        1 & 26.85 &      -- & 284.0 &      --  & 1.829 &      -- & 1.3724 &    --   & 2.22e-12 & 0.329 & 5\tabularnewline
  275 &      0.5 & 19.57 &  0.396& 152.9 &    0.893 & 1.736 &   0.075 & 0.9903 &  0.499 & 3.65e-13 & 0.257 & 7\tabularnewline
 1027 &     0.25 & 29.59 &  -0.622& 86.54 &    0.821 & 1.188 &   0.547 & 0.7004 &   0.441 & 1.07e-13 & 0.111 & 6\tabularnewline
 3971 &    0.125 & 12.53 &  0.924 & 72.47 &    0.256 & 0.642 &   0.886 & 0.4834 &   0.535 & 3.21e-14 & 0.104 & 7\tabularnewline
15619 &   0.0625 & 7.561 &  0.805 & 53.41 &    0.440 & 0.457 &   0.491 & 0.2863 &   0.756 & 3.39e-14 & 0.171 & 7\tabularnewline
\hline
\multicolumn{13}{|c|}{Error history under adaptive mesh refinement}\tabularnewline
\hline
   79 &        1 & 26.85 &     -- & 284.0 &   --  & 2.829 &     -- & 1.3724 &      -- & 2.22e-12 & 0.329 &4\tabularnewline
  275 &      0.5 & 15.92 &  0.824 & 154.8 & 0.973 & 1.725 &  0.937 & 0.9617 &   0.573 & 3.67e-13 & 0.261 &3\tabularnewline
  943 &      0.5 & 10.78 &  1.172 & 90.23 & 0.878 & 0.822 &  0.863 & 0.6793 &   1.064 & 2.18e-13 & 0.260 &4\tabularnewline
 1601 &      0.5 & 7.398 &  2.499 & 74.35 & 0.732 & 0.641 &  1.428 & 0.4398 &   1.642 & 2.28e-13 & 0.261 &4\tabularnewline
 2363 &      0.5 & 2.139 &  2.265 & 53.35 & 1.706 & 0.461 &  1.569 & 0.2683 &   2.539 & 2.15e-13 & 0.257 &3\tabularnewline
 4253 &   0.2877 & 3.420 &  1.394 & 29.41 & 2.027 & 0.235 &  2.295 & 0.1541 &   1.888 & 1.05e-13 & 0.258 &4\tabularnewline
11662 &     0.25 & 1.012 &  1.267 & 17.58 & 1.019 & 0.118 &  1.368 & 0.0873 &   1.126 & 1.07e-13 & 0.258 &5\tabularnewline
38174 &   0.1416 & 0.557 &  1.006 & 9.388 & 1.058 & 0.059 &  1.156 & 0.0464 &   1.063 & 9.01e-13 & 0.261 &4\tabularnewline
		\hline 
	\end{tabular}} 
	
	\medskip\caption{Example 2. Experimental errors and convergence rates for the approximate solutions 
		$\vu_h$, $p_h$, $s_h$ and $c_h$ of the stationary problem under uniform and adaptive 
		mesh refinement following the estimator $\Psi$. For the adaptive case we employ $\gamma_{\mathrm{ratio}} = 1\cdot10^{-4}$. } \label{tab:ex02}
\end{table}

\subsection{Example 3: robustness of the estimator for the transient problem}
Next we turn to the numerical verification of robustness of the {\it a posteriori} error estimator for the fully discrete approximations of the 
time-dependent coupled problem. We consider now the time interval $(0,0.01]$  and choose $\Dt =0.002$. The closed-form solutions on the unit square domain are as follows 
\begin{gather*}
\vu(x,y,t) = \sin(t)\begin{pmatrix}
\cos(\pi x)\sin(\pi y)\\
-\sin(\pi x)\cos(\pi y)\end{pmatrix}, \quad 
p(x,y,t) = \cos(t)(x^4-y^4), \\
c(x,y,t) = \frac{1}{2}(1+\cos(\pi/4 (xy))) \exp(-t), \quad s(x,y,t) = \frac{1}{2}(1+\sin(\pi/2 (xy))) \exp(-t).
\end{gather*}
Cumulative errors up to $t_{\text{final}}$ are computed as 
	\begin{gather*}
	E_{\vu} \coloneqq   \left(\Dt \sum_{n=1}^N \| \vu_h^n - \vu(t^n)\|_{1,\mathcal{T}_h}^2\right)^{1/2}, \quad E_{p} 
	 \coloneqq   \left(\Dt \sum_{n=1}^N \| p_h^n - p(t^n)\|_{0,\Omega}^2\right)^{1/2}, \\
	 E_{s} \coloneqq  \left(\Dt \sum_{n=1}^N \| s_h^n - s(t^n)\|_{1,\Omega}^2\right)^{1 / 2}, \quad E_{c} 
	  \coloneqq \left(\Dt \sum_{n=1}^N \| c_h^n - c(t^n)\|_{1,\Omega}^2\right)^{1/2},\end{gather*}
and the resulting error history, after six steps of uniform mesh refinement, is collected in Table~\ref{tab:ex03}. To be consistent with the development in Section~\ref{sec:aposteriori3}, the numerical verification in this set of tests was carried out using a backward Euler time discretisation. The {\it a posteriori} error estimator \eqref{fullydisest} is computed and the effectivity index is also tabulated, showing that the estimator is robust (and confirming the theoretical reliability bound as well as giving an heuristic indication of its efficiency). Note that in this case, 
since the mesh refinement is uniform, the auxiliary interpolation of the 
solutions at the last time step on the current mesh is not necessary. The average number of 
Newton-Raphson iterations required for convergence was 3.2.

\begin{table}[!t]
	\setlength{\tabcolsep}{4pt}
	\centering 
	{\small\begin{tabular}{|rc|ccccccccc|}
		\hline 
		{\tt DoF} & $h$ & $E_{\vu}$ & \texttt{rate} & $E_{p}$ & \texttt{rate} & $E_{s}$ & \texttt{rate} & $E_{c}$ & \texttt{rate} & $\texttt{eff}(\Upsilon)$  \tabularnewline
		\hline
 59 &   0.7071 & 0.00228 &      -- & 0.02080 &    -- & 0.01811 &   -- & 0.00711 &   -- & 0.0859\tabularnewline
  195 &   0.3536 & 0.00081 &    1.504 & 0.01191 &   0.843 & 0.00989 &   0.872 & 0.00303 &    1.227  & 0.0970\tabularnewline
  707 &   0.1768 & 0.00034 &    1.217 & 0.00621 &   0.940 & 0.00486 &    1.025 & 0.00149 &    1.022  & 0.0973\tabularnewline
 2691 &  0.0884 & 0.00015 &    1.119 & 0.00313 &   0.984 & 0.00241 &    1.009 & 0.00074 &    1.008  & 0.0967\tabularnewline
10499 &  0.0442 & 7.70e-05 &    1.053 & 0.00157 &   0.996 & 0.00120 &    1.004 & 0.00031 &    1.003  & 0.0961\tabularnewline
41475 &   0.0221 & 3.78e-05 &    1.025 & 0.00078 &    0.999 & 0.00060 &    1.002 & 0.00018 &    1.001  & 0.0972\tabularnewline
164867 &  0.0110 & 1.87e-05 &    1.012 & 0.00039 &   0.999 & 0.00030 &    1.001 & 9.26e-05 &        1.000 & 0.0966\tabularnewline
\hline 
	\end{tabular}}
	
	\medskip\caption{Example 3. Experimental errors and convergence rates for the approximate solutions 
		of the transient problem under uniform  
		mesh refinement following the estimator $\Upsilon$. } \label{tab:ex03}
\end{table}

\subsection{Example 4: simulation of salinity-driven flow instabilities}
To illustrate the behaviour of the model and the proposed method, 
 we simulate a salinity-driven flow problem. Similar examples  are found in   \cite{burns15} where direct numerical simulations (DNS) are applied to the version of 
  \eqref{eq:model}  that has constant viscosity. 

    \begin{figure}[!t]
    	\begin{center}
    		\includegraphics[width=0.223\textwidth]{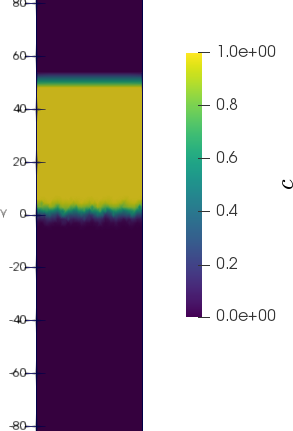}\includegraphics[width=0.223\textwidth]{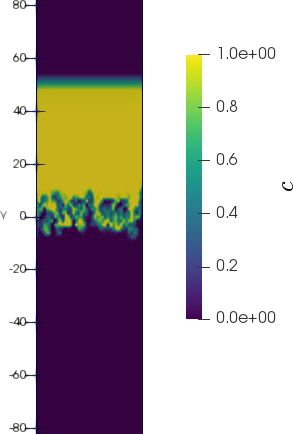}\includegraphics[width=0.223\textwidth]{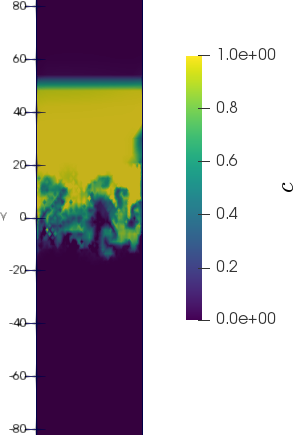}
    	\end{center}
	
\vspace{-4mm}
    	\caption{Example 4. Solid particle concentration profile at times $t=0.0, 3.5$ and $t=6.0$.}\label{fig:ex04_s0}  
    \end{figure}
    
 The configuration of layering in sedimentation is taken from   \cite{reali17}. We consider a rectangular domain with $L_x = 40$, $L_y=300$ and an initial solid-particle concentration profile that is periodic in the horizontal direction, and periodic with Gaussian noise in the vertical direction:
    \begin{align*}
    s(x,y,z,0) = A_0 \exp  (- z^2 / \sigma^2 ) + A_1 \sin(x),
    \end{align*}   
   with initial amplitudes $A_0, A_1$ and width $\sigma$ (see Figure~\ref{fig:ex04_s0}). For the velocity field, we use non-slip boundary condition in all four walls and we choose $\Dt = 0.1$. As discussed in \cite{reali17}, simulations at low density ratios are extremely costly because of the large Reynolds numbers of fingering convection. In consequence we choose an initial density ratio $R_0 = \alpha s_{0z}/\beta c_{0,z} \approx 4$, and we carry out the simulations on tall, thin domains.
   Apart from the specifications above, the remaining constant parameters needed in the model take  the following values 
   \begin{gather*}
   A_0 = 2.86, \quad A_1 = 0.5, \quad \sigma = 0.35, \quad 
   \nu = \num{1e-3}\, \si{kg/m^3}, \quad g = 9.8 \si{m/s^2}, \quad \Sc = 7.0, \\
   \tau = 25, \quad v_\mathrm{p} = 0.04 \si{m/s}, \quad \alpha = -2.0, \quad \beta = 0.5.
   \end{gather*}
According to \cite{reali17} a linear fingering instability occurs provided $1 < R_0 < \tau$, hence the 
instability shown in Figure~\ref{fig:ex04_s0} is expected.

\subsection{Example 5: adaptive simulation of exothermic flows}
To conclude this section, and to include an illustrative simulation exemplifying that the 
$\vH$(div)-conforming scheme along with the {\it a posteriori} error estimator perform 
well for an applicative problem, we address the computation of 
 exothermic flows that develop fingering instabilities. 
The problem configuration is adapted as a simplification of the problem solved in 
 \cite{lenarda17}, where the 
fields $c,s$ represent solutal concentration and temperature, respectively.  
The model assumes an additional drag term due to porosity so that the momentum 
equation is of Navier-Stokes-Brinkman type. The domain is the rectangular region $\Omega=(0,L) \times (0,H)$, and the initial solutal and temperature profiles 
are imposed as 
\begin{equation*}
  c^0(x,y)=\begin{cases}
    0.999+ 0.001\zeta_c \quad & \text{if $H-\epsilon \leq y \leq H$,}  \\
    0 & \text{otherwise,}
  \end{cases} \quad s^0(x,y)=\begin{cases}
    0.999+ 0.001\zeta_s \quad & \text{if $H-\epsilon \leq y \leq H$,}  \\
    0 & \text{otherwise}
  \end{cases},
\end{equation*}
where $\zeta_c, \zeta_s$ are  random fields uniformly distributed on $[0,1]$.
The geometric and model constants are $H=1000$, $L=2000$, $\Delta t=20$, $t_{\mathrm{end}}=1500$, $\nu=1 + 0.25\zeta_\nu$, $\kappa=1$, $1/\text{Sc} = 8$, $1/(\tau\text{Sc}) = 2.5$, $\rho_\mathrm{m} = 1$, $v_\mathrm{p} = 0$, 
$\alpha=5$, $\beta = -1 $. 

Boundary conditions are of mixed type for solutal and temperature. Both fields are prescribed to 0 and 1 on the bottom and top of the domain, respectively; while on the vertical walls we impose zero-flux boundary conditions. The velocity is of slip type on the whole boundary, and therefore a zero-mean condition for the pressure is considered using a real Lagrange multiplier. 
The solution algorithm, differently from the previous tests, is based on an inner fixed-point iteration between an Oseen and a transport system, rather than an exact Newton-Raphson method. An initial coarse mesh of 5300 elements is constructed, and 
an adaptive mesh refinement (only one iteration) guided by the estimator \eqref{fullydisest} 
is applied at the end of each time step. Figure~\ref{fig:ex05} shows snapshots of adapted meshes at different times, and also samples of solute concentration, temperature distribution, velocity, and pressure at the final time.

 \begin{figure}[!t]
    	\begin{center}
    \includegraphics[width=0.24\textwidth]{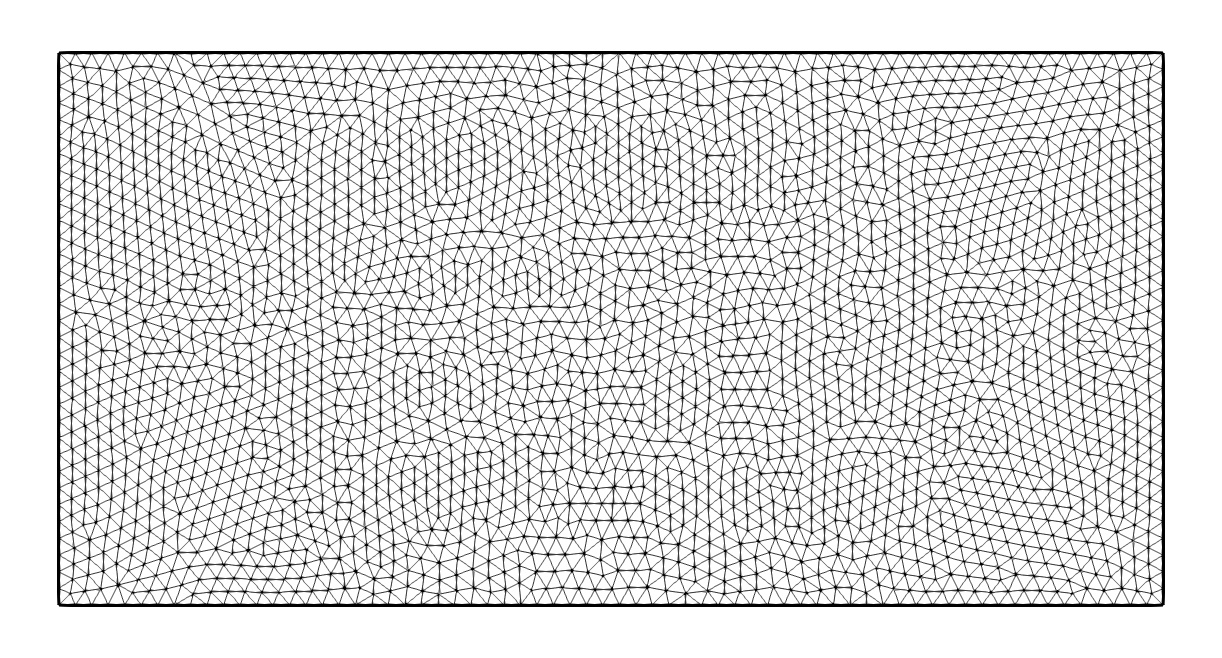}
    \includegraphics[width=0.24\textwidth]{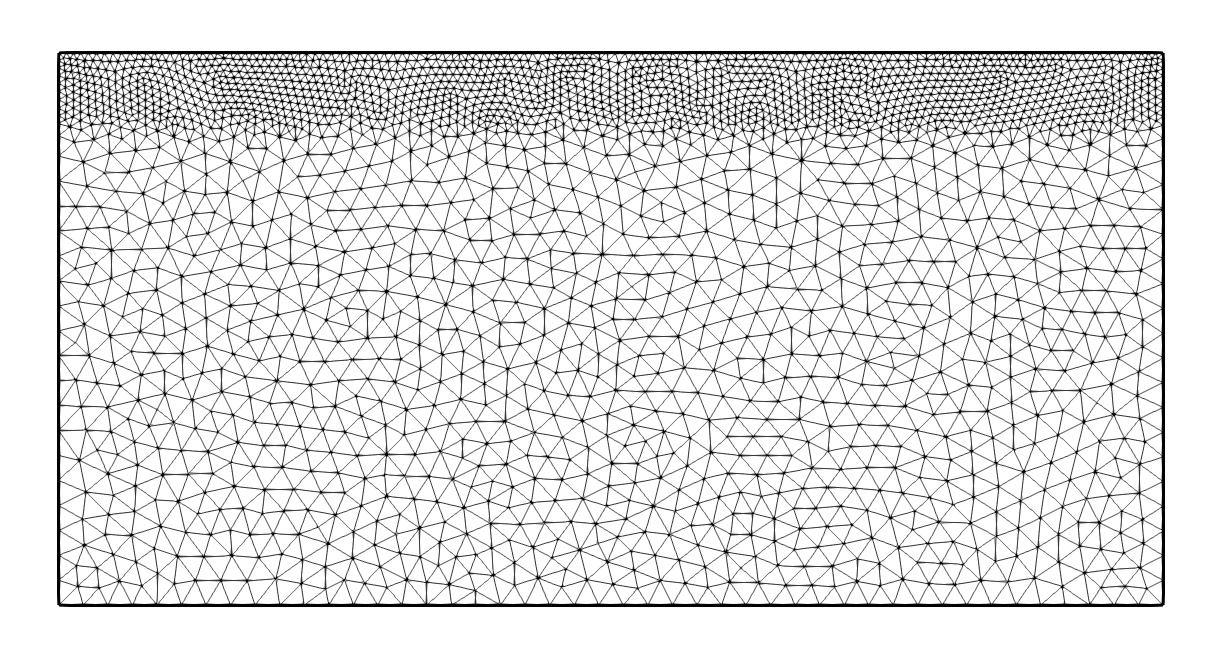}
    \includegraphics[width=0.24\textwidth]{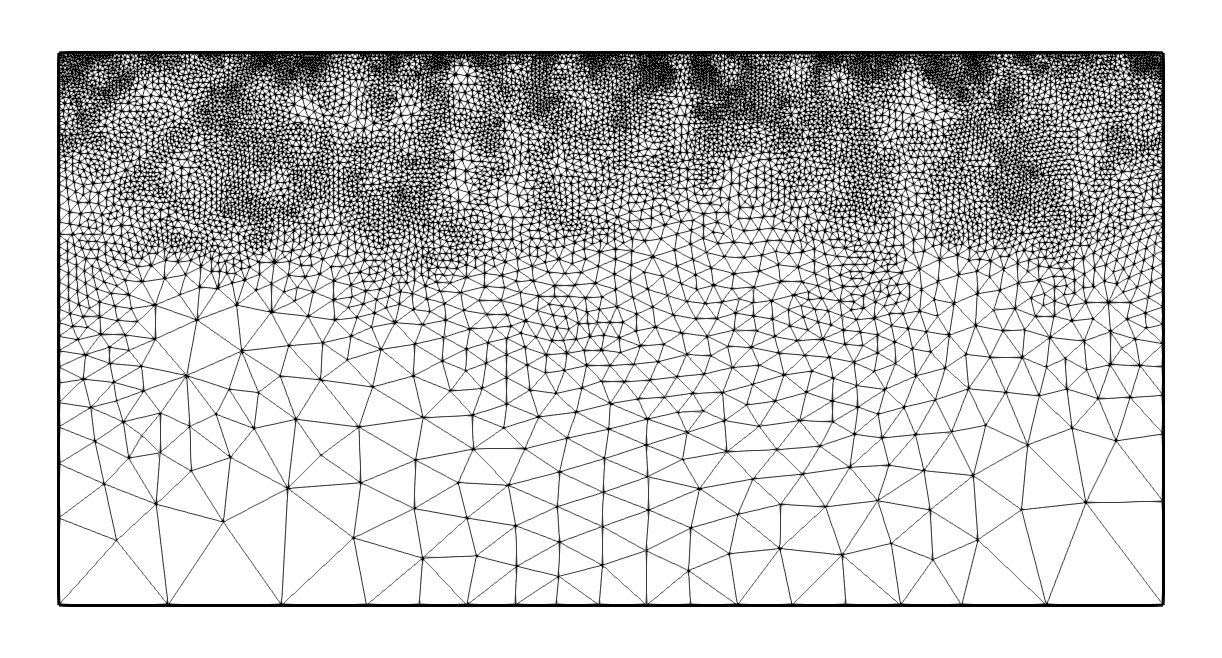}
    \includegraphics[width=0.24\textwidth]{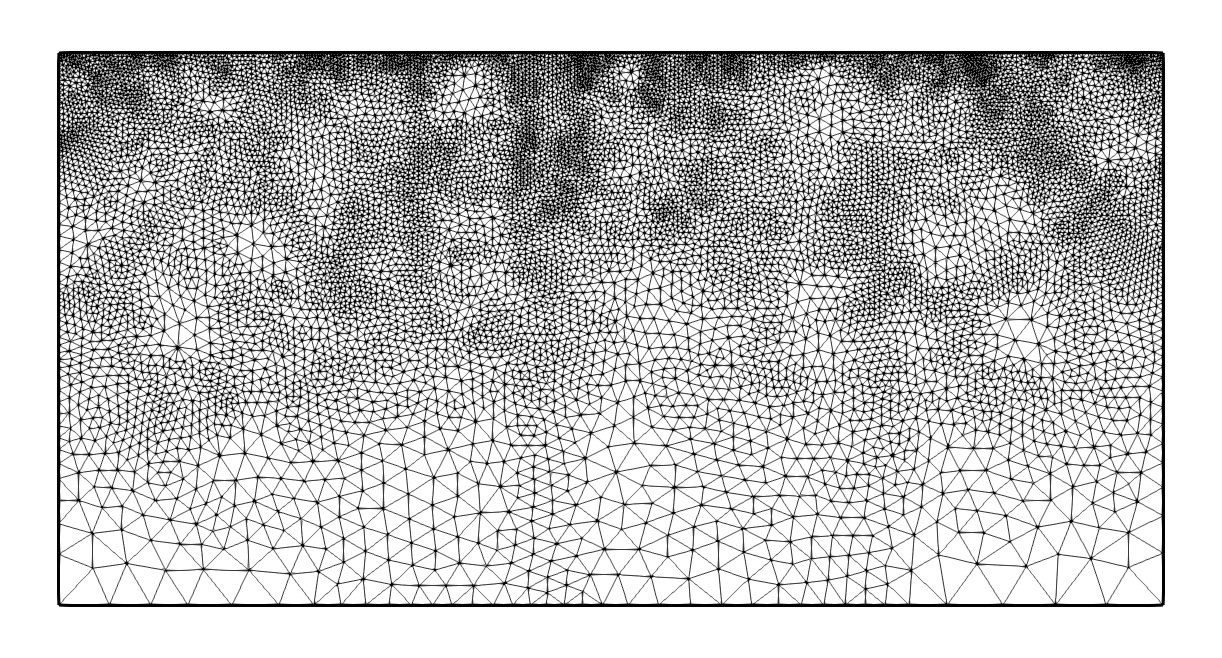}\\
    \includegraphics[width=0.24\textwidth]{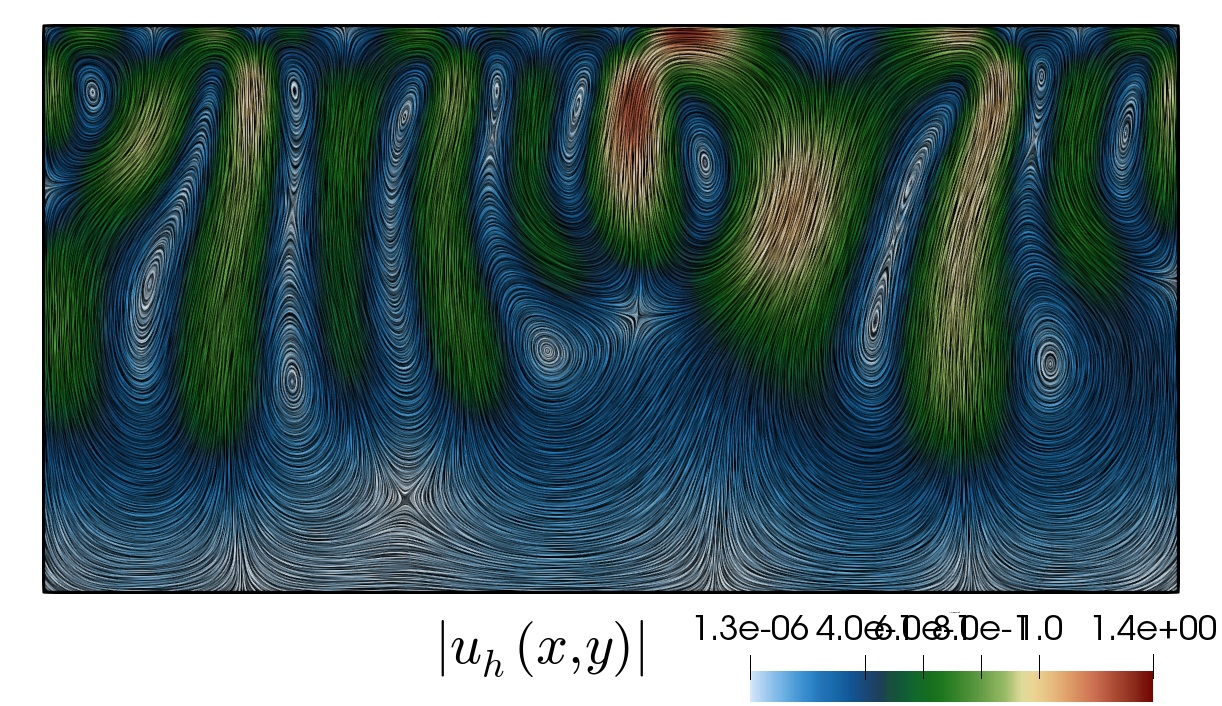}
     \includegraphics[width=0.24\textwidth]{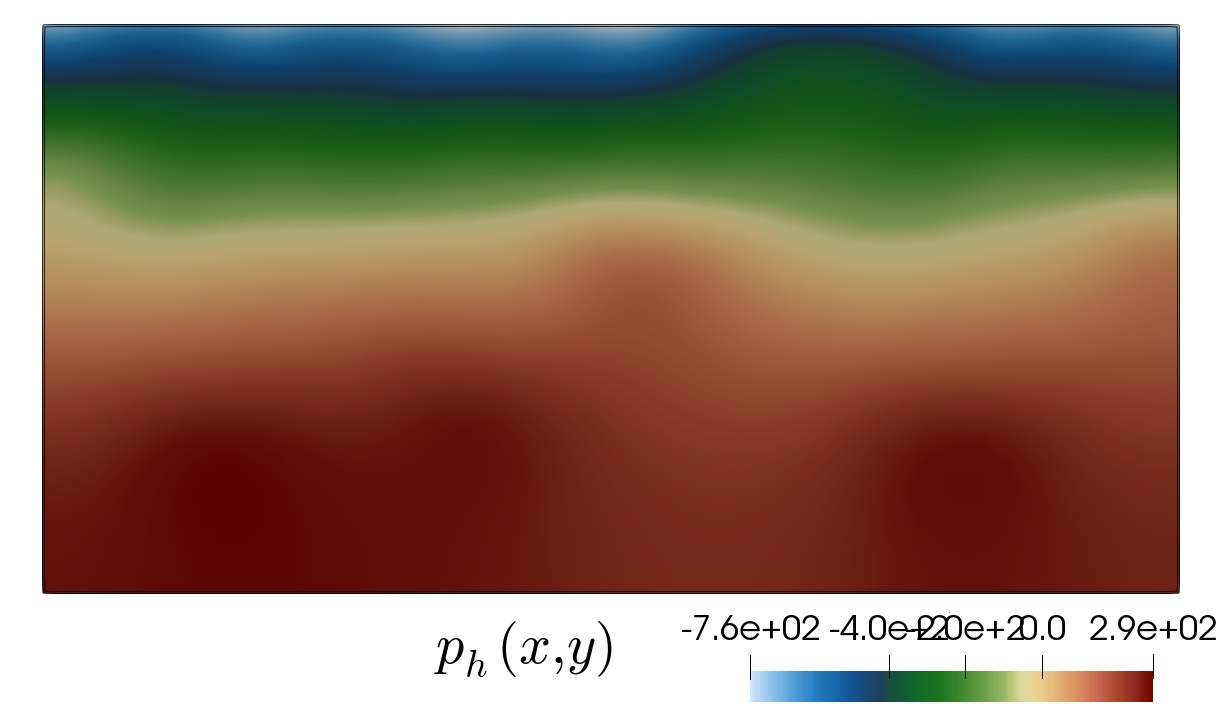}
      \includegraphics[width=0.24\textwidth]{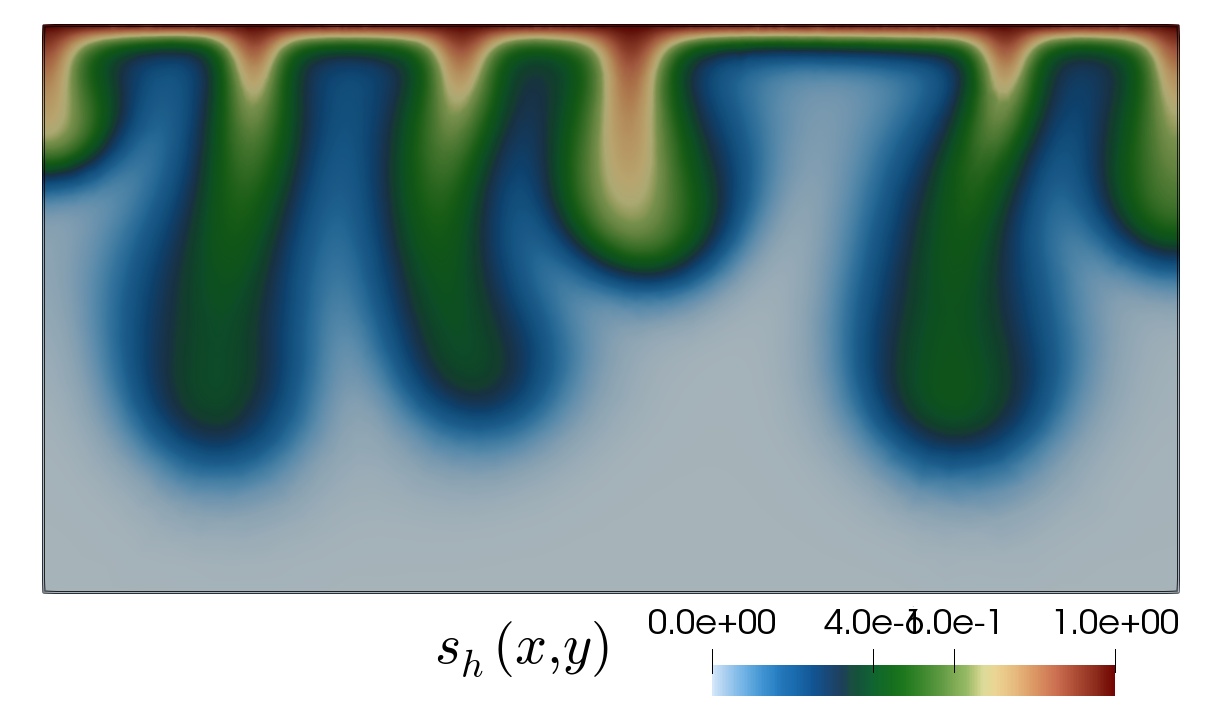}
       \includegraphics[width=0.24\textwidth]{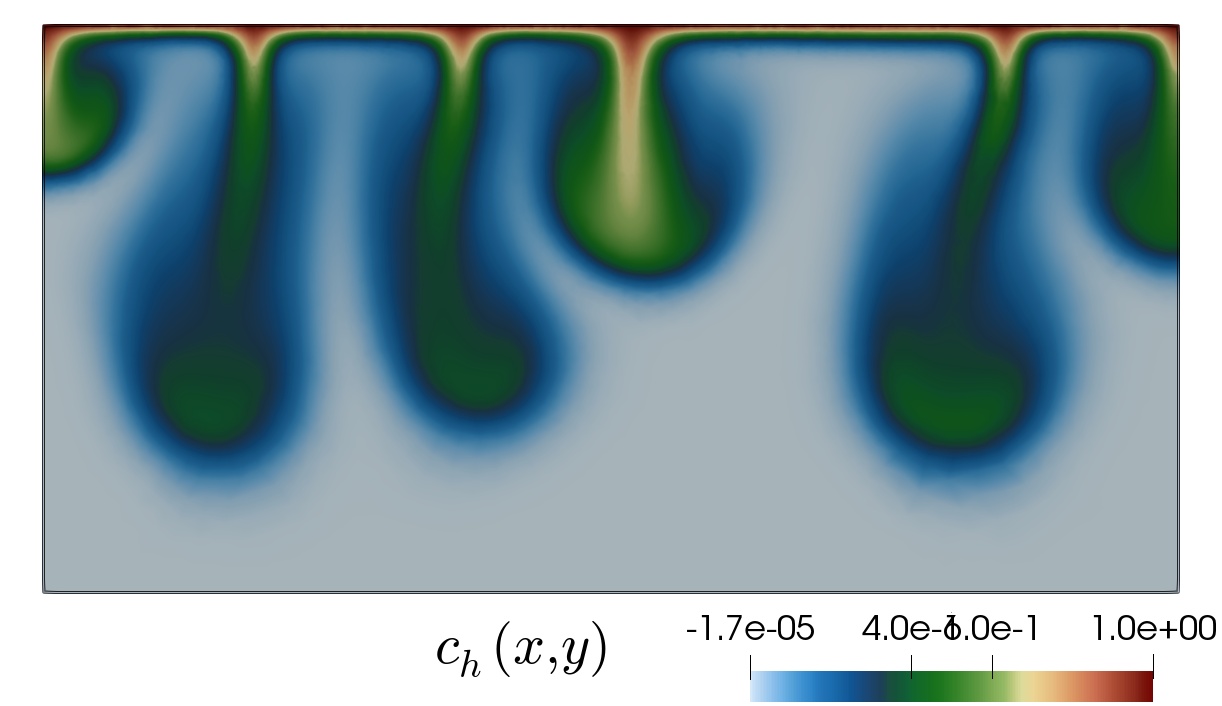}	
		    	\end{center}
\vspace{-3mm}
    	\caption{Example 5. Samples of adapted meshes at times $t=0,60,600,1200$ (top panels), 
	and approximate solutions shown at time $t=1500$ (bottom).}\label{fig:ex05}  
    \end{figure}


\end{document}